\documentclass{article}
\usepackage{amsfonts}
\usepackage{amsmath}
\usepackage{amssymb}
\usepackage{amsthm}
\usepackage{fullpage}
\usepackage{cite}

\newtheorem{theorem}{Theorem}

\newtheorem{lemma}[theorem]{Lemma}
\newtheorem{proposition}[theorem]{Proposition}

\DeclareMathOperator{\sgn}  {sgn} \DeclareMathOperator{\supp}{supp}
\DeclareMathOperator{\curl}{curl} \DeclareMathOperator{\diver}{div}
\numberwithin{equation}{section} \numberwithin{theorem}{section}
\allowdisplaybreaks[1]

\begin{document}

\title{Global regularity and convergence of a Birkhoff-Rott-$\alpha $
approximation of the dynamics of vortex sheets of the 2D Euler equations}
\author{Claude Bardos$^{1}$, Jasmine S.~Linshiz$^{2a}$ and Edriss S.~Titi$
^{2,3}$}
\date{} 
\maketitle

\begin{center}
$^{1}$\textit{Universit\'{e} Denis Diderot and Laboratory J.~L.~Lions\\[0pt]
Universit\'{e} Pierre et Marie Curie, Paris, France} \\[0pt]
bardos@ann.jussieu.fr \\[0pt]
$^{2}$\textit{Department of Computer Science and Applied Mathematics \\[0pt]
Weizmann Institute of Science \\[0pt]
Rehovot 76100, Israel}\\[0pt]
$^{a}$jasmine.tal@weizmann.ac.il \\[0pt]
$^{3}$\textit{Department of Mathematics and\\[0pt]
Department of Mechanical and Aerospace Engineering \\[0pt]
University of California \\[0pt]
Irvine, CA 92697-3875, USA} \\[0pt]
etiti@math.uci.edu \textit{and} edriss.titi@weizmann.ac.il
\end{center}



\begin{abstract}
We present an $\alpha $-regularization of the Birkhoff-Rott
equation, induced by the two-dimensional Euler-$\alpha$ equations,
for the vortex sheet dynamics. We show the convergence of the
solutions of Euler-$\alpha$ equations to a weak solution of the
Euler equations for initial vorticity being a finite Radon measure
of fixed sign, which includes the vortex sheets case. We also show
that, provided the initial density of vorticity is an integrable
function over the curve with respect to the arc-length measure, (i)
an initially Lipschitz chord arc vortex sheet (curve), evolving
under the BR-$\alpha $ equation, remains Lipschitz for all times,
(ii) an initially H\"{o}lder $C^{1,\beta }$, $0\leq \beta <1$, chord
arc curve
remains in $C^{1,\beta }$ for all times, and finally, (iii) an initially H%
\"{o}lder $C^{n,\beta },$ $n\geq 1,$ $0<\beta <1 $, closed chord arc
curve remains so for all times. In all these cases the weak
Euler-$\alpha $ and the BR-$\alpha $ descriptions of the vortex
sheet motion are equivalent.
\end{abstract}

\section{\label{sec:intro}Introduction}

The $\alpha $-regularization of the Navier-Stokes equations (NSE) is one of
the novel approaches for subgrid scale modeling of turbulence. The inviscid
Euler-$\alpha $ model was originally introduced in the Euler-Poincar\'{e}
variational framework in \cite{a_HMR98a,a_HMR98b}. In \cite%
{a_CFHOTW98,a_CFHOTW99, a_CFHOTW99_ChanPipe,a_FHT01,a_FHT02} the
corresponding Navier-Stokes-$\alpha $ (NS-$\alpha $) [also known as the
viscous Camassa-Holm equations or the Lagrangian-averaged Navier-Stokes-$%
\alpha $ (LANS-$\alpha $)] model, was obtained by introducing the
appropriate viscous term into the Euler-$\alpha $ equations. The extensive
research of the $\alpha $-models (see, e.g., \cite%
{a_GKT08,a_CHMZ99,a_HT05,a_CHOT05,a_ILT05,a_VTC05,a_CTV05,a_FHT02,a_FHT01,a_CFHOTW99,a_CFHOTW99_ChanPipe, a_CFHOTW98,a_MKSM03,a_LL06,a_LL03,b_BIL06,a_L06,a_GKT08,a_GH03,a_GH06,a_CLT06,a_BFR80,a_HN03,a_CHT04,a_CFR79,a_LT07}%
) stems, on the one hand, from the successful comparison of their steady
state solutions to empirical data, for a large range of huge Reynolds
numbers, for turbulent flows in infinite channels and pipes \cite%
{a_CFHOTW99,a_CFHOTW99_ChanPipe, a_CFHOTW98}. On the other hand, the $\alpha
$-models can also be viewed as numerical regularizations of the original,
Euler or Navier-Stokes, systems \cite{a_LT07,a_CLT06,b_BIL06,a_KT07}. The
main practical question arising is that of the applicability of these
regularizations to the correct predictions of the underlying flow phenomena.

In this paper we present some analytical results concerning the $\alpha $%
-regularization of the two-dimensional (2D) Euler equations in the context
of vortex sheet dynamics. The incompressible Euler equations are
\begin{equation}
\begin{split}
& \frac{\partial v}{\partial t}+(v\cdot \nabla )v+\nabla p=0, \\
& \nabla \cdot v=0, \\
& v(x,0)=v^{in}(x),
\end{split}
\label{grp:EulerEq}
\end{equation}
where $v$ the fluid velocity field and $p$, the pressure are the unknowns,
and $v^{in}$ is the given initial velocity. A vortex sheet is a surface of
codimension one (a curve in the plane) in inviscid incompressible flow,
across which the tangential component of the velocity has a jump
discontinuity, while the normal component is continuous. The flow outside
the sheet is irrotational. The evolution of the vortex sheet can be
described by the Birkhoff-Rott (BR) equation \cite{a_B62,a_R56,b_S92}. This
is a nonlinear singular integro-differential equation, which can be obtained
formally from the Euler equations assuming that the evolution of a vortex
sheet retains a curve-like structure:
\begin{equation*}
\frac{\partial \bar{z}}{\partial t}\left( \Gamma ,t\right) =\frac{1}{2\pi i}%
\ \mathrm{\ p.v.}\int_{-\infty }^{\infty }\frac{d\Gamma ^{\prime }}{z\left(
\Gamma ,t\right) -z\left( \Gamma ^{\prime },t\right) },
\end{equation*}%
here $z=x+iy$ is the complex position of the sheet and $\Gamma \in
(-\infty,\infty) $ represents the circulation, that is, $\gamma
=1/|z_{\Gamma }|$ is the vorticity density along the sheet. However,
the initial data problem for the BR equation is ill-posed due to the
Kelvin-Helmholtz instability \cite{a_B62, a_SB79}. Numerous results
show that an initially real analytic vortex sheet (curve) can
develop a finite time singularity in its curvature. This singularity
formation was studied with asymptotic
techniques in \cite{a_M79,a_CBT00} and numerically in \cite%
{a_MBO82,a_K86a,a_CBT00}. Specific examples of solutions were
constructed in \cite{a_DR88,a_CO89}, where the development, in a
finite time, of curvature singularity from initially analytic data
was rigorously proved. After the appearance of the first singularity
the solution becomes very irregular. This is a consequence of the
elliptic nature of the Birkhoff-Rott equations: if solutions have a
certain minimal regularity, then they are actually analytic
\cite{a_W02,a_W06,a_L02}. An open problem is the determination of
this threshold of regularity that will imply analyticity. It was
shown in \cite{a_L02} that any solution consisting of a closed chord
arc vortex sheet that near a point belongs to $C^{1,\beta }$, $\beta
>0$ must be analytic. The conclusion is maintained if the vortex
sheet is required to be a Lipschitz chord arc curve
\cite{a_W02,a_W06}.

The problem of the evolution of a vortex sheet can also be
approached, in the general framework of weak solutions (in the
distributional sense) of the Euler equations, as a problem of
evolution of the vorticity, which is concentrated as a measure along
a surface of codimension one. This approach was pioneered by DiPerna
and Majda in \cite{a_DM87,a_DM87b,a_DM88}. The general problem of
existence for mixed-sign vortex sheet initial data remains an open
question. However, in 1991, Delort \cite{a_D91} proved a global in
time existence of weak solutions of the 2D incompressible Euler
equations for the vortex sheet initial data with initial vorticity
being a Radon measure of a distinguished sign, see also
\cite{a_EM94,a_M93,a_LX95, a_S95, a_S96, b_MB02}. This result was
later obtained as an inviscid limit of the Navier-Stokes
regularizations of the Euler equations \cite{a_M93, a_S95}, and as a
limit of numerical vortex methods \cite{a_LX95, a_S96,a_LX00}. The
Delort's result \cite{a_D91} was also extended to the case of
mirror-symmetric flows with distinguished sign vorticity on each
side of the mirror \cite{a_LfNlX01}. It is worth mentioning that
uniqueness of solutions of the 2D Euler equations was obtained by
Yudovich \cite{a_Y63} for initially bounded vorticity, see, also,
\cite{a_V99} for an improvement
with vorticity in a class slightly larger than $L^{\infty }$, and \cite%
{a_T04} for review of relevant two-dimensional results. This does not
include vortex sheets, which admit measure-valued vorticity. There is also a
non-uniqueness result for velocity in $C\left( \left( 0,T\right) ,L_{\mathrm{%
weak}}^{2}\right) $ \cite{a_LS07, a_S93, a_S97}. However, the problem of
uniqueness of a weak solution with a fixed sign vortex sheet initial data is
still unanswered, numerical evidences of non-uniqueness can be found, e.g.,
in \cite{a_P89, a_FLLZ06}. Furthermore, the structure of weak solutions
given by Delort's theorem is not known, while the Birkhoff-Rott equations
assume \textit{a priori} that a vortex sheet remains a curve at a later
time. A proposed criterion for the equivalence of a weak solution of the 2D
Euler equations with vorticity being a Radon measure supported on a curve,
and a weak solution of the Birkhoff-Rott equation can be found in \cite%
{a_FLS06}. Also, another definition of weak solutions of Birkhoff-Rott
equation has been proposed in \cite{a_W02,a_W06}. For a recent survey of the
subject, see \cite{a_BT07}.

The Euler-$\alpha $ model \cite%
{a_CFHOTW99_ChanPipe,a_HMR98a,a_HMR98b,a_H02_pA,a_MS03,a_C01} is an inviscid
regularization of the Euler equations \eqref{grp:EulerEq} given by
\begin{equation}
\begin{split}
& \frac{\partial v}{\partial t}+\left( {u}\cdot \nabla \right) {v}+\sum_{j}{v%
}_{j}\nabla {u}_{j}+\nabla \pi =0, \\
& v=\left( 1-\alpha ^{2}\Delta \right) u, \\
& \nabla \cdot u=\nabla \cdot v=0, \\
& u(x,0)=u^{in}(x).
\end{split}
\label{grp:EulerAlphaEq}
\end{equation}
Here $u$ represents the \textquotedblleft filtered" fluid velocity vector, $%
\pi $ is the \textquotedblleft filtered" pressure, and $\alpha >0$ is a
regularization lengthscale parameter representing the width of the filter.

The question of global existence of weak solutions for the three-dimensional
Euler-$\alpha $ equations is still an open problem. On the other hand, the
2D Euler-$\alpha $ equations were studied in \cite{a_OS01}, where it has
been shown that there exists a unique global weak solution to the Euler-$%
\alpha $ equations with initial vorticity in the space of Radon measures on $%
{\mathbb{R}}^{2}$, with a unique Lagrangian flow map describing the
evolution of particles. In particular, it follows that the vorticity,
initially supported on a curve, remains supported on a curve for all times.

In this paper we relate the weak solutions of Euler-$\alpha $
equations with distinguished sign vortex sheet initial data to those
of the 2D Euler equations, by proving their convergence, as the
length scale $\alpha
\rightarrow 0$. This produces a variant of the result of Delort \cite{a_D91}%
, by obtaining a weak solution of Euler equations as a limit of an
inviscid regularization of Euler equations, in addition to
approximations obtained by smoothing the initial data, viscous
regularization, or numerical vortex methods \cite{a_D91,a_M93,
a_S95, b_MB02,a_LX95, a_S96,a_LX00}. Since a weak solution of Euler
equations with vortex sheet is unlikely to be unique, a different
regularization could produce a different weak solution.

We also present an analytical study of the $\alpha $-analogue of the
Birkhoff-Rott equation, the Birkhoff-Rott-$\alpha $ (BR-$\alpha $)
model, which is induced by the 2D Euler-$\alpha $ equations. The
BR-$\alpha $ results that were reported in a short communication
\cite{a_BLT08} are presented here with full details. The BR-$\alpha
$ model was implemented computationally in \cite{a_HNP06}, where a
numerical comparison between the BR-$\alpha $ regularization and the
existing regularizing methods, such as a vortex blob model
\cite{a_CB73, a_K86b, a_CK00,a_LX95,a_BP06}, has been performed. In
the BR-$\alpha $ case the singular kernel of the Biot-Savart law
determining the velocity in terms of the vorticity is smoothed by a
convolution with a smoothing function $G^{\alpha }\left( x\right) =\frac{1}{%
\alpha ^{2}}\frac{1}{2\pi }K_{0}\left( \frac{\left\vert x\right\vert }{%
\alpha }\right) $, which is the Green function associated with the Helmholtz
operator $\left( I-\alpha ^{2}\Delta \right) $. The function $K_{0}$ is a
modified Bessel function of the second kind of order zero. This is similar
to vortex blob methods, however, unlike the standard vortex blob methods
\cite{a_CB73,a_BM85, a_K86b,a_K87, a_CK00,a_BP06} (and, in particular, the
proof of convergence of vortex blobs methods to a weak solution of 2D Euler
equations \cite{a_LX95}), the BR-alpha smoothing function $G^{\alpha }$ is
unbounded at the origin. Also, unlike the vortex blob methods that
regularize the singular Biot-Savart kernel, the Euler-$\alpha $ model
regularizes the Euler equations themselves to obtain a smoother kernel.

Section~\ref{sec:Euler_alpha} contains the preliminaries about the
2D Euler-$ \alpha $ equations. In Section~\ref{sec:Convergence} we
investigate the convergence of solutions of the Euler-$\alpha $
equations for vortex sheet initial data to those of the 2D Euler
equations, as the regularization length scale $\alpha $ tends to
zero. Specifically, we prove that for the vortex sheet initial data
with initial vorticity of a distinguished sign Radon measure one can
extract subsequences of weak solutions of the Euler-$ \alpha $
equations which converge weak-$\ast $ in $L^{\infty }\left( \left[
0,T\right] ;\mathcal{M}({\mathbb{R}}^{2})\right) $, as $\alpha
\rightarrow 0$,
to a weak solution of the 2D Euler equations. The space $\mathcal{M}({%
\mathbb{R}}^{2})$ denotes the space of finite Radon measures on ${\mathbb{R}}^{2}$.

In Section~\ref{sec:BR_alpha} we describe the BR-$\alpha $ equation. Section~%
\ref{sec:LinStab} studies the linear stability of a flat vortex sheet with
uniform vorticity density for the 2D BR-$\alpha $ model. The linear
stability analysis shows that the BR-$\alpha $ regularization controls the
growth of high wave number perturbations, which is the reason for the
well-posedness. This is unlike the case for the original BR problem for
Euler equations that exhibits the Kelvin-Helmholtz instability, the main
mechanism for its ill-posedness. In Section~\ref{sec:GlobalReg} we show
global well-posedness of the 2D BR-$\alpha $ model in the space of Lipschitz
functions and in the H\"{o}lder space $C^{n,\beta }$, $n\geq 1$, which is
the space of $n$-times differentiable functions with H\"{o}lder continuous $%
n^{\text{th}}$ derivative. Specifically, we show that (i) an initially
Lipschitz chord arc vortex sheet (curve), evolving under the BR-$\alpha $
equation, remains Lipschitz for all times, (ii) an initially H\"{o}lder $%
C^{1,\beta }$, $0\leq \beta <1$, chord arc curve remains in $C^{1,\beta }$
for all times, and finally, (iii) an initially H\"{o}lder $C^{n,\beta },$ $%
n\geq 1,$ $0<\beta <1$, closed chord arc curve remains in $C^{n,\beta }$ for
all times. Notice that for $n>1$ we request $\beta $ to be strictly larger
than zero and the curve to be closed. In all these cases the weak Euler-$%
\alpha $ and the BR-$\alpha $ descriptions of the vortex sheet motion are
equivalent. The convergence of BR-$\alpha $ solutions to the solutions of
the original BR system on the short interval of existence of solutions will
be reported in a forthcoming paper.

\section{\label{sec:Euler_alpha}Euler-$\protect\alpha $ equations}

In two dimensions, the incompressible Euler equations in the vorticity form
are obtained by taking the curl of \eqref{grp:EulerEq} and are given by
\begin{equation}
\begin{split}
& \frac{\partial q}{\partial t}+\left( v\cdot \nabla \right) q=0, \\
& v=K\ast q, \\
& q(x,0)=q^{in}(x),
\end{split}
\label{grp:EulerEqVortForm}
\end{equation}%
where $K\left( x\right) =\frac{1}{2\pi }\nabla ^{\perp }\log \left\vert
x\right\vert $, $v$ is the fluid velocity field, \mbox{$q=\curl v$} is the
vorticity, and $q^{in}$ is the given initial vorticity. Delort \cite{a_D91}
proved a global in time existence of weak solutions of the 2D Euler
equations for the vortex sheet initial data with fixed sign initial
vorticity in $\mathcal{M}({\mathbb{R}}^{2})\cap H_{loc}^{-1}\left( \mathbb{R}%
^{2}\right) $. The space $\mathcal{M}({\mathbb{R}}^{2})$ is the space of
finite Radon measures on ${\mathbb{R}}^{2}$ with the norm
\begin{equation*}
\left\Vert \mu \right\Vert _{\mathcal{M}}=\sup \left\{ \,\left\vert \int_{%
\mathbb{R}^{2}}\varphi d\mu \right\vert :\varphi \in C_{0}\left( \mathbb{R}%
^{2}\right) ,\left\Vert \varphi \right\Vert _{L^{\infty }}\leq 1\right\} ,
\end{equation*}%
$\mathcal{C}_{0}({\mathbb{R}}^{2})$ is the space of continuous functions
vanishing at infinity. The space $H^{-s}$ denotes the dual of the Sobolev
space $H^{s}$. The localized Sobolev space $H_{loc}^{s}\left( \mathbb{R}%
^{2}\right) $, $s\in \mathbb{R}$ is the set of all distributions $f$ such
that $\rho f\in H^{s}(\mathbb{R}^{2})$ for any $\rho \in C_{c}^{\infty }(%
\mathbb{R}^{2})$, see, e.g., \cite{b_F99}.

A vorticity $q\in L^{\infty }\left( \left[ 0,T\right] ,\mathcal{M}({\mathbb{R%
}}^{2})\cap H_{loc}^{-1}\left( \mathbb{R}^{2}\right) \right) \cap Lip\left( %
\left[ 0,T\right] ,H_{loc}^{-L}\left( \mathbb{R}^{2}\right) \right) $, $L>1$%
, is called a weak solution of \eqref{grp:EulerEqVortForm}, if for every
test function $\psi \in $ $C_{c}^{\infty }\left( \mathbb{R}^{2}\times \left(
0,T\right) \right) $
\begin{equation}
W\left( q;\psi \right) \equiv \int_{0}^{T}\int_{\mathbb{R}^{2}}\partial
_{t}\psi \left( x,t\right) dq\left( x,t\right) dt+\int_{0}^{T}\int_{\mathbb{R%
}^{2}}\int_{\mathbb{R}^{2}}H_{\psi }\left( x,y,t\right) dq\left( y,t\right)
dq\left( x,t\right) dt=0,  \label{eq:EulerWeakVortForm}
\end{equation}%
where
\begin{equation}
H_{\psi }\left( x,y,t\right) =\frac{1}{4\pi }\frac{\left( x-y\right) ^{\perp
}\cdot \left( \nabla \psi \left( x,t\right) -\nabla \psi \left( y,t\right)
\right) }{\left\vert x-y\right\vert ^{2}}.  \label{eq:kernel_H}
\end{equation}%
The initial value is $q(x,0)=q^{in}(x)$ and it makes sense since $q\in
Lip\left( \left[ 0,T\right] ,H_{loc}^{-L}\left( \mathbb{R}^{2}\right)
\right) $. The kernel $H_{\psi }$ is bounded, continuous outside the
diagonal $x=y$ and vanishes at infinity. This weak vorticity formulation is
well-defined, since the $H^{-1}$ vorticity has no discrete part (i.e., $%
q\left( \left\{ x_{0}\right\} ,t\right) =0$ for all $x_{0}\in \mathbb{R}^{2}$%
), which implies that the diagonal $x=y$ has $q\left( x,t\right) q\left(
y,t\right) $-measure zero, see \cite{a_S95, a_D91}. Thorough discussions of
Delort's theorem, its extension and different proofs of the result can be
found in \cite{a_D91, b_MB02, b_C98, a_EM94, a_M93, a_LX95, a_S95, a_S96}.

Taking the curl of \eqref{grp:EulerAlphaEq} yields the vorticity formulation
of the 2D Euler-$\alpha $ model
\begin{equation}
\begin{split}
& \frac{\partial q}{\partial t}+\left( u\cdot \nabla \right) q=0, \\
& u=K^{\alpha }\ast q, \\
& q(x,0)=q^{in}(x).
\end{split}
\label{grp:Euler_alpha_vortForm}
\end{equation}%
Here $u$ represents the \textquotedblleft filtered" fluid velocity, and $%
\alpha >0$ is a regularization length scale parameter, which represents the
width of the filter. At the limit $\alpha =0$, we formally obtain the Euler
equations \eqref{grp:EulerEqVortForm}. The smoothed kernel is $K^{\alpha
}=G^{\alpha }\ast K$, where $G^{\alpha }$ is the Green function associated
with the Helmholtz operator $\left( I-\alpha ^{2}\Delta \right) $, given by
\begin{equation}
G^{\alpha }\left( x\right) =\frac{1}{\alpha ^{2}}G\left( \frac{x}{\alpha }%
\right) =\frac{1}{\alpha ^{2}}\frac{1}{2\pi }K_{0}\left( \frac{\left\vert
x\right\vert }{\alpha }\right) ,  \label{eq:GreenFunc_Helmholtz_2D}
\end{equation}%
here $x=\left( x_{1},x_{2}\right) \in \mathbb{R}^{2}$ and $K_{0}$ is a
modified Bessel function of the second kind of order zero \cite{b_W44}. To
see this relationship in $\mathbb{R}^{2}$ one can take a Fourier transform
of $v=\left( 1-\alpha ^{2}\Delta \right) u$, and obtain $G^{\alpha }$ as the
inverse Fourier transform of $\frac{1}{(1+\alpha ^{2}\left\vert k\right\vert
^{2})}$. Notice that
\begin{equation}
K^{\alpha }\left( x\right) =\nabla ^{\perp }\Psi ^{\alpha }\left( \left\vert
x\right\vert \right) =\frac{x^{\perp }}{\left\vert x\right\vert }D\Psi
^{\alpha }\left( \left\vert x\right\vert \right) ,  \label{eq:K_alpha}
\end{equation}%
where
\begin{align}
& \Psi ^{\alpha }\left( r\right) =\frac{1}{2\pi }\left[ K_{0}\left( \frac{r}{%
\alpha }\right) +\log r\right] ,  \label{eq:Psi_DPsi_def} \\
& D\Psi ^{\alpha }(r)=\frac{d\Psi ^{\alpha }}{dr}(r)=\frac{1}{2\pi }\left[ -%
\frac{1}{\alpha }K_{1}\left( \frac{r}{\alpha }\right) +\frac{1}{r}\right] ,
\notag
\end{align}%
and $K_{1}$ denotes a modified Bessel functions of the second kind of order
one. For details on Bessel functions, see, e.g., \cite{b_W44}.

A weak solution of \eqref{grp:Euler_alpha_vortForm} is $q\in C\left( \left[
0,T\right] ;\mathcal{M}({\mathbb{R}}^{2})\right) $ satisfying
\begin{equation}
W^{\alpha }\left( q;\psi \right) \equiv \int_{0}^{T}\int_{\mathbb{R}%
^{2}}\partial _{t}\psi \left( x,t\right) dq\left( x,t\right)
dt+\int_{0}^{T}\int_{\mathbb{R}^{2}}\int_{\mathbb{R}^{2}}H_{\psi }^{\alpha
}\left( x,y,t\right) dq\left( x,t\right) dq\left( y,t\right) dt=0,
\label{eq:EulerAlphaWeakVortForm}
\end{equation}%
for all test functions $\psi \in C_{c}^{\infty }\left( \mathbb{R}^{2}\times
\left( 0,T\right) \right) $. The initial value is $q(x,0)=q^{in}(x)$ and it
makes sense since $q\in C\left( \left[ 0,T\right] ;\mathcal{M}({\mathbb{R}}%
^{2})\right) $. The kernel $H_{\psi }^{\alpha }$ is a continuous vanishing
at infinity function given by
\begin{equation}
H_{\psi }^{\alpha }\left( x,y,t\right) =\frac{1}{2}D\Psi ^{\alpha }\left(
\left\vert x-y\right\vert \right) \frac{\left( x-y\right) ^{\perp }\cdot
\left( \nabla \psi \left( x,t\right) -\nabla \psi \left( y,t\right) \right)
}{\left\vert x-y\right\vert }.  \label{eq:kernel_H_alpha}
\end{equation}

Oliver and Shkoller \cite{a_OS01} showed global well-posedness of the Euler-$%
\alpha $ equations with initial vorticity in $\mathcal{M}({\mathbb{R}}^{2})$.

\begin{theorem}
\label{thm:OS01}\emph{(Oliver and Shkoller \cite{a_OS01})} For initial data $%
q^{in}\in \mathcal{M}({\mathbb{R}}^{2})$, there exists a unique global weak
solution of Euler-$\alpha $ equations \eqref{grp:Euler_alpha_vortForm} in
the sense of \eqref{eq:EulerAlphaWeakVortForm}. \newline
Let $\mathcal{G}$ denote the group of all homeomorphism of ${\mathbb{R}}^{2}$%
, which preserve the Lebesgue measure and let $\eta _{\alpha }=\eta _{\alpha
}(\cdot ,t)$ denote the Lagrangian flow map induced by %
\eqref{grp:Euler_alpha_vortForm}, i.e., which obeys the equation\newline
$\partial _{t}\eta _{\alpha }(x,t)=u\left( \eta _{\alpha }(x,t),t\right)
:=\int_{ \mathbb{R}^{2}}K^{\alpha }\left( \eta _{\alpha }(x,t),\eta _{\alpha
}(y,t)\right) dq^{in}\left( y,t\right) $, \thinspace $\eta _{\alpha }\left(
x,0\right) =x$. Then the unique Lagrangian flow map $\eta _{\alpha }\in
C^{1}\left( \left[ 0,T\right] ;\mathcal{G}\right) $ exists globally and the
vorticity $q_{\alpha }$ is transported by the flow, i.e., $q_{\alpha }\left(
x,t\right) =q^{in}\circ \eta _{\alpha }^{-1}\left( x,t\right) $. \newline
\end{theorem}

Notice that the original BR equations assume \textit{a priori} that a vortex
sheet remains a curve at a later time, however, in the 2D Euler-$\alpha $
case, it follows as a consequence of the existence of the unique Lagrangian
flow map, that the vorticity that is initially supported on a curve remains
supported on a curve for all times.

\section{\label{sec:Convergence} Convergence of a fixed sign Euler-$\protect%
\alpha $ vortex sheet to an Euler vortex sheet}

Let the initial vorticity $q^{in}\in \mathcal{M}({\mathbb{R}}^{2})\cap
H_{loc}^{-1}\left( \mathbb{R}^{2}\right) $ be of a fixed sign, $q^{in}\geq 0$%
, and compactly supported. In this section we show that there is a
subsequence of the solutions of 2D Euler-$\alpha $ model with initial data $%
q^{in}$, guaranteed by Theorem \ref{thm:OS01}, that converge to a
weak solution of 2D Euler equations in the sense of
\eqref{eq:EulerWeakVortForm}. This produces a variant of the result
of Delort \cite{a_D91}, by obtaining a weak solution of Euler
equations as a limit of solutions of inviscid regularization of
Euler equations, namely, the Euler-$\alpha $ equations. The above
regularization method is different from the various existing
regularizations that are obtained, for instance, by smoothing the
initial data, viscous regularization or numerical vortex methods
\cite{a_D91,a_M93, a_S95, b_MB02,a_LX95, a_S96,a_LX00}. Since a weak
solution of Euler equations with vortex sheet is unlikely to be
unique, a different regularization could produce a different weak
solution of Euler equations.

In order to prove the convergence of the solutions $q_{\alpha }$ of
the Euler-$\alpha $ equations \eqref{grp:Euler_alpha_vortForm} to a
weak solution of Euler equations \eqref{grp:EulerEqVortForm} we
follow ideas similar to those reported in \cite{a_D91, b_MB02,
a_M93, a_S95}. However, due to the structure of the Euler-$\alpha $
equations one needs to deal with various technical estimates
concerning the \textquotedblleft filtered" vorticity $\omega
_{\alpha }=\left( 1-\alpha ^{2}\Delta \right) ^{-1}q_{\alpha }$ and
$\alpha ^{2}\Delta \omega _{\alpha }=q_{\alpha }-\omega _{\alpha }$.
Specifically, we show in Lemma \ref{lemma:uniformDecay} and Lemma
\ref{lemma:laplacianOmegaNorm}, respectively, that $\omega _{\alpha
}$ have a uniform decay in small disks, $\sup_{\alpha >0,0\leq t\leq
T,0<R<1,x_{0}\in {\ \mathbb{R}}^{2}}\int_{\left\vert
x-x_{0}\right\vert <R}d\omega _{\alpha }\left( x,t\right) \leq
C\left(
T\right) \left\vert \log R\right\vert ^{-1/2} $, and the contribution of $%
\int_{\mathbb{R}^{2}}d\left\vert \alpha ^{2}\Delta \omega _{\alpha
}\right\vert $ converges to zero, as $\alpha \rightarrow 0$.

\begin{theorem}
\label{thm:Convergence}Let $q_{\alpha }$ be the solutions of the weak
vorticity formulation of Euler-$\alpha $ equations %
\eqref{eq:EulerAlphaWeakVortForm}, guaranteed by Theorem \ref{thm:OS01},
with initial data $q^{in}\in \mathcal{M}({\mathbb{R}}^{2})\cap
H_{loc}^{-1}\left( \mathbb{R}^{2}\right) $, $q^{in}\geq 0$ and compactly
supported and let $T>0$. Then there exists a subsequence $q_{\alpha _{j}}$
that weak-$\ast $ converges to $q$ in $L^{\infty }\left( \left[ 0,T\right] ;
\mathcal{M}({\mathbb{R}}^{2})\right) $ and in $\mathcal{M}\left( \mathbb{R}
^{2}\right) $ for each fixed $t$, as $\alpha _{j}\rightarrow 0$, and $q$ is
a weak solution of the Euler equations \eqref{grp:EulerEqVortForm} in the
sense of \eqref{eq:EulerWeakVortForm} with initial data $q^{in}$.
\end{theorem}

The weak-$\ast $ convergence in $L^{\infty }\left( \left[ 0,T\right] ;%
\mathcal{M}({\mathbb{R}}^{2})\right) $ means that
\begin{equation*}
\lim_{\alpha _{j}\rightarrow \infty }\int_{0}^{T}\int_{\mathbb{R}%
^{2}}\varphi \left( x,t\right) dq_{\alpha _{j}}\left( x,t\right)
dt=\int_{0}^{T}\int_{\mathbb{R}^{2}}\varphi \left( x,t\right) dq\left(
x,t\right) dt,
\end{equation*}%
for all $\varphi \in L^{1}\left( \left[ 0,T\right] ;\mathcal{C}_{0}({\mathbb{%
R}}^{2})\right) $.

We denote the velocity and the \textquotedblleft filtered" velocity
by $v_{\alpha }$ and $u_{\alpha }$, respectively, and their
corresponding
vorticities by $q_{\alpha }=\curl v_{\alpha }$ and $\omega _{\alpha }=\curl %
u_{\alpha }$.

Given $q_{\alpha }\in \mathcal{M}({\mathbb{R}}^{2})$, we define a linear
continuous functional $\omega _{\alpha }=\left( 1-\alpha ^{2}\Delta \right)
^{-1}q_{\alpha }$ acting on every $\varphi \in C_{0}\left( \mathbb{R}%
^{2}\right) $ by%
\begin{equation}
\left\langle \omega _{\alpha },\varphi \right\rangle =\int_{\mathbb{R}%
^{2}}\left( \left( 1-\alpha ^{2}\Delta \right) ^{-1}\varphi \right)
dq_{\alpha },  \label{eq:omega_from_q}
\end{equation}%
where $\psi =\left( 1-\alpha ^{2}\Delta \right) ^{-1}\varphi $ is defined as
the unique, vanishing at infinity, solution of $\varphi =\left( 1-\alpha
^{2}\Delta \right) \psi $, given by
\begin{equation}
\left( 1-\alpha ^{2}\Delta \right) ^{-1}\varphi =\int_{\mathbb{R}^{2}}\frac{1%
}{\alpha ^{2}}\frac{1}{2\pi }K_{0}\left( \frac{\left\vert y\right\vert }{%
\alpha }\right) \varphi \left( x-y\right) dy,  \label{eq:u_from_v}
\end{equation}%
the function $K_{0}$ is a modified Bessel function of the second kind of
order zero, $K_{0}>0$, $\int_{0}^{\infty }K_{0}\left( r\right) rdr=1$, see,
e.g., \cite{b_W44}. From the above its follows that $\left\Vert \left(
1-\alpha ^{2}\Delta \right) ^{-1}\varphi \right\Vert _{L^{\infty }}\leq
\left\Vert \varphi \right\Vert _{L^{\infty }}$.

We observe that if $q_{\alpha }\geq 0$ then $\omega _{\alpha }$ is a
nonnegative linear functional. Indeed, let $\varphi \in C_{0}\left( \mathbb{R%
}^{2}\right) $, $\varphi \geq 0$, then
\begin{equation*}
\left( 1-\alpha ^{2}\Delta \right) ^{-1}\varphi =\int_{\mathbb{R} ^{2}}\frac{%
1}{\alpha ^{2}}\frac{1}{2\pi }K_{0}\left( \frac{\left\vert y\right\vert }{%
\alpha }\right) \varphi \left( x-y\right) dy\geq 0,
\end{equation*}%
and hence by \eqref{eq:omega_from_q} $\left\langle \omega _{\alpha },\varphi
\right\rangle \geq 0$. Also,
\begin{equation*}
\left\vert \,\left\langle \omega _{\alpha },\varphi \right\rangle
\right\vert \leq \left\Vert q_{\alpha }\right\Vert _{\mathcal{M}}\left\Vert
\left( 1-\alpha ^{2}\Delta \right) ^{-1}\varphi \right\Vert _{L^{\infty
}}\leq \left\Vert q_{\alpha }\right\Vert _{\mathcal{M}}\left\Vert \varphi
\right\Vert _{L^{\infty }}.
\end{equation*}%
Therefore, by the Riesz representation theorem (see, e.g., \cite[Chapter 7]%
{b_F99} ) the functional $\omega _{\alpha }$ can be represented by a unique
nonnegative Radon measure, which we also denote by $\omega _{\alpha }$, and
\begin{equation}
\left\Vert \omega _{\alpha }\right\Vert _{\mathcal{M}}\leq \left\Vert
q_{\alpha }\right\Vert _{\mathcal{M}}.  \label{eq:omega_M_norm_bound}
\end{equation}%
Again, by the Riesz representation theorem, a linear functional $\left(
\alpha ^{2}\Delta \omega _{\alpha }\right) $ defined by
\begin{equation}
\left\langle \alpha ^{2}\Delta \omega _{\alpha },\varphi \right\rangle
=\int_{\mathbb{R}^{2}}\left( \alpha ^{2}\Delta \left( 1-\alpha ^{2}\Delta
\right) ^{-1}\varphi \right) dq_{\alpha },  \label{eq:laplacian_omega_from_q}
\end{equation}%
for every $\varphi \in C_{0}\left( \mathbb{R}^{2}\right) $, can be
identified with a Radon measure, which we also denote by $\alpha ^{2}\Delta
\omega _{\alpha }$. Observe that, since for every $\varphi \in C_{0}\left(
\mathbb{R}^{2}\right) $
\begin{equation*}
\alpha ^{2}\Delta \left( 1-\alpha ^{2}\Delta \right) ^{-1}\varphi =\left(
1-\alpha ^{2}\Delta \right) ^{-1}\varphi -\varphi ,
\end{equation*}%
we have
\begin{equation*}
\left\vert \,\left\langle \alpha ^{2}\Delta \omega _{\alpha },\varphi
\right\rangle \right\vert \leq \left\Vert q_{\alpha }\right\Vert _{\mathcal{M%
}}\left\Vert \alpha ^{2}\Delta \left( 1-\alpha ^{2}\Delta \right)
^{-1}\varphi \right\Vert _{L^{\infty }}\leq 2\left\Vert q_{\alpha
}\right\Vert _{\mathcal{M}}\left\Vert \varphi \right\Vert _{L^{\infty }},
\end{equation*}%
that is, $\left\Vert \alpha ^{2}\Delta \omega _{\alpha }\right\Vert _{%
\mathcal{M}\left( \mathbb{R}^{2}\right) }\leq 2\left\Vert q_{\alpha
}\right\Vert _{\mathcal{M}\left( \mathbb{R}^{2}\right) }.$

We note that by Theorem \ref{thm:OS01} the solution $q_{\alpha }$ of Euler-$%
\alpha $ equations \eqref{eq:EulerAlphaWeakVortForm} is transported
by the flow, that is, $q_{\alpha }\left( x,t\right) =q^{in}\circ
\eta _{\alpha }^{-1}\left( x,t\right) $, $\eta _{\alpha }\in
C^{1}\left( \left[ 0,T\right] ;\mathcal{G}\right) $, hence for all
$t$
\begin{equation}
\left\Vert q_{\alpha }\left( \cdot ,t\right) \right\Vert _{\mathcal{M}
}=\left\Vert q^{in}\right\Vert _{\mathcal{M}}.  \label{eq:Mspace_vort_bound}
\end{equation}%
In addition, if $q^{in}\geq 0$, then $q_{\alpha }\geq 0$ for all times, and
therefore also $\omega _{\alpha }\geq 0$ for all times.

The kernel $H_{\psi }$ appearing in the non-linear term of %
\eqref{eq:EulerWeakVortForm} is discontinuous on the diagonal $x=y$, so,
following \cite{a_DM87,b_MB02}, to prove the convergence of the non-linear
term we need the following estimate, which shows uniform decay of the
\textquotedblleft filtered" vorticity $\omega _{\alpha }$ in small disks.

\begin{lemma}
\label{lemma:uniformDecay}Let $q_{\alpha }$ be the solutions of %
\eqref{eq:EulerAlphaWeakVortForm} with initial data $q^{in}\in \mathcal{M}({%
\mathbb{R}}^{2})\cap H_{loc}^{-1}\left( \mathbb{R}^{2}\right) $, $q^{in}\geq
0$ and compactly supported. Then for $\omega _{\alpha }=\left( 1-\alpha
^{2}\Delta \right) ^{-1}q_{\alpha }$ defined by \eqref{eq:omega_from_q},
there exists a constant $C=C\left( T\right) $, such that for all $\alpha >0$%
, $0\leq t\leq T$, $0<R<1$ and $x_{0}\in {\mathbb{R}}^{2}$ we have
\begin{equation}
\int_{\left\vert x-x_{0}\right\vert <R}d\omega _{\alpha }\left( x,t\right)
\leq C\left( T\right) \left\vert \log R\right\vert ^{-1/2}.
\label{eq:uniform_decay}
\end{equation}
\end{lemma}

\begin{proof}
Recall that $\omega _{\alpha }\geq 0$ for all times. The idea of the proof,
which is shown in details below, is to convolve the initial data with a
standard $C_{c}^{\infty }\left( \mathbb{R}^{2}\right) $ mollifier to obtain
a sequence of solutions of the Euler-$\alpha $ equations that has a uniform
decay of the circulation on small disks%
\begin{equation*}
\int_{\left\vert x-x_{0}\right\vert \leq R}\omega _{\alpha ,\varepsilon
}\left( x,t\right) dx\leq C\left( T\right) \left\vert \log R\right\vert
^{-1/2},
\end{equation*}%
$0<\varepsilon \leq \varepsilon _{0}$, $0\leq t\leq T$, $R<1$, and then the
weak-$\ast $ limit in $L^{\infty }\left( \left[ 0,T\right] ,\mathcal{M}%
\left( \mathbb{R}^{2}\right) \right) $ of a subsequence $\omega _{\alpha
,\varepsilon _{j}}$ when $\varepsilon _{j}\rightarrow 0$, which is the
solution of Euler-$\alpha $ equations with initial data $q^{in}$, satisfies
a similar bound.

We observe that, similarly to the Euler equations, any smooth radially
symmetric vanishing at infinity vorticity $\bar{q}\left( \left\vert
x\right\vert \right) $ defines a stationary solution of Euler-$\alpha $
equations \eqref{grp:Euler_alpha_vortForm} with the corresponding velocity $%
\bar{v}\left( x\right) =\nabla ^{\perp }\Delta ^{-1}\bar{q}\left( \left\vert
x\right\vert \right) =\frac{x^{\perp }}{\left\vert x\right\vert ^{2}}
\int_{0}^{\left\vert x\right\vert }s\bar{q}\left( s\right) ds$. This could
be seen using the vorticity stream function formulation for Euler-$\alpha $
equations, which is%
\begin{align*}
& q_{t}+J\left( \varphi ,\Delta \psi \right) =0, \\
& q=\Delta \psi ,
\end{align*}%
where $J\left( \varphi ,\chi \right) =$ $\frac{\partial \varphi }{\partial
x_{1}}\frac{\partial \chi }{\partial x_{2}}-\frac{\partial \varphi }{%
\partial x_{2}}\frac{\partial \chi }{\partial x_{1}}$ is the Jacobian, $\psi
$ is the velocity stream function, $v=\nabla ^{\perp }\psi $, and $\varphi
=\left( 1-\alpha ^{2}\Delta \right) ^{-1}\psi $ is the \textquotedblleft
filtered" stream function, $u=\nabla ^{\perp }\varphi $. Since $\Delta $ and
$\left( 1-\alpha ^{2}\Delta \right) $ are rotationally invariant, we have
that the corresponding $\bar{\omega}=\left( 1-\alpha ^{2}\Delta \right) ^{-1}%
\bar{q}$, $\bar{\psi}=\Delta ^{-1}\bar{q}$ and $\bar{\varphi}=\left(
1-\alpha ^{2}\Delta \right) ^{-1}\bar{\psi}$ are also radially symmetric,
therefore $J\left( \bar{\varphi},\Delta \bar{\psi}\right) =0$ and hence $%
\bar{q}$ defines a stationary solution of Euler-$\alpha $ equations.

Let $\rho \in C_{c}^{\infty }\left( \mathbb{R}^{2}\right) $ be a standard
mollifier, for example,
\begin{equation*}
\rho \left( x\right) =\left\{
\begin{array}{cc}
C\exp \left( 1/\left( \left\vert x\right\vert ^{2}-1\right) \right) & \text{%
if }\left\vert x\right\vert <1, \\
0 & \text{if }\left\vert x\right\vert \geq 1,%
\end{array}%
\right.
\end{equation*}%
$\int_{\mathbb{R}^{2}}\rho =1$, $\rho _{\varepsilon }\left( x\right) =\frac{1%
}{\varepsilon ^{2}}\rho \left( \frac{x}{\varepsilon }\right) $. Smoothing
the initial data by a mollification with $\rho _{\varepsilon }$, $%
q_{\varepsilon }^{in}=\rho _{\varepsilon }\ast q^{in}$, we have that for all
$0<\varepsilon <\varepsilon _{0}$ the smoothed initial vorticities satisfy $%
q_{\varepsilon }^{in}\geq 0$, $\supp q_{\varepsilon }^{in}\subseteq \left\{
x|\left\vert x\right\vert <R_{0}\right\} $ (since $q^{in}$ is compactly
supported), $\int_{ \mathbb{R} ^{2}}q_{\varepsilon }^{in}\left( x\right)
dx=\int_{ \mathbb{R} ^{2}}dq^{in}\left( x\right) $. Following \cite%
{a_DM87,b_MB02} for the 2D Euler case we decompose the velocity into a
combination of a stationary bounded velocity plus a time dependent velocity
with finite total energy. Let $\bar{q}\left( \left\vert x\right\vert \right)
$ be any smooth radially symmetric function with compact support, such that $%
\int_{\mathbb{R} ^{2}}\bar{q}\left( \left\vert x\right\vert \right) dx=\int_{%
\mathbb{R} ^{2}}dq^{in}\left( x\right) $. Define $\bar{v}=K\ast \bar{q}$, $%
\tilde{q}_{\varepsilon }^{in}=q_{\varepsilon }^{in}-\bar{q}$ and $\tilde{v}%
_{\varepsilon }^{in}=$ $K\ast \tilde{q}_{\varepsilon }^{in}$. Notice, that
by direct calculation $\diver\bar{v}=0$ and $\bar{v},\nabla \bar{v},\partial
^{2}\bar{v}\in L^{\infty }\left( \mathbb{R}^{2}\right) $. Since $\int_{%
\mathbb{R}^{2}}\tilde{q}_{\varepsilon }^{in}=0$, and
$\tilde{q}_{\varepsilon }^{in}$ has compact support we have that
$\tilde{v}_{\varepsilon }^{in}\in L^{2}\left( \mathbb{R}^{2}\right)
$. Also, due to the fact that $q^{in}\in
\mathcal{M}({\mathbb{R}}^{2})\cap H_{loc}^{-1}\left(
\mathbb{R}^{2}\right) $ with compact support, and hence, for
$\varepsilon \leq \varepsilon _{0}$, the smooth $q_{\varepsilon
}^{in}$ are uniformly bounded in $L^{1}$ with a common compact
support and $v_{\varepsilon }^{in}=K\ast q_{\varepsilon }^{in} $ are
uniformly bounded in $L_{loc}^{2}$, and since $\bar{q}$ is
independent of $\varepsilon$, we have that $\tilde{v}_{\varepsilon
}^{in}$ are uniformly bounded in $L^{2}\left( \mathbb{R}^{2}\right) $, for $%
\varepsilon \leq \varepsilon _{0}$.

Observe, that the stationary part
\begin{equation*}
\bar{u}\left( x\right) =\left( 1-\alpha ^{2}\Delta \right) ^{-1}\bar{v}%
\left( x\right) =\int_{\mathbb{R}^{2}}\frac{1}{\alpha ^{2}}\frac{1}{2\pi }%
K_{0}\left( \frac{\left\vert y\right\vert }{\alpha }\right) \bar{v}\left(
x-y\right) dy,
\end{equation*}
satisfies
\begin{align}
& \left\Vert \bar{u}\right\Vert _{L^{\infty }}\leq \left\Vert \bar{v}%
\right\Vert _{L^{\infty }},  \label{eq:u_bar_bounds} \\
& \left\Vert \nabla \bar{u}\right\Vert _{L^{\infty }}\leq \left\Vert \nabla
\bar{v}\right\Vert _{L^{\infty }},  \notag \\
& \left\Vert \partial ^{2}\bar{u}\right\Vert _{L^{\infty }}\leq \frac{1}{%
2\pi }\frac{1}{\alpha }\left\Vert \nabla \bar{v}\right\Vert _{L^{\infty }},
\notag
\end{align}%
since $K_{0}$ and its derivative are smooth functions outside of the origin,
satisfying $\left\vert K_{0}\left( r\right) \right\vert $ $\leq C\log r$, $%
\left\vert DK_{0}\left( r\right) \right\vert \leq Cr^{-1}$ and rapidly
decaying at infinity.

Consider the partial differential equation
\begin{align}
\frac{\partial }{\partial t}\tilde{v}_{\alpha ,\varepsilon }& +\left( \tilde{%
u}_{\alpha ,\varepsilon }\cdot \nabla \right) \tilde{v}_{\alpha ,\varepsilon
}+\sum_{j}\left( \tilde{v}_{\alpha ,\varepsilon }\right) _{j}\nabla \left(
\tilde{u}_{\alpha ,\varepsilon }\right) _{j}  \label{eq:v_tilda} \\
& +\left( \tilde{u}_{\alpha ,\varepsilon }\cdot \nabla \right) \bar{v}%
+\sum_{j}\bar{v}_{j}\nabla \left( \tilde{u}_{\alpha ,\varepsilon }\right)
_{j}  \notag \\
& +\left( \bar{u}\cdot \nabla \right) \tilde{v}_{\alpha ,\varepsilon
}+\sum_{j}\left( \tilde{v}_{\alpha ,\varepsilon }\right) _{j}\nabla \bar{u}%
_{j}+\nabla \tilde{\pi}_{\alpha ,\varepsilon }=0,  \notag \\
\tilde{v}_{\alpha ,\varepsilon }=& \left( 1-\alpha ^{2}\Delta \right) \tilde{%
u}_{\alpha ,\varepsilon }.  \notag
\end{align}%
This evolution equation is similar to the Euler-$\alpha $ equations.
Moreover, if $\tilde{v}_{\alpha ,\varepsilon }\left( x,t\right) $ is the
solution of the equation \eqref{eq:v_tilda} with initial data $\tilde{v}%
_{\varepsilon }^{in}$, then $v_{\alpha ,\varepsilon }\left( x,t\right) =%
\tilde{v}_{\alpha ,\varepsilon }\left( x,t\right) +\bar{v}\left( x\right) $
is the solution of the 2D Euler-$\alpha $ equations \eqref{grp:EulerAlphaEq}
with initial data $v_{\varepsilon }^{in}=$ $K\ast q_{\varepsilon }^{in}.$

Similarly to the Euler case (see, e.g., \cite{b_MB02}) this equation has a
unique global infinitely smooth solution, since, as in 2D Euler case, we
have an \textit{a priori} uniform control over the $L^{\infty }$ norm of the
$\tilde{q}_{\alpha ,\varepsilon }$, which implies the global existence, as
in the proof of the Beale-Kato-Majda criterion \cite{a_BKM84}. The solution $%
\tilde{v}_{\alpha ,\varepsilon }$ is in $C^{1}\left( \left[ 0,\infty \right)
,H^{s}\left( \mathbb{R} ^{2}\right) \right) $ for all $s>2$, and hence, by
Sobolev embedding theorem, $\partial ^{k}\tilde{v}_{\alpha ,\varepsilon }$
and, consequently, $\partial ^{k}\tilde{u}_{\alpha ,\varepsilon }\left(
x\right) =\int_{\mathbb{R}^{2}}\frac{1}{\alpha ^{2}}\frac{1}{2\pi }%
K_{0}\left( \frac{\left\vert y\right\vert }{\alpha }\right) \tilde{v}%
_{\alpha ,\varepsilon }\left( x-y\right) dy$ are also in $C_{0}\left(
\mathbb{R}^{2}\right) $ for all $k$.

Moreover, the solution $\tilde{u}_{\alpha ,\varepsilon }$ is in $L^{\infty
}\left( \left[ 0,\infty \right) ;H^{1}\left( \mathbb{R}^{2}\right) \right) $
due to the following \textit{a priori }estimate. Taking the inner product of %
\eqref{eq:v_tilda} with $\tilde{u}_{\alpha ,\varepsilon }$ we have (omitting
the subindices $\alpha $ and $\varepsilon $)%
\begin{align*}
\frac{1}{2}\frac{d}{dt}\left( \left\vert \tilde{u}\right\vert
_{L^{2}}^{2}+\alpha ^{2}\left\vert \nabla \tilde{u}\right\vert
_{L^{2}}^{2}\right) & =\alpha ^{2}\left( \left( \bar{u}\cdot \nabla \right)
\Delta \tilde{u},\tilde{u}\right) -\sum_{j}\left( \tilde{v}_{j}\nabla \bar{u}%
_{j},\tilde{u}\right) \\
& =I_{1}-I_{2}.
\end{align*}%
Since $\diver\bar{u}=0$, for $I_{1}$ we have%
\begin{align*}
I_{1}& =-\alpha ^{2}\sum_{i,j,k}\int \bar{u}_{i}\frac{\partial ^{2}\tilde{u}%
_{j}}{\partial x_{k}^{2}}\frac{\partial }{\partial x_{i}}\tilde{u}_{j}\qquad
\\
& =\alpha ^{2}\sum_{i,j,k}\int \frac{\partial \bar{u}_{i}}{\partial x_{k}}%
\frac{\partial \tilde{u}_{j}}{\partial x_{k}}\frac{\partial \tilde{u}_{j}}{%
\partial x_{i}}+\alpha ^{2}\sum_{i,j,k}\int \bar{u}_{i}\left( \frac{\partial
^{2}}{\partial x_{i}\partial x_{k}}\tilde{u}_{j}\right) \frac{\partial }{%
\partial x_{k}}\tilde{u}_{j}.
\end{align*}%
Since the second term on the right is zero, we obtain that%
\begin{equation*}
\left\vert I_{1}\right\vert \leq C\alpha ^{2}\left\Vert \nabla \bar{u}%
\right\Vert _{L^{\infty }}\left\Vert \nabla \tilde{u}\right\Vert
_{L^{2}}^{2}.
\end{equation*}%
Now we estimate $I_{2}$
\begin{align*}
I_{2}& =\sum_{i,j,k}\int \tilde{u}_{j}\nabla \bar{u}_{j}\cdot \tilde{u}%
-\alpha ^{2}\int \Delta \tilde{u}_{j}\nabla \bar{u}_{j}\cdot \tilde{u} \\
& =I_{2_{1}}+I_{2_{2}}.
\end{align*}%
We have%
\begin{equation*}
\left\vert I_{2_{1}}\right\vert \leq C\left\Vert \tilde{u}\right\Vert
_{L^{2}}^{2}\left\Vert \nabla \bar{u}\right\Vert _{L^{\infty }}
\end{equation*}%
and%
\begin{equation*}
I_{2_{2}}=\alpha ^{2}\sum_{i,j,k}\int \frac{\partial \tilde{u}_{j}}{\partial
x_{k}}\frac{\partial \bar{u}_{j}}{\partial x_{i}}\frac{\partial \tilde{u}_{i}%
}{\partial x_{k}}+\alpha ^{2}\sum_{i,j,k}\int \frac{\partial \tilde{u}_{j}}{%
\partial x_{k}}\frac{\partial ^{2}\bar{u}_{j}}{\partial x_{k}\partial x_{i}}%
\tilde{u}_{i},
\end{equation*}%
hence%
\begin{equation*}
\left\vert I_{2_{2}}\right\vert \leq C\alpha ^{2}\left\Vert \nabla \tilde{u}%
\right\Vert _{L^{2}}^{2}\left\Vert \nabla \bar{u}\right\Vert _{L^{\infty
}}+C\alpha ^{2}\left\Vert \nabla \tilde{u}\right\Vert _{L^{2}}\left\Vert
\tilde{u}\right\Vert _{L^{2}}\left\Vert \partial ^{2}\bar{u}\right\Vert
_{L^{\infty }}.
\end{equation*}%
To conclude, we obtain%
\begin{equation*}
\frac{1}{2}\frac{d}{dt}\left( \left\Vert \tilde{u}\right\Vert
_{L^{2}}^{2}+\alpha ^{2}\left\Vert \nabla \tilde{u}\right\Vert
_{L^{2}}^{2}\right) \leq C\left( \alpha ^{2}\left\Vert \nabla \bar{u}%
\right\Vert _{L^{\infty }}\left\Vert \nabla \tilde{u}\right\Vert
_{L^{2}}^{2}+\left\Vert \tilde{u}\right\Vert _{L^{2}}^{2}\left\Vert \nabla
\bar{u}\right\Vert _{L^{\infty }}+\alpha \left\Vert \nabla \tilde{u}%
\right\Vert _{L^{2}}\left\Vert \tilde{u}\right\Vert _{L^{2}}\alpha
\left\Vert \partial ^{2}\bar{u}\right\Vert _{L^{\infty }}\right) .
\end{equation*}%
Hence, thanks to \eqref{eq:u_bar_bounds},
\begin{equation*}
\frac{1}{2}\frac{d}{dt}\left( \left\Vert \tilde{u}\right\Vert
_{L^{2}}^{2}+\alpha ^{2}\left\Vert \nabla \tilde{u}\right\Vert
_{L^{2}}^{2}\right) \leq C\left\Vert \nabla \bar{v}\right\Vert _{L^{\infty
}}\left( \alpha ^{2}\left\Vert \nabla \tilde{u}\right\Vert
_{L^{2}}^{2}+\left\Vert \tilde{u}\right\Vert _{L^{2}}^{2}\right) ,
\end{equation*}%
and by Gr\"{o}nwall inequality%
\begin{align*}
\left\Vert \tilde{u}\left( \cdot ,t\right) \right\Vert _{L^{2}}^{2}+\alpha
^{2}\left\Vert \nabla \tilde{u}\left( \cdot ,t\right) \right\Vert
_{L^{2}}^{2}& \leq e^{C\left\Vert \nabla \bar{v}\right\Vert _{L^{\infty
}}t}\left( \left\Vert \tilde{u}\left( \cdot ,0\right) \right\Vert
_{L^{2}}^{2}+\alpha ^{2}\left\Vert \nabla \tilde{u}\left( \cdot ,0\right)
\right\Vert _{L^{2}}^{2}\right) \\
& \leq e^{C\left\Vert \nabla \bar{v}\right\Vert _{L^{\infty }}t}\left\Vert
\tilde{v}\left( \cdot ,0\right) \right\Vert _{L^{2}}^{2}.
\end{align*}%
Hence we have that for all $0<\varepsilon \leq \varepsilon _{0}$, $0\leq
t\leq T$, the solution of Euler-$\alpha $ equations with the smoothed
initial data satisfies (we now put back the subindices $\alpha $ and $%
\varepsilon $)
\begin{align*}
\left\Vert u_{\alpha ,\varepsilon }\left( \cdot ,t\right) \right\Vert
_{L^{2}\left( B\left( x_{0},1\right) \right) }& \leq \left\Vert \tilde{u}%
_{\alpha ,\varepsilon }\left( \cdot ,t\right) \right\Vert _{L^{2}\left(
B\left( x_{0},1\right) \right) }+\left\Vert \bar{u}\right\Vert _{L^{2}\left(
B\left( x_{0},1\right) \right) } \\
& \leq \left\Vert \tilde{u}_{\alpha ,\varepsilon }\left( \cdot ,t\right)
\right\Vert _{L^{2}\left( \mathbb{R}^{2}\right) }+\pi \left\Vert \bar{u}%
\right\Vert _{L^{\infty }\left( \mathbb{R}^{2}\right) } \\
& \leq C(T),
\end{align*}%
where $C(T)=C\left( \left\Vert q^{in}\right\Vert _{\mathcal{M}},\left\Vert
q^{in}\right\Vert _{H^{-1}},\left\Vert \bar{q}\right\Vert _{L^{\infty
}},\varepsilon _{0},R_{0}\right) e^{C\left\Vert \nabla \bar{v}\right\Vert
_{L^{\infty }}t}+\pi \left\Vert \bar{u}\right\Vert _{L^{\infty }\left(
\mathbb{R}^{2}\right) }$. This is enough to show uniform decay of the
vorticity $\omega _{\alpha ,\varepsilon }$ in small disks (see \cite{a_S95},
we remark that here the fixed sign of the vorticity comes into place\footnote{%
In \cite{a_S95} to prove the uniform decay of the vorticity in small circles
one defines for $R<1$%
\begin{equation*}
\mathcal{\delta }_{R}\left( x\right) =\left\{
\begin{array}{cc}
1 & \left\vert x\right\vert \leq R, \\
\frac{\log \left( \sqrt{R}/\left\vert x\right\vert \right) }{\log \left( 1/%
\sqrt{R}\right) } & R\leq \left\vert x\right\vert \leq \sqrt{R}, \\
0 & \left\vert x\right\vert \geq \sqrt{R}.%
\end{array}%
\right.
\end{equation*}%
Then $\left\vert \nabla \mathcal{\delta }_{R}\right\vert _{L^{2}}\leq
C\left\vert \log R\right\vert ^{-1/2}$ . We have%
\begin{align*}
\int_{\left\vert x-x_{0}\right\vert \leq R}\omega _{\alpha ,\varepsilon
}\left( x,t\right) dx& \leq \int_{\mathbb{R}^{2}}\mathcal{\delta }_{R}\left(
x-x_{0}\right) \omega _{\alpha ,\varepsilon }\left( x,t\right) dx \\
& \leq \left\vert \int_{\mathbb{R}^{2}}\nabla ^{\perp }\mathcal{\delta }%
_{R}\left( x-x_{0}\right) u_{\alpha ,\varepsilon }\left( x,t\right)
dx\right\vert \\
& \leq \left\vert \nabla \mathcal{\delta }_{R}\right\vert _{L^{2}}\left\vert
u_{\alpha ,\varepsilon }\left( \cdot ,t\right) \right\vert _{L^{2}\left(
B\left( x_{0},1\right) \right) } \\
& \leq C\left( T\right) \left\vert \log R\right\vert ^{-1/2}.
\end{align*}%
Here in the second transaction we used the fact that $\omega _{\alpha
,\varepsilon }\geq 0$.} 
): for $R<1$, $\varepsilon \leq \varepsilon _{0}$
\begin{equation*}
\sup_{0\leq t\leq T}\int_{\left\vert x-x_{0}\right\vert \leq R}\omega
_{\alpha ,\varepsilon }\left( x,t\right) dx\leq C\left( T\right) \left\vert
\log R\right\vert ^{-1/2}.
\end{equation*}%
By \eqref{eq:omega_M_norm_bound} $\left\Vert \omega _{\alpha ,\varepsilon
}\left( \cdot ,t\right) \right\Vert _{\mathcal{M}}\leq \left\Vert q_{\alpha
,\varepsilon }\left( \cdot ,t\right) \right\Vert _{\mathcal{M}}=\left\Vert
q_{\varepsilon }^{in}\right\Vert _{\mathcal{M}}=\left\Vert q^{in}\right\Vert
_{\mathcal{M}}$, hence there exists a subsequence $\omega _{\alpha
,\varepsilon _{j}}$ which converges weak-$\ast $ in $L^{\infty }\left( \left[
0,T\right] ,\mathcal{M}\left( \mathbb{R}^{2}\right) \right) $ to the limit $%
\omega _{\alpha }$. This limit has a similar decay%
\begin{equation*}
\sup_{0\leq t\leq T}\int_{\left\vert x-x_{0}\right\vert <R}d\omega _{\alpha
}\left( x,t\right) \leq \liminf_{\varepsilon _{j}\rightarrow 0}\sup_{0\leq
t\leq T}\int_{\left\vert x-x_{0}\right\vert <R}\omega _{\alpha ,\varepsilon
_{j}}\left( x,t\right) dx\leq C\left( T\right) \left\vert \log R\right\vert
^{-1/2}.
\end{equation*}%
Furthermore, $q_{\alpha }=\left( 1-\alpha ^{2}\Delta \right) \omega
_{\alpha }$ is the solution of the Euler-$\alpha $ equations
\eqref{grp:Euler_alpha_vortForm}, the passing to the limit in $%
\lim_{\varepsilon _{j}\rightarrow 0}W^{\alpha }\left( q_{\alpha ,\varepsilon
_{j}};\psi \right) =W^{\alpha }\left( q_{\alpha };\psi \right) $ is
straightforward since $H_{\psi }^{\alpha }$ $\in C\left( \left[ 0,T\right]
,\left( C_{0}\left( \mathbb{R}^{2}\right) \right) ^{2}\right) $ and $%
q_{\alpha ,\varepsilon _{j}}$ are equicontinuous in time with values in a
negative Sobolev space $W^{-2,1}$ (which, together with $q_{\alpha
,\varepsilon _{j}}\overset{\ast }{\rightharpoonup }q_{\alpha }$ in $%
L^{\infty }\left( \left[ 0,T\right] ,\mathcal{M}\left( \mathbb{R}^{2}\right)
\right) $, implies $q_{\alpha ,\varepsilon _{j}}\left( x,t\right) q_{\alpha
,\varepsilon _{j}}\left( y,t\right) \overset{\ast }{\rightharpoonup }%
q_{\alpha }\left( x,t\right) q_{\alpha }\left( y,t\right) $ in $L^{\infty
}\left( \left[ 0,T\right] ,\mathcal{M}\left( \mathbb{R}^{2}\right) \right) $%
, see \cite[Lemma 3.2]{a_S95}). The equicontinuity follows from the fact
that $\left\vert x\right\vert D\Psi ^{\alpha }\left( \left\vert x\right\vert
\right) $ is bounded (in fact, it is bounded independent of $\alpha $) and
hence we have for all $\psi \in C_{c}^{\infty }\left( \mathbb{R}^{2}\times
\left( 0,T\right) \right) $
\begin{align}
& \left\vert \int_{0}^{T}\int_{\mathbb{R}^{2}}\partial _{t}\psi \left(
x,t\right) q_{\alpha ,\varepsilon _{j}}\left( x,t\right) dxdt\right\vert =
\label{eq:equicont} \\
& =\left\vert \frac{1}{2}\int_{0}^{T}\int_{\mathbb{R}^{2}}\int_{\mathbb{R}%
^{2}}D\Psi ^{\alpha }\left( \left\vert x-y\right\vert \right) \frac{\left(
x-y\right) ^{\perp }\cdot \left( \nabla \psi \left( x,t\right) -\nabla \psi
\left( y,t\right) \right) }{\left\vert x-y\right\vert }q_{\alpha
,\varepsilon _{j}}\left( x,t\right) q_{\alpha ,\varepsilon _{j}}\left(
y,t\right) dxdydt\right\vert  \notag \\
& \leq \frac{1}{2}\left\Vert \left\vert x-y\right\vert D\Psi ^{\alpha
}\left( x-y\right) \right\Vert _{L^{\infty }}\int_{0}^{T}\left\Vert
D^{2}\psi \left( \cdot ,t\right) \right\Vert _{L^{\infty }\left( \mathbb{R}%
^{2}\right) }\int_{\mathbb{R}^{2}}q_{\alpha ,\varepsilon _{j}}\left(
x,t\right) dx\int_{\mathbb{R}^{2}}q_{\alpha ,\varepsilon _{j}}\left(
y,t\right) dydt  \notag \\
& \leq C\left\Vert q^{in}\right\Vert _{\mathcal{M}}^{2}\left\Vert \psi
\right\Vert _{L^{1}\left( \left[ 0,T\right] ,W^{2,\infty }\left( \mathbb{R}%
^{2}\right) \right) }  \notag \\
& \leq C\left\Vert q^{in}\right\Vert _{\mathcal{M}}^{2}\left\Vert \psi
\right\Vert _{L^{1}\left( \left[ 0,T\right] ,H^{4}\left( \mathbb{R}%
^{2}\right) \right) },  \notag
\end{align}%
where in the last inequality we used the Sobolev embedding theorem. Hence $%
\partial _{t}q_{\alpha ,\varepsilon _{j}}$ are uniformly bounded in $%
L^{\infty }\left( \left[ 0,T\right] ,H^{-4}\left( \mathbb{R}^{2}\right)
\right) $, and hence $q_{\alpha ,\varepsilon _{j}}$ are uniformly bounded in
$Lip\left( \left[ 0,T\right] ;H^{-4}\left( \mathbb{R}^{2}\right) \right) $.
\end{proof}

We also need the following result

\begin{lemma}
\label{lemma:laplacianOmegaNorm}Let $q$ be a finite Radon measure, $q=\left(
1-\alpha ^{2}\Delta \right) \omega ,$ as defined in \eqref{eq:omega_from_q}-%
\eqref{eq:laplacian_omega_from_q}, then
\begin{equation*}
\int_{\mathbb{R}^{2}}d\left\vert \alpha ^{2}\Delta \omega \right\vert \leq
C\alpha \left\Vert q\right\Vert _{\mathcal{M}}.
\end{equation*}
\end{lemma}

\begin{proof}
For the theory of Radon measures, see, e.g., \cite{b_F99}. First, we show
that for all compact $K\subset \mathbb{R}^{2}$%
\begin{equation*}
\left\vert \alpha ^{2}\Delta \omega \right\vert \left( K\right) \leq C\alpha
\left\Vert q\right\Vert _{\mathcal{M}}.
\end{equation*}%
By Riesz representation theorem (see, e.g., \cite[Chapter 7]{b_F99} )
\begin{equation*}
\left\vert \alpha ^{2}\Delta \omega \right\vert \left( K\right) =\inf
\left\{ \int_{\mathbb{R}^{2}}fd\left\vert \alpha ^{2}\Delta \omega
\right\vert :f\in C_{c}\left( \mathbb{R}^{2}\right) ,f\geq \chi _{K}\right\}
.
\end{equation*}%
Let $R$ be such that $K\subset B\left( 0,R\right) $, take $\theta \in
C_{c}^{\infty }(\mathbb{R}^{2})$ with $0\leq \theta (x)\leq 1$ for all $x$,
with $\theta (x)=1$ if $|x|\leq R$, $\theta (x)=0$ if $|x|\geq R+1$. For
example, $\theta =\chi _{B\left( 0,R+1/2\right) }\ast \rho ^{\varepsilon
=1/4}$. Then by \eqref{eq:omega_from_q} and using that $\left\Vert \alpha
\Delta \left( 1-\alpha ^{2}\Delta \right) ^{-1}\theta \right\Vert
_{L^{\infty }}\leq C\left\Vert \nabla \theta \right\Vert _{L^{\infty }}$
(see, e.g., \eqref{eq:u_bar_bounds}), we have
\begin{align*}
\left\vert \alpha ^{2}\Delta \omega \right\vert \left( K\right) & \leq \int_{%
\mathbb{R}^{2}}\theta d\left\vert \alpha ^{2}\Delta \omega \right\vert \\
& \leq \int_{\mathbb{R}^{2}}\left\vert \alpha ^{2}\Delta \left( 1-\alpha
^{2}\Delta \right) ^{-1}\theta \right\vert d\left\vert q\right\vert \\
& \leq C\alpha \left\Vert q\right\Vert _{\mathcal{M}}\left\Vert \nabla
\theta \right\Vert _{L^{\infty }} \\
& \leq C\alpha \left\Vert q\right\Vert _{\mathcal{M}}.
\end{align*}%
Now, since a Radon measure is inner regular we have
\begin{align*}
\left\vert \alpha ^{2}\Delta \omega \right\vert \left( \mathbb{R}^{2}\right)
& =\sup \left\{ \left\vert \alpha ^{2}\Delta \omega \right\vert \left(
K\right) :K\subset \mathbb{R}^{2},K\text{ compact}\right\} \\
& \leq C\alpha \left\Vert q\right\Vert _{\mathcal{M}}.
\end{align*}
\end{proof}

Now we are ready to prove Theorem \ref{thm:Convergence}. We notice that
\eqref{eq:equicont} implies that $q_{\alpha }\in Lip\left( \left[ 0,T\right]
;H^{-4}\left( \mathbb{R}^{2}\right) \right) $. Hence due to
\eqref{eq:Mspace_vort_bound} there exists a subsequence, that we relabel as $%
q_{\alpha }$, such that $q_{\alpha }\rightharpoonup q$ weak-$\ast $ in $%
L^{\infty }\left( \left[ 0,T\right] ,\mathcal{M}\left( \mathbb{R}^{2}\right)
\right) $ and in $\mathcal{M}\left( \mathbb{R}^{2}\right) $ for each fixed $%
t $, as $\alpha \rightarrow 0$. Also, due to $q_{\alpha }\in Lip\left( \left[
0,T\right] ;H^{-4}\left( \mathbb{R}^{2}\right) \right) $ and $q_{\alpha
}\rightharpoonup q$ weak-$\ast $ in $L^{\infty }\left( \left[ 0,T\right] ,%
\mathcal{M}\left( \mathbb{R}^{2}\right) \right) $, we have that $q_{\alpha
}\left( x,t\right) q_{\alpha }\left( y,t\right) \rightharpoonup q\left(
x,t\right) q\left( y,t\right) $ weak-$\ast $ both in $L^{\infty }\left( %
\left[ 0,T\right] ,\mathcal{M}\left( \mathbb{R}^{2}\right) \right) $ and in $%
\mathcal{M}\left( \mathbb{R}^{2}\right) $ for each fixed $t$ $\in \left[ 0,T%
\right] $, as $\alpha \rightarrow 0$ (see \cite[Lemma 3.2]{a_S95}).

Since $q_{\alpha }$ is uniformly bounded in $\mathcal{M}(\mathbb{R} ^{2})$
and $Lip\left( \left[ 0,T\right] ;H^{-4}\left( \mathbb{R} ^{2}\right)
\right) $ (by \eqref{eq:equicont}), $\mathcal{M}(\mathbb{R}
^{2})\hookrightarrow H_{loc}^{-s}\left( \mathbb{R} ^{2}\right) \overset{comp}%
{\hookrightarrow }$ $H_{loc}^{-4}\left( \mathbb{R} ^{2}\right) $ for $1<s<4$%
, then by Arzela-Ascoli theorem there is a subsequence of $q_{\alpha }$ that
converges to some $\bar{q}$ in $C\left( \left[ 0,T\right] ;H_{loc}^{-4}%
\right) $, and hence $\bar{q}$ is also in $Lip\left( \left[ 0,T\right]
;H_{loc}^{-4}\right) $. Applying both types of convergence of the $q_{\alpha
}$ to the integral $\int_{0}^{T}\int_{\mathbb{R} ^{2}}\psi \left( t\right)
\varphi \left( x\right) dq_{\alpha }\left( x,t\right) $ for every $\psi \in
C_{c}\left( \left[ 0,T\right] \right) ,\varphi \in C_{c}\left( \mathbb{R}%
^{2}\right) $ shows that $\bar{q}=q$, and hence the limit $q$ belongs to $%
Lip\left( \left[ 0,T\right] ,H_{loc}^{-4}\left( \mathbb{R} ^{2}\right)
\right) $ as well.

We observe that $\omega _{\alpha }\left( t,\cdot \right) $ also weak-$\ast $
converges to $q$ in $\mathcal{M}\left( \mathbb{R}^{2}\right) $ for every $%
t\in \left[ 0,T\right] $, as $\alpha \rightarrow 0$. Indeed, let $\varphi
\in C_{c}\left( \mathbb{R}^{2}\right) $ then%
\begin{align*}
\left\vert \int_{\mathbb{R}^{2}}\varphi \left( x\right) dq\left( x,t\right)
-\int_{\mathbb{R}^{2}}\varphi \left( x\right) d\omega _{\alpha }\left(
x,t\right) \right\vert & \leq \left\vert \int_{\mathbb{R}^{2}}\varphi \left(
x\right) dq\left( x,t\right) -\int_{\mathbb{R}^{2}}\varphi \left( x\right)
dq_{\alpha }\left( x,t\right) \right\vert \\
& +\left\vert \int_{\mathbb{R}^{2}}\varphi \left( x\right) dq_{\alpha
}\left( x,t\right) -\int_{\mathbb{R}^{2}}\varphi \left( x\right) d\omega
_{\alpha }\left( x,t\right) \right\vert
\end{align*}%
the first term on the right-hand side converges to $0$, since $q_{\alpha }%
\overset{\ast }{\rightharpoonup }q$ in $\mathcal{M}\left( \mathbb{R}%
^{2}\right) $, as $\alpha \rightarrow 0$, and the second term is equal to $%
\left\vert \int_{\mathbb{R}^{2}}\varphi d\left( \alpha ^{2}\Delta \omega
_{\alpha }\right) \right\vert \leq \left\Vert \varphi \right\Vert
_{L^{\infty }}\int_{\mathbb{R}^{2}}d\left\vert \alpha ^{2}\Delta \omega
\right\vert \rightarrow 0$, as $\alpha \rightarrow 0$, due to Lemma \ref%
{lemma:laplacianOmegaNorm}. Hence also $q$ decays in small disks, that is,
for all $0\leq t\leq T$, $0<R<1$ and $x_{0}\in {\mathbb{R}}^{2}$
\begin{equation}
\int_{\left\vert x-x_{0}\right\vert <R}dq\left( x,t\right) \leq
\liminf_{\alpha \rightarrow 0}\int_{\left\vert x-x_{0}\right\vert <R}d\omega
_{\alpha }\left( x,t\right) \leq C\left( T\right) \left\vert \log
R\right\vert ^{-1/2}.  \label{eq:q_decay_in_small_circles}
\end{equation}

Next we show that $q$ is a weak solution of the Euler equations %
\eqref{eq:EulerWeakVortForm}, namely, for every test function $\psi \in
C_{c}^{\infty }\left( \mathbb{R} ^{2}\times \left( 0,T\right) \right) $
\begin{equation*}
W\left( q;\psi \right) =\lim_{\alpha \rightarrow 0}W^{\alpha }\left(
q_{\alpha };\psi \right) =0\text{.}
\end{equation*}%
The convergence of the linear term is obvious from the weak-$\ast $
convergence $q_{\alpha }\rightharpoonup q$ in $L^{\infty }\left( \left[ 0,T%
\right] ;\mathcal{M}({\mathbb{R}}^{2})\right) $, as $\alpha \rightarrow 0$.
Hence we need to show the convergence for the non-linear term%
\begin{equation*}
\lim_{\alpha \rightarrow 0}W_{NL}^{\alpha }\left( q_{\alpha };\psi \right)
=W_{NL}\left( q;\psi \right) .
\end{equation*}%
We rewrite $W_{NL}\left( q;\psi \right) -W_{NL}^{\alpha }\left( q_{\alpha
};\psi \right) $ as%
\begin{align*}
W_{NL}\left( q;\psi \right) -W_{NL}^{\alpha }\left( q_{\alpha };\psi \right)
& =\int_{0}^{T}\int_{\mathbb{R} ^{2}}\int_{\mathbb{R} ^{2}}H_{\psi }\left(
x,y,t\right) \left[ dq\left( x,t\right) dq\left( y,t\right) -dq_{\alpha
}\left( x,t\right) dq_{\alpha }\left( y,t\right) \right] dt \\
& +\int_{0}^{T}\int_{\mathbb{R} ^{2}}\int_{\mathbb{R} ^{2}}\left( H_{\psi
}\left( x,y,t\right) -H_{\psi }^{\alpha }\left( x,y,t\right) \right)
dq_{\alpha }\left( x,t\right) dq_{\alpha }\left( y,t\right) dt \\
& =I_{1}+I_{2}.
\end{align*}%
We recall that the kernel $H_{\psi }$ is bounded by a constant times $%
\left\Vert D^{2}\psi \right\Vert _{L^{\infty }}$, tends to zero at infinity,
and it is discontinuous on the diagonal $x=y$ (see \cite{a_S95}).

Let $\theta \left( \left\vert x\right\vert \right) \in C_{c}^{\infty }(%
\mathbb{R}^{2})$ be a fixed cutoff function $0\leq \theta \leq 1$ with $%
\theta =1$ for $\left\vert x\right\vert \leq 1$ and $\theta $ $=0$ for $%
\left\vert x\right\vert \geq 2$. Let $0<\delta <1$. Write $I_{1}$ as%
\begin{align*}
I_{1}& =\int_{0}^{T}\int_{\mathbb{R}^{2}}\int_{\mathbb{R}^{2}}\left[
1-\theta \left( \frac{\left\vert x-y\right\vert }{\delta }\right) \right]
H_{\psi }\left( x,y,t\right) \left( dq\left( x,t\right) dq\left( y,t\right)
-dq_{\alpha }\left( x,t\right) dq_{\alpha }\left( y,t\right) \right) dt \\
& +\int_{0}^{T}\int_{\mathbb{R}^{2}}\int_{\mathbb{R}^{2}}\theta \left( \frac{%
\left\vert x-y\right\vert }{\delta }\right) H_{\psi }\left( x,y,t\right)
\left( dq\left( x,t\right) dq\left( y,t\right) -dq_{\alpha }\left(
x,t\right) dq_{\alpha }\left( y,t\right) \right) dt \\
& =I_{11}+I_{12}.
\end{align*}%
Since $\left[ 1-\theta \left( \frac{\left\vert x-y\right\vert }{\delta }%
\right) \right] H_{\psi }\in C\left( \left[ 0,T\right] ,\left( C_{0}\left(
\mathbb{R}^{2}\right) \right) ^{2}\right) $ and $q_{\alpha }\left(
x,t\right) q_{\alpha }\left( y,t\right) \rightharpoonup q\left( x,t\right)
q\left( y,t\right) $ weak-$\ast $ in $L^{\infty }\left( \left[ 0,T\right] ,%
\mathcal{M}\left( \mathbb{R}^{2}\right) \right) $ as $\alpha \rightarrow 0$,
then $\lim_{\alpha \rightarrow 0}I_{11}=0$. Now we estimate $I_{12}$
\begin{align*}
\left\vert I_{12}\right\vert & \leq
\int_{0}^{T}\int_{\mathbb{R}^{2}} \int_{\left\vert x-y\right\vert
<2\delta }\left\vert H_{\psi }\left( x,y,t\right) \right\vert
dq\left( x,t\right) dq\left( y,t\right) dt \\
& +\int_{0}^{T}\int_{\mathbb{R}^{2}} \int_{\left\vert x-y\right\vert
<2\delta }\left\vert H_{\psi }\left( x,y,t\right) \right\vert
dq_{\alpha }\left( x,t\right)
dq_{\alpha }\left( y,t\right) dt \\
& =I_{121}+I_{122}.
\end{align*}%
For $I_{121}$, due to uniform decay of the vorticity $q$ in small disks %
\eqref{eq:q_decay_in_small_circles}, we have for $2\delta <1$
\begin{align*}
I_{121}& \leq \left\vert H_{\psi }\right\vert _{L^{\infty
}}\int_{0}^{T}\int_{\mathbb{R}^{2}}\int_{B\left( y,2\delta \right) }dq\left(
x,t\right) dq\left( y,t\right) dt \\
& \leq C\left( T\right) \left\vert \log 2\delta \right\vert
^{-1/2}\left\Vert q^{in}\right\Vert _{\mathcal{M}}.
\end{align*}%
To estimate $I_{122}$ we use \eqref{eq:uniform_decay} (for $2\delta <1$) and
Lemma \ref{lemma:laplacianOmegaNorm}.
\begin{align*}
I_{122}& \leq \left\vert H_{\psi }\right\vert _{L^{\infty
}}\int_{0}^{T}\int_{\mathbb{R}^{2}} \int_{\left\vert x-y\right\vert
<2\delta }d\left( \left( 1-\alpha ^{2}\Delta \right) \omega _{\alpha
}\right) \left( x,t\right) d\left( \left( 1-\alpha
^{2}\Delta \right) \omega _{\alpha }\right) \left( y,t\right) )dt \\
& =\left\vert H_{\psi }\right\vert _{L^{\infty
}}\int_{0}^{T}\int_{\mathbb{R}^{2}} \int_{\left\vert x-y\right\vert
<2\delta }d\omega _{\alpha }\left(
x,t\right) d\omega _{\alpha }\left( y,t\right) )dt \\
& +\left\vert H_{\psi }\right\vert _{L^{\infty }}\int_{0}^{T}\left( 2\int_{%
\mathbb{R}^{2}}d\omega _{\alpha }\left( x,t\right) \int_{\mathbb{R}%
^{2}}d\left\vert \alpha ^{2}\Delta \omega _{\alpha }\left( x,t\right)
\right\vert +\left( \int_{\mathbb{R}^{2}}d\left\vert \alpha ^{2}\Delta
\omega _{\alpha }\left( x,t\right) \right\vert \right) ^{2}\right) dt \\
& \leq C\left( T\right) \left\vert \log 2\delta \right\vert
^{-1/2}\left\Vert q^{in}\right\Vert _{\mathcal{M}}^{2}+\alpha \left(
1+\alpha \right) CT\left\Vert q^{in}\right\Vert _{\mathcal{M}}^{2}.
\end{align*}%
Thus, $I_{12}\rightarrow 0$, as $\delta $ and $\alpha $ converge to zero.

It remains to estimate $I_{2}$, by \eqref{eq:kernel_H}, %
\eqref{eq:kernel_H_alpha} and \eqref{eq:Psi_DPsi_def}%
\begin{equation*}
I_{2}=\frac{1}{4\pi }\int_{0}^{T}\int_{\mathbb{R}^{2}}\int_{\mathbb{R}^{2}}%
\frac{1}{\alpha }K_{1}\left( \frac{\left\vert x-y\right\vert }{\alpha }%
\right) \left( x-y\right) ^{\perp }\frac{\nabla \left( \psi \left(
x,t\right) -\psi \left( y,t\right) \right) }{\left\vert x-y\right\vert }%
dq_{\alpha }\left( x,t\right) dq_{\alpha }\left( y,t\right) dt.
\end{equation*}%
Now, for $\frac{r}{\alpha }\rightarrow \infty $ $\frac{r}{\alpha }%
K_{1}\left( \frac{r}{\alpha }\right) \leq C\left( \frac{\pi }{2}\right)
^{1/2}\frac{\sqrt{\frac{r}{\alpha }}}{e^{\frac{r}{\alpha }}}\rightarrow 0$
\cite{b_W44}. Hence, for each $\varepsilon >0$, there is an $L$ large
enough, depending on $\varepsilon $, such that $\frac{r}{\alpha }K_{1}\left(
\frac{r}{\alpha }\right) <\varepsilon $, whenever $\frac{r}{\alpha }\geq L$.
Write $I_{2}$ as%
\begin{align*}
I_{2}& =\frac{1}{4\pi }\int_{0}^{T}\int_{\mathbb{R}^{2}}\int_{\mathbb{R}^{2}}%
\left[ 1-\theta \left( \frac{\left\vert x-y\right\vert }{\alpha L}\right) %
\right] \frac{1}{\alpha }K_{1}\left( \frac{\left\vert x-y\right\vert }{%
\alpha }\right) \left( x-y\right) ^{\perp }\frac{\nabla \left( \psi \left(
x,t\right) -\psi \left( y,t\right) \right) }{\left\vert x-y\right\vert }%
dq_{\alpha }\left( x,t\right) dq_{\alpha }\left( y,t\right) dt \\
& +\frac{1}{4\pi }\int_{0}^{T}\int_{\mathbb{R}^{2}}\int_{\mathbb{R}%
^{2}}\theta \left( \frac{\left\vert x-y\right\vert }{\alpha L}\right) \frac{1%
}{\alpha }K_{1}\left( \frac{\left\vert x-y\right\vert }{\alpha }\right)
\left( x-y\right) ^{\perp }\frac{\nabla \left( \psi \left( x,t\right) -\psi
\left( y,t\right) \right) }{\left\vert x-y\right\vert }dq_{\alpha }\left(
x,t\right) dq_{\alpha }\left( y,t\right) dt \\
& =I_{21}+I_{22}.
\end{align*}%
We have%
\begin{align*}
I_{21}& \leq \frac{1}{4\pi }\int_{0}^{T}\int_{\mathbb{R}^{2}}
\int_{\frac{\left\vert x-y\right\vert }{\alpha }>L}\left[ 1-\theta
\left( \frac{\left\vert
x-y\right\vert }{\alpha L}\right) \right] \frac{\left\vert x-y\right\vert }{%
\alpha }K_{1}\left( \frac{\left\vert x-y\right\vert }{\alpha }\right) \frac{%
\left\vert \nabla \left( \psi \left( x,t\right) -\psi \left( y,t\right)
\right) \right\vert }{\left\vert x-y\right\vert }dq_{\alpha }\left(
x,t\right) dq_{\alpha }\left( y,t\right) dt \\
& \leq \varepsilon \left\Vert D^{2}\psi \right\Vert _{L^{\infty }}\frac{1}{%
4\pi }\int_{0}^{T}\int_{\mathbb{R}^{2}} \int_{\frac{\left\vert x-y\right\vert }{\alpha }%
>L}dq_{\alpha }\left( x,t\right) dq_{\alpha }\left( y,t\right) dt \\
& \leq \varepsilon \left\Vert D^{2}\psi \right\Vert _{L^{\infty }}\frac{1}{%
4\pi }\int_{0}^{T}\left( \int_{\mathbb{R}^{2}}dq_{\alpha }\left( x,t\right)
\right) ^{2} \\
& \leq \varepsilon \left\Vert D^{2}\psi \right\Vert _{L^{\infty }}\frac{1}{%
4\pi }T\left\Vert q^{in}\right\Vert _{\mathcal{M}}^{2}.
\end{align*}%
Since $\frac{r}{\alpha }K_{1}\left( \frac{r}{\alpha }\right) \leq C$ for all
$r$ (independent of $\alpha )$, then similarly to the bound on $I_{122}$, we
have that for $\alpha <\frac{1}{2L}$
\begin{equation*}
I_{22}\leq C\left( T\right) \left\vert \log 2\alpha L\right\vert
^{-1/2}\left\Vert D^{2}\psi \right\Vert _{L^{\infty }}\left\Vert
q^{in}\right\Vert _{\mathcal{M}}+\alpha \left( 1+\alpha \right) CT\left\Vert
D^{2}\psi \right\Vert _{L^{\infty }}\left\Vert q^{in}\right\Vert _{\mathcal{M%
}}^{2}.
\end{equation*}%
Hence for each $\varepsilon >0$, there is an $L$ large enough, depending on $%
\varepsilon $, such that (for $\alpha <\frac{1}{2L}$)
\begin{equation*}
I_{2}\leq C\left( T,\psi ,\left\Vert q^{in}\right\Vert _{\mathcal{M}}\right)
(\varepsilon +\left\vert \log 2\alpha L\right\vert ^{-1/2}+\alpha \left(
1+\alpha \right) ).
\end{equation*}%
For each $\varepsilon >0$, there is $\delta ^{\ast }$ such that $\left\vert
\log r\right\vert ^{-1/2}<\varepsilon $, whenever $r<\delta ^{\ast }$.
Hence, for $\alpha <\min \left\{ \frac{\delta ^{\ast }}{2L},\frac{1}{2L}%
,\varepsilon \right\} $
\begin{equation*}
I_{2}\leq \varepsilon C\left( T,\psi ,\left\Vert q^{in}\right\Vert _{%
\mathcal{M}}\right) .
\end{equation*}%
Therefore, $\lim_{\alpha \rightarrow 0}I_{2}=0$. This concludes the proof
that $q$ is a weak solution of the Euler equations %
\eqref{eq:EulerWeakVortForm} with initial data $q^{in}$.

\section{\label{sec:BR_alpha}Birkhoff-Rott-$\protect\alpha $ equation}

The Birkhoff-Rott-$\alpha $ equation, based on the Euler-$\alpha $ equations %
\eqref{grp:Euler_alpha_vortForm} is derived similarly to the original
Birkhoff-Rott equation. Detailed descriptions of the Birkhoff-Rott equation
as a model for the evolution of the vortex sheet can be found, e.g., in \cite%
{b_S92,b_MP94,b_MB02}. We remark, however, that while the BR equations
assume \textit{a priori} that a vortex sheet remains a curve at a later
time, in the 2D Euler-$\alpha $ case, if the vorticity is initially
supported on a curve, then due to the existence of the unique Lagrangian
flow map $\partial _{t}\eta (x,t)=\int_{ \mathbb{R} ^{2}}K^{\alpha }\left(
x,y\right) dq\left( y,t\right) $, \thinspace $\eta \left( x,0\right) =x$, $%
q\left( x,t\right) =q^{in}\circ \eta ^{-1}\left( x,t\right) $, given by
Theorem \ref{thm:OS01} of Oliver and Shkoller \cite{a_OS01}, it remains
supported on a curve for all times. Existence of the unique Lagrangian flow
map implies that the BR-$\alpha $ equation gives an equivalent description
of the vortex sheet evolution, as the weak solution of 2D Euler-$\alpha $
equations. It is described in the following proposition.

\begin{proposition}
\label{prop:BR_alpha:equivalence} Let $q^{in}$ $\in \mathcal{M}({\mathbb{R}}%
^{2})$ supported on the sheet (curve) \newline
$\Sigma ^{in}=\left\{ x=x(\sigma )\in \mathbb{R}^{2}|\sigma _{0}^{in}\leq
\sigma \leq \sigma _{1}^{in}\right\} $, with a density $\gamma ^{in}(\sigma
) $, that is, the vorticity $q^{in}$ satisfies
\begin{equation*}
\int_{{\mathbb{R}}^{2}}\varphi (x)dq^{in}(x)=\int_{\sigma _{0}^{in}}^{\sigma
_{1}^{in}}\varphi \left( x(\sigma )\right) \gamma ^{in}(\sigma )|x_{\sigma
}\left( \sigma \right) |d\sigma ,
\end{equation*}%
for every $\varphi \in C_{c}^{\infty }\left( {\mathbb{R}}^{2}\right) $, $%
\gamma ^{in}\in L^{1}(\left\vert x_{\sigma }\right\vert d\sigma )\footnote{%
Let $\Sigma $ be a curve parametrized by $x(\sigma ):\left[ \sigma
_{0},\sigma _{1}\right] \rightarrow \mathbb{R}^{2}$, such that $x_{\sigma }$
$\in L^{1}(\left[ \sigma _{0},\sigma _{1}\right] )$, and let $q\in $ $%
\mathcal{M}({\mathbb{R}}^{2})$ be supported on the curve $\Sigma $, with a
density $\gamma $. Then $\gamma \in L^{1}(\left\vert x_{\sigma }\right\vert
d\sigma )$ (and vice versa).
\par
\begin{proof}
First, assume $q\geq 0$, and let $\theta _{n}$ be a truncating sequence, $%
\theta _{n}\in C_{c}^{\infty }(\mathbb{R}^{2})$, $\theta _{n}\left( x\right)
=\theta _{1}\left( \frac{x}{n}\right) $, $\theta _{1}\in C_{c}^{\infty }(%
\mathbb{R}^{2})$, $0\leq \theta _{1}\leq 1$ with $\theta _{1}=1$ for $%
\left\vert x\right\vert \leq 1$ and $\theta _{1}$ $=0$ for $\left\vert
x\right\vert \geq 2$. Then, on the one hand,
\begin{equation*}
\int_{\sigma _{0}}^{\sigma _{1}}\theta _{n}\left( x(\sigma )\right) \gamma
(\sigma )|x_{\sigma }\left( \sigma \right) |d\sigma \geq \int_{\left\{
\sigma :\left\vert x\left( \sigma \right) \right\vert \leq n\right\} \cap
\left[ \sigma _{0},\sigma _{1}\right] }\gamma (\sigma )|x_{\sigma }\left(
\sigma \right) |d\sigma ,
\end{equation*}%
on the other hand%
\begin{equation*}
\int_{\sigma _{0}}^{\sigma _{1}}\theta _{n}\left( x(\sigma )\right) \gamma
(\sigma )|x_{\sigma }\left( \sigma \right) |d\sigma =\int_{{\mathbb{R}}%
^{2}}\theta _{n}(x)dq^{in}(x)\leq \left\Vert \theta _{n}\right\Vert
_{L^{\infty }}\left\Vert q\right\Vert _{\mathcal{M}}\leq \left\Vert
q\right\Vert _{\mathcal{M}},
\end{equation*}%
hence%
\begin{equation*}
\int_{\left\{ \sigma :\left\vert x\left( \sigma \right) \right\vert \leq
n\right\} \cap \left[ \sigma _{0},\sigma _{1}\right] }\gamma (\sigma
)|x_{\sigma }\left( \sigma \right) |d\sigma \leq \left\Vert q\right\Vert _{%
\mathcal{M}}.
\end{equation*}%
Since $n$ can be taken arbitrary large this implies that $\int_{\sigma
_{0}}^{\sigma _{1}}\gamma (\sigma )|x_{\sigma }\left( \sigma \right)
|d\sigma <\infty $. Now, for a signed measure we apply the previous result
to each of the nonnegative measures $q^{+}$, $q^{-}$, given by the Jordan
Decomposition of $q$, $q=q^{+}-q^{-}$, which is defined by
\begin{equation*}
\int_{{\mathbb{R}}^{2}}\varphi (x)dq^{\pm }(x)=\int_{\sigma _{0}}^{\sigma
_{1}}\varphi \left( x(\sigma )\right) \gamma ^{\pm }(\sigma )|x_{\sigma
}\left( \sigma \right) |d\sigma .
\end{equation*}%
\end{proof}
}$. Let $q$ be the solution of \eqref{grp:Euler_alpha_vortForm} in the sense
of the Theorem \ref{thm:OS01}. Then, for as long as the curve $\Sigma
(t)=\left\{ x=x(\sigma ,t)\in \mathbb{R}^{2}|\text{\thinspace }\sigma
_{0}\left( t\right) \leq \sigma \leq \sigma _{1}\left( t\right) \right\} $
remains nice enough so that $x_{\sigma }$ makes sense a.e., $q$ has a
density $\gamma (\sigma ,t)$ supported on the sheet $\Sigma (t)$, $\gamma
\left( \cdot ,t\right) \in L^{1}(\left\vert x_{\sigma }\right\vert d\sigma )$%
, $\gamma \left( \sigma ,t\right) \left\vert x_{\sigma }\left( \sigma
,t\right) \right\vert d\sigma =\gamma \left( \sigma ,0\right) \left\vert
x_{\sigma }\left( \sigma ,0\right) \right\vert d\sigma $ and the sheet
evolves according to the equation
\begin{equation}
\frac{\partial }{\partial t}x\left( \sigma ,t\right) =\int_{\sigma
_{0}\left( t\right) }^{\sigma _{1}\left( t\right) }K^{\alpha }\left( x\left(
\sigma ,t\right) -x\left( \sigma ^{\prime },t\right) \right) \gamma \left(
\sigma ^{\prime },t\right) \left\vert x_{\sigma }\left( \sigma ^{\prime
},t\right) \right\vert d\sigma ^{\prime },  \label{eq:BR_alpha_gen}
\end{equation}%
where $K^{\alpha }$ is given by \eqref{eq:K_alpha}. Additionally, if $\Gamma
\left( \sigma ,t\right) =\int_{\sigma ^{\ast }}^{\sigma }\gamma \left(
\sigma ^{\prime },t\right) \left\vert x_{\sigma }\left( \sigma ^{\prime
},t\right) \right\vert d\sigma ^{\prime }$, where $x\left( \sigma ^{\ast
},t\right) $ is some fixed reference point on $\Sigma (t)$, defines a
strictly increasing function of $\sigma $ (e.g., as in the case of positive
vorticity), then the evolution equation is given by the Birkhoff-Rott-$%
\alpha $ (BR-$\alpha $) equation
\begin{equation}
\frac{\partial }{\partial t}x\left( \Gamma ,t\right) =\int_{\Gamma
_{0}}^{\Gamma _{1}}K^{\alpha }\left( x\left( \Gamma ,t\right) -x\left(
\Gamma ^{\prime },t\right) \right) d\Gamma ^{\prime }  \label{eq:BR_alpha}
\end{equation}%
with $\gamma =1/|x_{\Gamma }|$ being the vorticity density along the curve
and $-\infty <\Gamma _{0}<\Gamma _{1}<\infty $.
\end{proposition}

In Section \ref{sec:GlobalReg} we show the global well-posedness of the
Birkhoff-Rott-$\alpha $ \eqref{eq:BR_alpha} in the space of Lipschitz
functions and in the H\"{o}lder space $C^{n,\beta }$, $n\geq 1$, which is
the space of $n$-times differentiable functions with H\"{o}lder continuous $%
n^{\text{th}}$ derivative. Thus the solutions the Birkhoff-Rott-$\alpha $
and of the Euler-$\alpha $ are equivalent for the initial data being a
finite positive Radon measure supported on Lipschitz or H\"{o}lder $%
C^{1,\beta }\left( \left( \Gamma _{0},\Gamma _{1}\right) \right) $, $0\leq
\beta <1$, chord arc curve, or supported on $C^{n,\beta }\left( \left(
\Gamma _{0},\Gamma _{1}\right) \right) $, $n\geq 1$, $0<\beta <1$, closed
chord arc curve.

Here $\sigma _{0},\sigma _{1}$ can represent either a finite length curve,
or an infinite one. We remark that the smoothed kernel $K^{\alpha }\left(
x\right) $ is a bounded continuous function, that for $\frac{\left\vert
x\right\vert }{\alpha }\rightarrow 0$ behaves asymptotically as $K^{\alpha
}\left( x\right) =-\frac{1}{4\pi }\frac{1}{\alpha ^{2}}x^{\perp }\log \frac{%
\left\vert x\right\vert }{\alpha }+O\left( \frac{|x|}{\alpha ^{2}}\right) $,
i.e., it is non-singular kernel at the origin. For the case where $\gamma
\left( \cdot ,t\right) \in L^{1}(\left\vert x_{\sigma }\right\vert d\sigma )$
we can show the integrability of the relevant terms, even though $\left\vert
K^{\alpha }\left( x\right) \right\vert $ is decaying like $\left\vert
x\right\vert ^{-1}$ at infinity.

\section{\label{sec:LinStab}Linear stability of a flat vortex sheet with
uniform vorticity density for 2D BR-$\protect\alpha $ model}

The initial data problem for the BR equation is highly unstable due to an
ill-posed response to small perturbations called Kelvin-Helmholtz
instability \cite{a_B62, a_SB79}. The linear stability analysis of the BR-$%
\alpha $ equation shows that the ill-posedness of the original problem is
mollified, and the Kelvin-Helmholtz instability of the original system now
disappears. We assume that the vortex sheet can be parameterized as a graph $%
x_{2}=x_{2}\left( x_{1},t\right) $, the proof can be easily adapted to
establish the result in general. In this case, following calculations
presented in \cite[Section 6.1]{b_MP94}, one can infer from %
\eqref{eq:BR_alpha_gen} the system
\begin{align}
\frac{\partial x_{2}}{\partial t}& =-\frac{\partial x_{2}}{\partial x_{1}}%
u_{1}+u_{2},  \label{eq:2dSystem} \\
\frac{\partial \gamma }{\partial t}& =-\frac{\partial }{\partial x_{1}}%
\left( \gamma u_{1}\right) ,  \notag
\end{align}%
where the velocity field $u=\left( u_{1},u_{2}\right) ^{t}$ is given by
\begin{equation*}
u\left( x_{1},t\right) =\mathrm{p.v.}\int_{\mathbb{R}}K^{\alpha }\left(
x\left( x_{1},t\right) -x\left( x_{1}^{\prime },t\right) \right) \gamma
\left( x_{1}^{\prime },t\right) dx_{1}^{\prime },
\end{equation*}%
here $x\left( x_{1},t\right) =\left( x_{1},x_{2}\left( x_{1},t\right)
\right) ^{t}$ and $\mathrm{p.v.}$ is taken at infinity. We recall that $%
K^{\alpha }\left( x\right) =\frac{x^{\perp }}{\left\vert x\right\vert }D\Psi
^{\alpha }\left( \left\vert x\right\vert \right) $ and $D\Psi ^{\alpha }(r)=%
\frac{1}{2\pi }\left[ -\frac{1}{\alpha }K_{1}\left( \frac{r}{\alpha }\right)
+\frac{1}{r}\right] $, and $K_{1}$ denotes a modified Bessel functions of
the second kind of order one. 
By linearization of \eqref{eq:2dSystem} about the flat sheet $%
x_{2}^{0}\equiv 0$ with uniformly concentrated intensity $\gamma _{0}$,
which is a stationary solution of \eqref{eq:2dSystem}, we obtain the
following linear system
\begin{align*}
& \frac{\partial \tilde{x}_{2}}{\partial t}=\tilde{u}_{2}, \\
& \frac{\partial \tilde{\gamma}}{\partial t}=-\gamma _{0}\frac{\partial
\tilde{u}_{1}}{\partial x_{1}},
\end{align*}%
where
\begin{align*}
\tilde{u}_{1}\left( x_{1},t\right) & =-\gamma _{0}\left( \sgn\left(
x_{1}\right) D\Psi ^{\alpha }\left( \left\vert x_{1}\right\vert \right)
\right) \ast \frac{\partial \tilde{x}_{2}}{\partial x_{1}}, \\
\tilde{u}_{2}\left( x_{1},t\right) & =\left( \sgn\left( x_{1}\right) D\Psi
^{\alpha }\left( \left\vert x_{1}\right\vert \right) \right) \ast \tilde{%
\gamma},
\end{align*}%
and $\left( \tilde{x}_{2},\tilde{\gamma}\right) $ is a small perturbation
about the flat sheet $x_{2}\equiv 0$ with the constant density $\gamma
=\gamma _{0}$. See \cite{a_SSBF81} for the description of the original
Birkhoff-Rott system in such a case.

Consequently, the equation for the Fourier modes (transform) of the above
system is given by
\begin{equation}
\frac{d}{dt}%
\begin{pmatrix}
\widehat{\tilde{x}_{2}} \\
\widehat{\tilde{\gamma}}%
\end{pmatrix}%
=%
\begin{pmatrix}
0 & \frac{i}{2}\sgn(k)d(k) \\
-i\frac{\gamma _{0}^{2}}{2}k^{2}\sgn(k)d(k) & 0%
\end{pmatrix}%
\begin{pmatrix}
\widehat{\tilde{x}_{2}} \\
\widehat{\tilde{\gamma}}%
\end{pmatrix}%
,  \label{eq:FourierModesSystem}
\end{equation}%
where
\begin{equation*}
d(k)=\left( 1+\frac{1}{\alpha ^{2}k^{2}}\right) ^{-1/2}-1.
\end{equation*}%
Observe that in order to calculate the Fourier transform
\begin{equation*}
\mathcal{F}\left( \sgn\left( x_{1}\right) D\Psi ^{\alpha }\left( \left\vert
x_{1}\right\vert \right) \right) \left( k\right) =\frac{i}{2}\sgn(k)d(k),
\end{equation*}%
we used here the integral representation of the modified Bessel function of
the second kind \newline
$K_{1}\left( x_{1}\right) =x_{1}\int_{1}^{\infty }e^{-x_{1}t}\left(
t^{2}-1\right) ^{1/2}dt$, (see, e.g., \cite{b_W44}). The eigenvalues of the
coefficient matrix, given in \eqref{eq:FourierModesSystem}, are
\begin{equation}
\lambda (k)=\pm \frac{1}{2}\left\vert \gamma _{0}\right\vert \left\vert
k\right\vert \left( 1-\left( 1+\frac{1}{\alpha ^{2}k^{2}}\right)
^{-1/2}\right) .  \label{eq:eigenval}
\end{equation}%
We observe that while the linear system for the original Birkhoff-Rott
equation is elliptic (in space and time)
\begin{equation*}
\frac{\partial ^{2}\widehat{\tilde{x}_{2}}}{\partial t^{2}}-\frac{1}{4}%
\gamma _{0}^{2}k^{2}\widehat{\tilde{x}_{2}}=0,
\end{equation*}%
for a Birkhoff-Rott-$\alpha $ equation it is no longer an elliptic system%
\begin{equation*}
\frac{\partial ^{2}\widehat{\tilde{x}_{2}}}{\partial t^{2}}-\frac{1}{4}%
\gamma _{0}^{2}d^{2}(k)k^{2}\widehat{\tilde{x}_{2}}=0,
\end{equation*}%
since $\left\vert d^{2}(k)k^{2}\right\vert $ is bounded and behaves like $%
\frac{1}{\alpha ^{4}k^{2}}$, as $k\rightarrow \infty $ (for $\alpha $ fixed).

To conclude, the $\alpha $-regularization mollifies the Kelvin-Helmholtz
instability as follows: we have an algebraic decay of the eigenvalues to
zero of order $\frac{1}{\alpha ^{2}\left\vert k\right\vert }$, as $%
k\rightarrow \infty $ (for $\alpha $ fixed). While, for $\alpha $ $%
\rightarrow 0$, for fixed $k$, we recover the eigenvalues of the original BR
equations $\pm \frac{1}{2}\left\vert \gamma _{0}\right\vert \left\vert
k\right\vert $ (see, e.g., \cite{a_SSBF81}).

For the sake of comparison, we note that for the vortex blob regularization
of Krasny \cite{a_K86a}, where the singular BR kernel, $K(x)$, was replaced
with the smoothed kernel
\begin{equation*}
K_{\delta }\left( x\right) =K\left( x\right) \frac{\left\vert x\right\vert
^{2}}{\left\vert x\right\vert ^{2}+\delta ^{2}}=\frac{1}{2\pi }\frac{%
x^{\perp }}{\left\vert x\right\vert ^{2}+\delta ^{2}},
\end{equation*}%
the eigenvalues are
\begin{equation*}
\lambda (k)=\pm \frac{1}{2}e^{-\delta k}\left\vert \gamma _{0}\right\vert
\left\vert k\right\vert
\end{equation*}%
with an exponential decay to zero, as $k\rightarrow \infty $ ($\delta >0$ is
fixed). As $\delta $ $\rightarrow 0$, for fixed $k$, one recovers again the
eigenvalues of the original BR equations.

The behavior of the eigenvalues of the linearized system %
\eqref{eq:FourierModesSystem} indicates that high wave number perturbations
grow exponentially in time with a rate that decays to zero, as $k\rightarrow
\infty $, which is the reason for well-posedness of the $\alpha $%
-regularized model. This is unlike the original BR problem that exhibits the
Kelvin-Helmholtz instability. It is worth mentioning that the $\alpha $%
-regularization is ``closer" to the original system than the vortex-blob
method at the high wave numbers, due to the algebraic decay instead of
exponential one in the vortex blob method. This result was also evaluated
computationally in \cite{a_HNP06}.

\section{\label{sec:GlobalReg}Global regularity for BR-$\protect\alpha $
equation}

In this section we present the global existence and uniqueness of solutions
of the BR-$\alpha $ equation \eqref{eq:BR_alpha} in the space of Lipschitz
functions and in the H\"{o}lder space $C^{n,\beta }$, $n\geq 1$, which is
the space of $n$-times differentiable functions with H\"{o}lder continuous $%
n^{\text{th}}$ derivative.

Let us first describe the H\"{o}lder space $C^{n,\beta }\left( J\subset
\mathbb{R};\mathbb{R}^{2}\right) $, $0<\beta \leq 1$, which is the space of
functions $x:J\subset \mathbb{R}\rightarrow \mathbb{R}^{2}$, with a finite
norm
\begin{equation*}
\left\Vert x\right\Vert _{n,\beta \left( J\right) }=\sum_{k=0}^{n}\left\Vert
\frac{d^{k}}{d{\Gamma }^{k}}x\right\Vert _{C^{0}\left( J\right) }+\left\vert
\frac{d^{n}}{d{\Gamma }^{n}}x\right\vert _{\beta \left( J\right) },
\end{equation*}
where
\begin{equation*}
\left\Vert x\right\Vert _{C^{0}\left( J\right) }=\sup_{\Gamma \in
J}\left\vert x\left( \Gamma \right) \right\vert
\end{equation*}
and $\left\vert \cdot \right\vert _{\beta }$ is the H\"{o}lder semi-norm
\begin{equation*}
\left\vert x\right\vert _{\beta \left( J\right) }=\sup_{\substack{ \Gamma
,\Gamma ^{\prime }\in J  \\ \Gamma \neq \Gamma ^{\prime }}}\frac{\left\vert
x\left( \Gamma \right) -x\left( \Gamma ^{\prime }\right) \right\vert }{%
\left\vert \Gamma -\Gamma ^{\prime }\right\vert ^{\beta }}.
\end{equation*}
The Lipschitz space $\mathrm{Lip}\left( J\right) $ is the $C^{0,1}$space,
that is, with the finite norm $\left\Vert x\right\Vert _{\mathrm{Lip}\left(
J\right) }=\left\Vert x\right\Vert _{C^{0}\left( J\right) }+\left\vert
x\right\vert _{1\left( J\right) }$.

We also use the notation
\begin{equation*}
\left\vert x\right\vert _{\ast }=\inf \frac{\left\vert x\left( \Gamma
\right) -x\left( \Gamma ^{\prime }\right) \right\vert }{\left\vert \Gamma
-\Gamma ^{\prime }\right\vert },
\end{equation*}
where the infimum is taken over all $\Gamma ,\Gamma ^{\prime }\in J$ such
that $\Gamma \neq \Gamma ^{\prime }$, or, in the case of a closed curve
(without loss of generality, over $S^{1}$), the infimum is taken over all $%
\Gamma ,\Gamma ^{\prime }\in S^{1}$ such that $\Gamma \neq \Gamma ^{\prime }%
\mod 2\pi $.

We consider the BR-$\alpha $ equation as an evolution functional equation in
the Banach spaces $\mathrm{Lip}$, $C^{1}$ or $C^{n,\beta }$, $n\geq 1$, $%
0<\beta <1$,
\begin{equation}
\begin{split}
& \frac{\partial x}{\partial t}\left( \Gamma ,t\right) =\int_{\Gamma
_{0}}^{\Gamma _{1}}K^{\alpha }\left( x\left( \Gamma ,t\right) -x\left(
\Gamma ^{\prime },t\right) \right) d\Gamma ^{\prime }, \\
& x\left( \Gamma ,0\right) =x_{0}\left( \Gamma \right)
\end{split}
\label{grp:BR_alpha_onBanach}
\end{equation}%
with $\gamma =1/|x_{\Gamma }|$ being the vorticity density along the sheet
and $-\infty <\Gamma _{0}<\Gamma _{1}<\infty $. Notice that the density $%
\gamma \left( \Gamma \right) $ is in $C^{n-1,\beta }\left( \left( \Gamma
_{0},\Gamma _{1}\right) \right) $ for the subset $\left\{ \left\vert
x\right\vert _{\ast }>0\right\} $ of $C^{n,\beta }\left( \left( \Gamma
_{0},\Gamma _{1}\right) \right) $, and $\gamma \left( \Gamma \right) \in
L^{\infty }\left( \left( \Gamma _{0},\Gamma _{1}\right) \right) $ for the
subset $\left\{ \left\vert x\right\vert _{\ast }>0\right\} $ of $\mathrm{Lip}%
\left( \left( \Gamma _{0},\Gamma _{1}\right) \right) $.

\begin{theorem}
\label{thm:GlobalEx} Let $-\infty <\Gamma _{0}<\Gamma _{1}<\infty $,
$V$ be either the space $C^{1,\beta }\left( \left( \Gamma
_{0},\Gamma _{1}\right) \right) $, $0\leq \beta <1$ or the space
$\mathrm{Lip}\left( \left( \Gamma _{0},\Gamma _{1}\right) \right) $
and let $x_{0}\in V\cap \left\{ \left\vert x\right\vert _{\ast
}>0\right\} $. Then for any $T>0$ there is a unique solution {$x\in
C^{1}\left( [-T,T];V\cap \left\{
\left\vert x\right\vert _{\ast }>0\right\} \right) $} of %
\eqref{grp:BR_alpha_onBanach} with initial value $x\left( \Gamma ,0\right)
=x_{0}\left( \Gamma \right) $.

Furthermore, let $x_{0}$ be a closed curve and without loss of
generality, we assume $x_{0}\left( \Gamma \right) \in C^{n,\beta
}\left( S^{1}\right) \cap \left\{ \left\vert x\right\vert _{\ast
}>0\right\} $, then for all $n\geq 1$, $0<\beta <1$, $T>0$ there is
a unique solution {$x\in C^{1}\left( [-T,T];C^{n,\beta }\left(
S^{1}\right) \cap \left\{ \left\vert x\right\vert _{\ast }>0\right\}
\right) $} of \eqref{eq:BR_alpha} with initial value $x\left( \Gamma
,0\right) =x_{0}\left( \Gamma \right) $. In particular, if $x_{0}\in
C^{\infty }\left( S^{1}\right) \cap \left\{ \left\vert x\right\vert
_{\ast }>0\right\} $ then \newline $x\in C^{1}\left(
[-T,T];C^{\infty }\left( S^{1}\right) \cap \left\{ \left\vert
x\right\vert _{\ast }>0\right\} \right) $.
\end{theorem}

Notice that for $n>1$ we request $\beta $ to be strictly larger than zero
and the curve to be closed.

We remark that, although the kernel $K^{\alpha }$ is a continuous bounded
function, its derivatives are unbounded near the origin, and the chord arc
condition $|x|_{\ast }>0$, which implies simple curves, allows us to show
the integrability of the relevant terms.

The following are the main steps involved in the proof of Theorem \ref%
{thm:GlobalEx}. In the first step, we apply the contraction mapping
principle to the BR-$\alpha $ equation \eqref{eq:BR_alpha} to prove
the short time existence and uniqueness of solutions in the
appropriate space of functions. Specifically, we show that an
initially Lipschitz or $C^{1,\beta },0\leq \beta <1$ smooth
solutions of \eqref{eq:BR_alpha} remain, respectively, Lipschitz or
$C^{1,\beta }$ smooth for a finite short time. Next, we derive an
\textit{a priori} bound for the controlling quantity for
continuing the solution for all time. At step three we extend the $%
C^{1,\beta },0<\beta <1$, result for higher derivatives for the case
of a closed curve.

\subsection{Step 1. Local well-posedness.}

First we show the local existence and uniqueness of solutions. To
apply the contraction mapping principle to the BR-$\alpha $ equation
\eqref{grp:BR_alpha_onBanach}, we first prove the following result

\begin{proposition}
\label{prop:br_alpha_Lip_map}Let $-\infty <\Gamma _{0}<\Gamma _{1}<\infty $,
$1<M<\infty $, $V$ be either the space $C^{1,\beta }\left( \left( \Gamma
_{0},\Gamma _{1}\right) \right) $, $0\leq \beta <1$ or the space $\mathrm{Lip%
}\left( \left( \Gamma _{0},\Gamma _{1}\right) \right) $, and let $K^{M}$ be
the set
\begin{equation*}
K^{M}=\left\{ x\in V:\left\vert x\right\vert _{1}<M,\left\vert x\right\vert
_{\ast }>\frac{1}{M}\right\} .
\end{equation*}%
Then the mapping
\begin{equation}
x\left( \Gamma \right) \mapsto u\left( x\left( \Gamma \right) \right)
=\int_{\Gamma _{0}}^{\Gamma _{1}}K^{\alpha }\left( x\left( \Gamma \right)
-x\left( \Gamma ^{\prime }\right) \right) d\Gamma ^{\prime }
\label{eq:vel_map}
\end{equation}%
defines a locally Lipschitz continuous map from $K^{M}$, equipped with the
topology induced by the $\left\Vert \cdot \right\Vert _{V}$ norm, into $V$.
\end{proposition}

\begin{proof}
Notice that $K^{M}$ is an open set in $V$. We recall that $K^{\alpha }\left(
x\right) =\nabla ^{\perp }\Psi ^{\alpha }\left( \left\vert x\right\vert
\right) =\frac{x^{\perp }}{\left\vert x\right\vert }D\Psi ^{\alpha }\left(
\left\vert x\right\vert \right) ,$ where $\Psi ^{\alpha }\left( r\right) =%
\frac{1}{2\pi }\left[ K_{0}\left( \frac{r}{\alpha }\right) +\log r\right] $
and $D\Psi ^{\alpha }(r)=\frac{d\Psi ^{\alpha }}{dr}(r)=\frac{1}{2\pi }\left[
-\frac{1}{\alpha }K_{1}\left( \frac{r}{\alpha }\right) +\frac{1}{r}\right] $%
. The functions $K_{0}$ and $K_{1}$ denote the modified Bessel functions of
the second kind of orders zero and one, respectively. For details on Bessel
functions, see, e.g., \cite{b_W44}. We observe that $D\Psi ^{\alpha }$ is
bounded
\begin{equation}
D\Psi ^{\alpha }\left( r\right) \leq \frac{C}{\alpha },
\label{eq:Dpsi_bounded}
\end{equation}%
derivatives of $\Psi ^{\alpha }$ decay to zero as $\frac{r}{\alpha }%
\rightarrow \infty $, and as $\frac{r}{\alpha }\rightarrow 0$ satisfy%
\begin{align}
D\Psi ^{\alpha }\left( r\right) & =-\frac{1}{4\pi }\frac{r}{\alpha ^{2}}\log
\frac{r}{\alpha }+O\left( \frac{r}{\alpha ^{2}}\right) ,
\label{eq:Der_psi_at_origin} \\
D^{2}\Psi ^{\alpha }\left( r\right) & =-\frac{1}{4\pi }\frac{1}{\alpha ^{2}}%
\log \frac{r}{\alpha }+O\left( \frac{1}{\alpha ^{2}}\right) ,  \notag \\
D^{3}\Psi ^{\alpha }\left( r\right) & =-\frac{1}{4\pi }\frac{1}{r\alpha ^{2}}%
+O\left( \frac{r}{\alpha ^{4}}\log \frac{r}{\alpha }\right) .  \notag
\end{align}%
The constant $C$ will denote a generic constant independent of the
parameters, while, $C(\diamond )$ denotes a constant which depends on $%
\diamond $.

First we show the local existence and uniqueness of solutions in $C^{1,\beta
}$, $0<\beta <1$.

We start by showing that $u\left( x\left( \Gamma \right) \right) $ maps $%
K^{M}$ into $C^{1,\beta }$. Let $x\in K^{M}$. Using the boundness of $D\Psi
^{\alpha }$ \eqref{eq:Dpsi_bounded} we have
\begin{equation}
\left\vert u\left( x\left( \Gamma \right) \right) \right\vert \leq
\int_{\Gamma _{0}}^{\Gamma _{1}}D\Psi ^{\alpha }\left( \left\vert x\left(
\Gamma \right) -x\left( \Gamma ^{\prime }\right) \right\vert \right) d\Gamma
^{\prime }\leq \frac{C}{\alpha }\left( \Gamma _{1}-\Gamma _{0}\right) .
\label{eq:Ux_C0_bnd}
\end{equation}%
Using that
\begin{equation*}
\frac{du}{d\Gamma }\left( x\left( \Gamma \right) \right) =\int_{\Gamma
_{0}}^{\Gamma _{1}}\nabla K^{\alpha }\left( x\left( \Gamma \right) -x\left(
\Gamma ^{\prime }\right) \right) \frac{dx}{d\Gamma }\left( \Gamma \right)
d\Gamma ^{\prime },
\end{equation*}%
(which can be justified by applying Lebesgue dominated convergence theorem)
and the fact that $\left\vert \frac{dx}{d\Gamma }\left( \Gamma \right)
\right\vert <M$, we obtain
\begin{align*}
\left\vert \frac{du}{d\Gamma }\left( x\left( \Gamma \right) \right)
\right\vert & \leq M\int_{\Gamma _{0}}^{\Gamma _{1}}\left\vert \nabla
K^{\alpha }\left( x\left( \Gamma \right) -x\left( \Gamma ^{\prime }\right)
\right) \right\vert d\Gamma ^{\prime }, \\
& =M\left( \int_{\left( \Gamma _{0},\Gamma _{1}\right) \cap \left\{ \frac{%
\left\vert \Gamma -\Gamma ^{\prime }\right\vert }{\alpha }<\varepsilon
\right\} }+\int_{\left( \Gamma _{0},\Gamma _{1}\right) \backslash \left\{
\frac{\left\vert \Gamma -\Gamma ^{\prime }\right\vert }{\alpha }<\varepsilon
\right\} }\right) , \\
& =M\left( I_{1}+I_{2}\right) ,
\end{align*}%
where $\varepsilon $ is to be defined later. Due to \eqref{eq:K_alpha}, %
\eqref{eq:Der_psi_at_origin} and
\begin{equation}
\frac{1}{M}\frac{\left\vert \Gamma -\Gamma ^{\prime }\right\vert }{\alpha }%
<\left\vert x\right\vert _{\ast }\frac{\left\vert \Gamma -\Gamma ^{\prime
}\right\vert }{\alpha }\leq \frac{\left\vert x\left( \Gamma \right) -x\left(
\Gamma ^{\prime }\right) \right\vert }{\alpha }\leq \left\Vert \frac{dx}{%
d\Gamma }\right\Vert _{C^{0}}\frac{\left\vert \Gamma -\Gamma ^{\prime
}\right\vert }{\alpha }\leq M\varepsilon ,  \label{eq:inequality_in_SM}
\end{equation}%
we have that for a fixed small $\varepsilon $
\begin{align*}
I_{1}& \leq \int_{\frac{\left\vert \Gamma -\Gamma ^{\prime }\right\vert }{%
\alpha }<\varepsilon }\left( \frac{1}{4\pi \alpha ^{2}}\left\vert \log
\left( C\left( M\right) \frac{\left\vert \Gamma -\Gamma ^{\prime
}\right\vert }{\alpha }\right) \right\vert +C\left( M\right) \frac{1}{\alpha
^{2}}\right) d\Gamma ^{\prime } \\
& =C\left( M\right) \frac{1}{\alpha }.
\end{align*}%
For $I_{2}$ due to the boundness of $\left\vert \nabla K^{\alpha
}\right\vert $ in $\left( \Gamma _{0},\Gamma _{1}\right) \backslash \left\{
\frac{\left\vert \Gamma -\Gamma ^{\prime }\right\vert }{\alpha }<\varepsilon
\right\} $ we obtain%
\begin{equation*}
I_{2}\leq \frac{C\left( M\right) }{\alpha ^{2}}\left( \Gamma _{1}-\Gamma
_{0}\right) .
\end{equation*}%
Summing up,
\begin{equation}
\int_{\Gamma _{0}}^{\Gamma _{1}}\left\vert \nabla K^{\alpha }\left( x\left(
\Gamma \right) -x\left( \Gamma ^{\prime }\right) \right) \right\vert d\Gamma
^{\prime }\leq C\left( M,\Gamma _{1},\Gamma _{0},\frac{1}{\alpha }\right)
\label{eq:grad_K_alpha_bnd}
\end{equation}%
and hence%
\begin{equation}
\left\vert \frac{du}{d\Gamma }\left( x\left( \Gamma \right) \right)
\right\vert \leq C\left( M,\Gamma _{1},\Gamma _{0},\frac{1}{\alpha }\right) .
\label{eq:Ux_C1_bnd}
\end{equation}%
To show the H\"{o}lder continuity of $\frac{du}{d\Gamma }\left( x\left(
\Gamma \right) \right) $ we write
\begin{align*}
\left\vert \frac{du}{d\Gamma }\left( x\left( \Gamma \right) \right) -\frac{du%
}{d\Gamma }\left( x\left( \bar{\Gamma}\right) \right) \right\vert & \leq
\int_{\Gamma _{0}}^{\Gamma _{1}}\left\vert \nabla K^{\alpha }\left( x\left(
\Gamma \right) -x\left( \Gamma ^{\prime }\right) \right) \right\vert
\left\vert \frac{dx}{d\Gamma }\left( \Gamma \right) -\frac{dx}{d\Gamma }%
\left( \bar{\Gamma}\right) \right\vert d\Gamma ^{\prime } \\
& +\int_{\Gamma _{0}}^{\Gamma _{1}}\left\vert \nabla K^{\alpha }\left(
x\left( \Gamma \right) -x\left( \Gamma ^{\prime }\right) \right) -\nabla
K^{\alpha }\left( x\left( \bar{\Gamma}\right) -x\left( \Gamma ^{\prime
}\right) \right) \right\vert \left\vert \frac{dx}{d\Gamma }\left( \bar{\Gamma%
}\right) \right\vert d\Gamma ^{\prime } \\
& \leq \left\vert \frac{dx}{d\Gamma }\right\vert _{\beta }\left\vert \Gamma
-\Gamma ^{\prime }\right\vert ^{\beta }\int_{\Gamma _{0}}^{\Gamma
_{1}}\left\vert \nabla K^{\alpha }\left( x\left( \Gamma \right) -x\left(
\Gamma ^{\prime }\right) \right) \right\vert d\Gamma ^{\prime }+ \\
& +M\int_{\Gamma _{0}}^{\Gamma _{1}}\left\vert \nabla K^{\alpha }\left(
x\left( \Gamma \right) -x\left( \Gamma ^{\prime }\right) \right) -\nabla
K^{\alpha }\left( x\left( \bar{\Gamma}\right) -x\left( \Gamma ^{\prime
}\right) \right) \right\vert d\Gamma ^{\prime }.
\end{align*}%
The first term on the right-hand side is bounded by $C\left( M,\Gamma
_{1},\Gamma _{0},\frac{1}{\alpha }\right) \left\vert \frac{dx}{d\Gamma }%
\right\vert _{\beta }\left\vert \Gamma -\bar{\Gamma}\right\vert ^{\beta }$%
due to \eqref{eq:grad_K_alpha_bnd}, as for the second one, let $r=\frac{%
\left\vert \Gamma -\bar{\Gamma}\right\vert }{\alpha }$, and write%
\begin{align*}
I& =\int_{\Gamma _{0}}^{\Gamma _{1}}\left\vert \nabla K^{\alpha }\left(
x\left( \Gamma \right) -x\left( \Gamma ^{\prime }\right) \right) -\nabla
K^{\alpha }\left( x\left( \bar{\Gamma}\right) -x\left( \Gamma ^{\prime
}\right) \right) \right\vert d\Gamma ^{\prime } \\
& =\int_{\left( \Gamma _{0},\Gamma _{1}\right) \cap \left\{ \frac{\left\vert
\Gamma -\Gamma ^{\prime }\right\vert }{\alpha }<2r\right\} }+\int_{\left(
\Gamma _{0},\Gamma _{1}\right) \cap \left\{ \frac{\left\vert \Gamma -\Gamma
^{\prime }\right\vert }{\alpha }\geq 2r\right\} } \\
& =I_{1}+I_{2}.
\end{align*}%
For $I_{1}$ we have $\frac{\left\vert \Gamma -\Gamma ^{\prime }\right\vert }{%
\alpha }<2r$, and hence $\frac{\left\vert \bar{\Gamma}-\Gamma ^{\prime
}\right\vert }{\alpha }<3r$. Due to \eqref{eq:K_alpha}, the fact that
\begin{equation}
\left\vert D^{2}\Psi ^{\alpha }\left( s\right) \right\vert \leq \frac{1}{%
4\pi }\frac{1}{\alpha ^{2}}\left\vert \log \frac{s}{\alpha }\right\vert +%
\frac{C}{\alpha ^{2}}  \label{eq:D2psi_bound}
\end{equation}%
and \eqref{eq:inequality_in_SM} we obtain
\begin{align*}
I_{1}& \leq \int_{\left( \Gamma _{0},\Gamma _{1}\right) \cap \left\{ \frac{%
\left\vert \Gamma -\Gamma ^{\prime }\right\vert }{\alpha }<2r\right\}
}\left\vert \nabla K^{\alpha }\left( x\left( \Gamma \right) -x\left( \Gamma
^{\prime }\right) \right) \right\vert +\left\vert \nabla K^{\alpha }\left(
x\left( \bar{\Gamma}\right) -x\left( \Gamma ^{\prime }\right) \right)
\right\vert d\Gamma ^{\prime } \\
& \leq \frac{C}{\alpha ^{2}}\left( \int_{\frac{\left\vert \Gamma -\Gamma
^{\prime }\right\vert }{\alpha }<2r}\left\vert \log \frac{\left\vert x\left(
\Gamma \right) -x\left( \Gamma ^{\prime }\right) \right\vert }{\alpha }%
\right\vert d\Gamma ^{\prime }+\int_{\frac{\left\vert \bar{\Gamma}-\Gamma
^{\prime }\right\vert }{\alpha }<3r}\left\vert \log \frac{\left\vert x\left(
\bar{\Gamma}\right) -x\left( \Gamma ^{\prime }\right) \right\vert }{\alpha }%
\right\vert d\Gamma ^{\prime }+r\alpha \right) \\
& \leq \frac{C}{\alpha ^{2}}\left( \int_{\frac{\left\vert \Gamma -\Gamma
^{\prime }\right\vert }{\alpha }<2r}\left\vert \log C\left( M\right) \frac{%
\left\vert \Gamma -\Gamma ^{\prime }\right\vert }{\alpha }\right\vert
d\Gamma ^{\prime }+\int_{\frac{\left\vert \bar{\Gamma}-\Gamma ^{\prime
}\right\vert }{\alpha }<3r}\left\vert \log C\left( M\right) \frac{\left\vert
\bar{\Gamma}-\Gamma ^{\prime }\right\vert }{\alpha }\right\vert d\Gamma
^{\prime }+r\alpha \right) \\
& \leq C\left( M,\frac{1}{\alpha }\right) r\left( \left\vert \log
r\right\vert +1\right) .
\end{align*}%
For $I_{2}$ we have $\frac{\left\vert \Gamma -\Gamma ^{\prime }\right\vert }{%
\alpha }\geq 2r$, and hence $\frac{\left\vert \bar{\Gamma}-\Gamma ^{\prime
}\right\vert }{\alpha }\geq r$. By the mean value theorem (MVT), %
\eqref{eq:K_alpha} and the fact that
\begin{equation}
D^{3}\Psi ^{\alpha }\left( s\right) \leq \frac{1}{4\pi \alpha ^{2}}\frac{1}{s%
}+\frac{C}{\alpha ^{3}}  \label{eq:D3psi_bound}
\end{equation}%
we have that for $\Gamma ^{\prime \prime }\in \left[ \Gamma ,\bar{\Gamma}%
\right] $
\begin{align*}
\left\vert \nabla K^{\alpha }\left( x\left( \Gamma \right) -x\left( \Gamma
^{\prime }\right) \right) -\nabla K^{\alpha }\left( x\left( \bar{\Gamma}%
\right) -x\left( \Gamma ^{\prime }\right) \right) \right\vert & \leq r\alpha
\left\Vert \frac{dx}{d\Gamma }\right\Vert _{C^{0}}\left( \frac{C}{\alpha
^{2}\left\vert x\left( \Gamma ^{\prime \prime }\right) -x\left( \Gamma
^{\prime }\right) \right\vert }+\frac{C}{\alpha ^{3}}\right) \\
& \leq C\left( M,\frac{1}{\alpha }\right) r\left( \frac{1}{\left\vert \Gamma
^{\prime \prime }-\Gamma ^{\prime }\right\vert }+1\right) ,
\end{align*}%
we also have that $\frac{\left\vert \Gamma ^{\prime \prime }-\Gamma ^{\prime
}\right\vert }{\alpha }\geq r$ and $\Gamma _{0}\leq \Gamma ^{\prime \prime
}\leq \Gamma _{1}$. Hence
\begin{align*}
\left\vert I_{2}\right\vert & \leq r\frac{C\left( M\right) }{\alpha ^{3}}%
\int_{\left( \Gamma _{0},\Gamma _{1}\right) \cap \left\{ \frac{\left\vert
\Gamma ^{\prime \prime }-\Gamma ^{\prime }\right\vert }{\alpha }\geq
r\right\} }\left( \frac{\alpha }{\left\vert \Gamma ^{\prime \prime }-\Gamma
^{\prime }\right\vert }+1\right) d\Gamma ^{\prime } \\
& \leq C\left( M,\frac{1}{\alpha },\Gamma _{1},\Gamma _{0}\right) r\left(
1+\left\vert \log r\right\vert \right) .
\end{align*}%
Summing up we obtain%
\begin{equation}
\left\vert I\right\vert \leq C\left( M,\frac{1}{\alpha },\Gamma _{1},\Gamma
_{0}\right) \left\vert \bar{\Gamma}-\Gamma \right\vert \left( \left\vert
\log \left\vert \bar{\Gamma}-\Gamma \right\vert \right\vert +1\right) ,
\label{eq:grad_K_alpha_Holder}
\end{equation}%
which implies the H\"{o}lder continuity of $\frac{du}{d\Gamma }\left(
x\left( \Gamma \right) \right) $ for $0<\beta <1$.

It remains to show that $u\left( x\right) $ is locally Lipschitz continuous
on $K^{M}$. We will show that for $x\in K^{M}$, $y\in C^{1,\beta }\left(
\left( \Gamma _{0},\Gamma _{1}\right) \right) $%
\begin{equation*}
\left\Vert D_{x}u\left( x\right) y\right\Vert _{1,\beta }\leq C\left( \frac{1%
}{\alpha },M,\left\Vert x\right\Vert _{1,\beta },\Gamma _{1},\Gamma
_{0},\beta \right) \left\Vert y\right\Vert _{1,\beta }.
\end{equation*}%
Hence for any $x\in K^{M}$, let $\delta $ be such that $B\left( x,\delta
\right) \subset K^{M}$, then for every $\bar{x},\tilde{x}\in B\left(
x,\delta \right) $ by the fundamental theorem of calculus%
\begin{align*}
\left\Vert u\left( \bar{x}\right) -u\left( \tilde{x}\right) \right\Vert
_{1,\beta }& =\left\Vert \int_{0}^{1}\frac{d}{d\varepsilon }u\left( \bar{x}%
+\varepsilon \left( \tilde{x}-\bar{x}\right) \right) d\varepsilon
\right\Vert _{1,\beta } \\
& =\left\Vert \int_{0}^{1}D_{x}u\left( \bar{x}+\varepsilon \left( \tilde{x}-%
\bar{x}\right) \right) \left( \tilde{x}-\bar{x}\right) d\varepsilon
\right\Vert _{1,\beta } \\
& \leq \left\Vert \tilde{x}-\bar{x}\right\Vert _{1,\beta
}\int_{0}^{1}C\left( \frac{1}{\alpha },M,\left\Vert \bar{x}+\varepsilon
\left( \tilde{x}-\bar{x}\right) \right\Vert _{1,\beta },\Gamma _{1},\Gamma
_{0},\beta \right) d\varepsilon \\
& \leq C\left( \frac{1}{\alpha },M,\left\Vert x\right\Vert _{1,\beta
},\delta ,\Gamma _{1},\Gamma _{0},\beta \right) \left\Vert \tilde{x}-\bar{x}%
\right\Vert _{1,\beta },
\end{align*}%
that is, the map is locally Lipschitz.

Let $x\in K^{M}$, $y\in C^{1,\beta }\left( \left( \Gamma _{0},\Gamma
_{1}\right) \right) $, we now compute
\begin{align*}
D_{x}u\left( x\left( \Gamma \right) \right) y\left( \Gamma \right) & =\left.
\frac{d}{d\varepsilon }u\left( x\left( \Gamma \right) +\varepsilon y\left(
\Gamma \right) \right) \right\vert _{\varepsilon =0} \\
& =\left. \frac{d}{d\varepsilon }\int_{\Gamma _{0}}^{\Gamma _{1}}K^{\alpha
}\left( x\left( \Gamma \right) +\varepsilon y\left( \Gamma \right) -x\left(
\Gamma ^{\prime }\right) -\varepsilon y\left( \Gamma ^{\prime }\right)
\right) d\Gamma ^{\prime }\right\vert _{\varepsilon =0} \\
& =\int_{\Gamma _{0}}^{\Gamma _{1}}\nabla K^{\alpha }\left( x\left( \Gamma
\right) -x\left( \Gamma ^{\prime }\right) \right) \left( y\left( \Gamma
\right) -y\left( \Gamma ^{\prime }\right) \right) d\Gamma ^{\prime }.
\end{align*}%
Now, we show that
\begin{equation*}
\left\Vert D_{x}u\left( x\left( \cdot \right) \right) y\left( \cdot \right)
\right\Vert _{1,\beta }\leq C\left( \frac{1}{\alpha },M,\left\Vert
x\right\Vert _{1,\beta },\Gamma _{1},\Gamma _{0},\beta \right) \left\Vert
y\right\Vert _{1,\beta }.
\end{equation*}%
To estimate $\left\Vert D_{x}u\left( x\right) y\right\Vert _{C^{0}}$ we use %
\eqref{eq:grad_K_alpha_bnd}
\begin{equation}
\left\vert D_{x}u\left( x\right) y\right\vert \leq C\left( M,\Gamma
_{1},\Gamma _{0},\frac{1}{\alpha }\right) \left\Vert y\right\Vert _{C^{0}}.
\label{eq:DxUy_C0_bnd}
\end{equation}%
Next we estimate $\left\Vert \frac{d}{d\Gamma }D_{x}u\left( x\right)
y\right\Vert _{C^{0}}$. For $\Gamma ^{\prime }\neq \Gamma ,$ $\nabla
K^{\alpha }\left( x\left( \Gamma \right) -x\left( \Gamma ^{\prime
}\right) \right) \left( y\left( \Gamma \right) -y\left( \Gamma
^{\prime }\right) \right) $ is differentiable in $\Gamma $, hence
(can be justified by
using Lebesgue dominated convergence theorem)%
\begin{align*}
\frac{d}{d\Gamma }D_{x}u\left( x\left( \Gamma \right) \right) y\left( \Gamma
\right) & =\int_{\Gamma _{0}}^{\Gamma _{1}}\nabla K^{\alpha }\left( x\left(
\Gamma \right) -x\left( \Gamma ^{\prime }\right) \right) \frac{dy}{d\Gamma }%
\left( \Gamma \right) d\Gamma ^{\prime } \\
& +\int_{\Gamma _{0}}^{\Gamma _{1}}\sum_{i,j=1}^{2}\partial _{x_{i}}\partial
_{x_{j}}K^{\alpha }\left( x\left( \Gamma \right) -x\left( \Gamma ^{\prime
}\right) \right) \frac{dx_{i}}{d\Gamma }\left( \Gamma \right) \left(
y_{j}\left( \Gamma \right) -y_{j}\left( \Gamma ^{\prime }\right) \right)
d\Gamma ^{\prime }.
\end{align*}%
Notice, that although, for $\Gamma ^{\prime }$ close to $\Gamma $, $%
\left\vert D^{2}K^{\alpha }\left( x\left( \Gamma \right) -x\left( \Gamma
^{\prime }\right) \right) \right\vert =O\left( \frac{1}{\alpha
^{2}\left\vert x\left( \Gamma \right) -x\left( \Gamma ^{\prime }\right)
\right\vert }\right) $ (see \eqref{eq:K_alpha} and %
\eqref{eq:Der_psi_at_origin}), the term $\left( y\left( \Gamma \right)
-y\left( \Gamma ^{\prime }\right) \right) $ cancels the singularity in $%
\frac{1}{x\left( \Gamma \right) -x\left( \Gamma ^{\prime }\right) }$, so
this is not a singular integral.

We have
\begin{align*}
\left\vert \frac{d}{d\Gamma }D_{x}u\left( x\left( \Gamma \right) \right)
y\left( \Gamma \right) \right\vert & \leq \left\Vert \frac{dy}{d\Gamma }%
\right\Vert _{C^{0}}\int_{\Gamma _{0}}^{\Gamma _{1}}\left\vert \nabla
K^{\alpha }\left( x\left( \Gamma \right) -x\left( \Gamma ^{\prime }\right)
\right) \right\vert d\Gamma ^{\prime } \\
& +\left\Vert \frac{dx}{d\Gamma }\right\Vert _{C^{0}}\int_{\Gamma
_{0}}^{\Gamma _{1}}\left\vert D^{2}K^{\alpha }\left( x\left( \Gamma \right)
-x\left( \Gamma ^{\prime }\right) \right) \right\vert \left\vert y\left(
\Gamma \right) -y\left( \Gamma ^{\prime }\right) \right\vert d\Gamma
^{\prime }.
\end{align*}%
Write the second integral on the right-hand side as%
\begin{align*}
\int_{\Gamma _{0}}^{\Gamma _{1}}\left\vert D^{2}K^{\alpha }\left( x\left(
\Gamma \right) -x\left( \Gamma ^{\prime }\right) \right) \right\vert
\left\vert y\left( \Gamma \right) -y\left( \Gamma ^{\prime }\right)
\right\vert d\Gamma ^{\prime }& =\int_{\left( \Gamma _{0},\Gamma _{1}\right)
\cap \left\{ \frac{\left\vert \Gamma -\Gamma ^{\prime }\right\vert }{\alpha }%
<\varepsilon \right\} }+\int_{\left( \Gamma _{0},\Gamma _{1}\right)
\backslash \left\{ \frac{\left\vert \Gamma -\Gamma ^{\prime }\right\vert }{%
\alpha }<\varepsilon \right\} } \\
& =I_{1}+I_{2},
\end{align*}%
for a fixed small $\varepsilon $. Then due to \eqref{eq:K_alpha}, %
\eqref{eq:Der_psi_at_origin} and \eqref{eq:inequality_in_SM}, we obtain%
\begin{align*}
I_{1}& \leq C\frac{1}{\alpha ^{2}}\left\Vert \frac{dy}{d\Gamma }\right\Vert
_{C^{0}}\int_{\frac{\left\vert \Gamma -\Gamma ^{\prime }\right\vert }{\alpha
}<\varepsilon }\frac{M}{\left\vert \Gamma -\Gamma ^{\prime }\right\vert }%
\left\vert \Gamma -\Gamma ^{\prime }\right\vert d\Gamma ^{\prime } \\
& \leq CM\frac{1}{\alpha }\left\Vert \frac{dy}{d\Gamma }\right\Vert _{C^{0}}.
\end{align*}%
For $I_{2}$ we have%
\begin{align*}
I_{2}& \leq C\left( M,\frac{1}{\alpha }\right) \left\Vert \frac{dy}{d\Gamma }%
\right\Vert _{C^{0}}\int_{\left( \Gamma _{0},\Gamma _{1}\right) \cap \left\{
\left\vert \Gamma -\Gamma ^{\prime }\right\vert \geq \varepsilon \alpha
\right\} }\left\vert \Gamma -\Gamma ^{\prime }\right\vert d\Gamma ^{\prime }
\\
& \leq C\left( M,\Gamma _{1},\Gamma _{0},\frac{1}{\alpha }\right) \left\Vert
\frac{dy}{d\Gamma }\right\Vert _{C^{0}}.
\end{align*}%
Hence%
\begin{equation}
\left\vert \frac{d}{d\Gamma }D_{x}u\left( x\left( \Gamma \right) \right)
y\left( \Gamma \right) \right\vert \leq \left\Vert y\right\Vert
_{C^{1}}C\left( M,\Gamma _{1},\Gamma _{0},\frac{1}{\alpha }\right) .
\label{eq:DxUy_C1_bnd}
\end{equation}%
It remains to estimate $\left\vert \frac{d}{d\Gamma }D_{x}u\left( x\right)
y\right\vert _{\beta }$.
\begin{align*}
\frac{d}{d\Gamma }D_{x}u\left( x\left( \Gamma \right) \right) y\left( \Gamma
\right) -\frac{d}{d\Gamma }D_{x}u\left( x\left( \bar{\Gamma}\right) \right)
y\left( \bar{\Gamma}\right) & =\int_{\Gamma _{0}}^{\Gamma _{1}}\left( \nabla
K^{\alpha }\left( x\left( \Gamma \right) -x\left( \Gamma ^{\prime }\right)
\right) \frac{dy}{d\Gamma }\left( \Gamma \right) -\nabla K^{\alpha }\left(
x\left( \bar{\Gamma}\right) -x\left( \Gamma ^{\prime }\right) \right) \frac{%
dy}{d\Gamma }\left( \bar{\Gamma}\right) \right) d\Gamma ^{\prime } \\
& +\int_{\Gamma _{0}}^{\Gamma _{1}}\left[ \sum_{i,j}\partial
_{x_{i}}\partial _{x_{j}}K^{\alpha }\left( x\left( \Gamma \right) -x\left(
\Gamma ^{\prime }\right) \right) \frac{dx_{i}}{d\Gamma }\left( \Gamma
\right) \left( y_{j}\left( \Gamma \right) -y_{j}\left( \Gamma ^{\prime
}\right) \right) -\right. \\
& \left. -\sum_{i,j}\partial _{x_{i}}\partial _{x_{j}}K^{\alpha }\left(
x\left( \bar{\Gamma}\right) -x\left( \Gamma ^{\prime }\right) \right) \frac{%
dx_{i}}{d\Gamma }\left( \bar{\Gamma}\right) \left( y_{j}\left( \bar{\Gamma}%
\right) -y_{j}\left( \Gamma ^{\prime }\right) \right) \right] d\Gamma
^{\prime } \\
& =I_{1}+I_{2}.
\end{align*}%
We write $I_{1}$ as
\begin{align*}
\left\vert I_{1}\right\vert & \leq \int_{\Gamma _{0}}^{\Gamma
_{1}}\left\vert \nabla K^{\alpha }\left( x\left( \Gamma \right) -x\left(
\Gamma ^{\prime }\right) \right) \right\vert \left\vert \frac{dy}{d\Gamma }%
\left( \Gamma \right) -\frac{dy}{d\Gamma }\left( \bar{\Gamma}\right)
\right\vert d\Gamma ^{\prime } \\
& +\int_{\Gamma _{0}}^{\Gamma _{1}}\left\vert \nabla K^{\alpha }\left(
x\left( \Gamma \right) -x\left( \Gamma ^{\prime }\right) \right) -\nabla
K^{\alpha }\left( x\left( \bar{\Gamma}\right) -x\left( \Gamma ^{\prime
}\right) \right) \right\vert \left\vert \frac{dy}{d\Gamma }\left( \bar{\Gamma%
}\right) \right\vert d\Gamma ^{\prime } \\
& \leq I_{11}+I_{12}.
\end{align*}%
Using \eqref{eq:grad_K_alpha_bnd} to bound $I_{11}$ and %
\eqref{eq:grad_K_alpha_Holder} to bound $I_{12}$ we obtain that
\begin{equation*}
\left\vert I_{1}\right\vert \leq C\left( M,\frac{1}{\alpha },\Gamma
_{1},\Gamma _{0}\right) \left\vert \bar{\Gamma}-\Gamma \right\vert ^{\beta
}\left\Vert y\right\Vert _{1,\beta }.
\end{equation*}%
Now we estimate $I_{2}$. Let $r=\frac{\left\vert \Gamma -\bar{\Gamma}%
\right\vert }{\alpha }$, write $I_{2}$ as
\begin{align*}
I_{2}& =\int_{\left( \Gamma _{0},\Gamma _{1}\right) \cap \left\{ \frac{%
\left\vert \Gamma -\Gamma ^{\prime }\right\vert }{\alpha }<2r\right\}
}+\int_{\left( \Gamma _{0},\Gamma _{1}\right) \cap \left\{ \frac{\left\vert
\Gamma -\Gamma ^{\prime }\right\vert }{\alpha }\geq 2r\right\} } \\
& =I_{21}+I_{22}.
\end{align*}%
Using \eqref{eq:K_alpha} and \eqref{eq:D3psi_bound} for $I_{21}$ we have
\begin{align*}
\left\vert I_{21}\right\vert & \leq \int_{\left( \Gamma _{0},\Gamma
_{1}\right) \cap \frac{\left\vert \Gamma -\Gamma ^{\prime }\right\vert }{%
\alpha }<2r}\left\vert D^{2}K^{\alpha }\left( x\left( \Gamma \right)
-x\left( \Gamma ^{\prime }\right) \right) \right\vert \left\vert \frac{dx}{%
d\Gamma }\left( \Gamma \right) \right\vert \left\vert y\left( \Gamma \right)
-y\left( \Gamma ^{\prime }\right) \right\vert d\Gamma ^{\prime } \\
& +\int_{\left( \Gamma _{0},\Gamma _{1}\right) \cap \frac{\left\vert \bar{%
\Gamma}-\Gamma ^{\prime }\right\vert }{\alpha }<3r}\left\vert D^{2}K^{\alpha
}\left( x\left( \bar{\Gamma}\right) -x\left( \Gamma ^{\prime }\right)
\right) \right\vert \left\vert \frac{dx}{d\Gamma }\left( \bar{\Gamma}\right)
\right\vert \left\vert y\left( \bar{\Gamma}\right) -y\left( \Gamma ^{\prime
}\right) \right\vert d\Gamma ^{\prime } \\
& \leq C\left\Vert \frac{dx}{d\Gamma }\right\Vert _{C^{0}}\left\Vert \frac{dy%
}{d\Gamma }\right\Vert _{C^{0}}\int_{\frac{\left\vert \Gamma -\Gamma
^{\prime }\right\vert }{\alpha }<3r}\left( \frac{C}{\alpha ^{2}}\frac{M}{%
\left\vert \Gamma -\Gamma ^{\prime }\right\vert }+\frac{C}{\alpha ^{3}}%
\right) \left\vert \Gamma -\Gamma ^{\prime }\right\vert d\Gamma ^{\prime } \\
& \leq C\left( \frac{1}{\alpha },M,\Gamma _{1},\Gamma _{0}\right)
\left\Vert \frac{dx}{d\Gamma }\right\Vert _{C^{0}}\left\Vert
y\right\Vert _{1,\beta }\left\vert \Gamma -\bar{\Gamma}\right\vert
^{\beta }.
\end{align*}%
We write $I_{22}$ as
\begin{align*}
I_{22}& =\int_{\left( \Gamma _{0},\Gamma _{1}\right) \cap \left\{ \frac{%
\left\vert \Gamma -\Gamma ^{\prime }\right\vert }{\alpha }\geq 2r\right\}
}\sum_{i,j}\partial _{x_{i}}\partial _{x_{j}}K^{\alpha }\left( x\left(
\Gamma \right) -x\left( \Gamma ^{\prime }\right) \right) \frac{dx_{i}}{%
d\Gamma }\left( \Gamma \right) \left( y_{j}\left( \Gamma \right)
-y_{j}\left( \bar{\Gamma}\right) \right) d\Gamma ^{\prime } \\
& +\int_{\left( \Gamma _{0},\Gamma _{1}\right) \cap \left\{ \frac{\left\vert
\Gamma -\Gamma ^{\prime }\right\vert }{\alpha }\geq 2r\right\}
}\sum_{i,j}\partial _{x_{i}}\partial _{x_{j}}K^{\alpha }\left( x\left(
\Gamma \right) -x\left( \Gamma ^{\prime }\right) \right) \left( \frac{dx_{i}%
}{d\Gamma }\left( \Gamma \right) -\frac{dx_{i}}{d\Gamma }\left( \bar{\Gamma}%
\right) \right) \left( y_{j}\left( \bar{\Gamma}\right) -y_{j}\left( \Gamma
^{\prime }\right) \right) d\Gamma ^{\prime } \\
& +\int \sum_{i,j}\left( \partial _{x_{i}}\partial _{x_{j}}K^{\alpha }\left(
x\left( \Gamma \right) -x\left( \Gamma ^{\prime }\right) \right) -\partial
_{x_{i}}\partial _{x_{j}}K^{\alpha }\left( x\left( \bar{\Gamma}\right)
-x\left( \Gamma ^{\prime }\right) \right) \right) x_{_{i}\Gamma }\left( \bar{%
\Gamma}\right) \left( y_{j}\left( \bar{\Gamma}\right) -y_{j}\left( \Gamma
^{\prime }\right) \right) d\Gamma ^{\prime } \\
& =I_{221}+I_{222}+I_{223}.
\end{align*}%
For $I_{221}:$
\begin{align*}
\left\vert I_{221}\right\vert & \leq \left\Vert \frac{dx}{d\Gamma }%
\right\Vert _{C^{0}}\left\Vert \frac{dy}{d\Gamma }\right\Vert
_{C^{0}}\left\vert \Gamma -\bar{\Gamma}\right\vert \int_{\left( \Gamma
_{0},\Gamma _{1}\right) \cap \left\{ \frac{\left\vert \Gamma -\Gamma
^{\prime }\right\vert }{\alpha }\geq 2r\right\} }\left\vert D^{2}K^{\alpha
}\left( x\left( \Gamma \right) -x\left( \Gamma ^{\prime }\right) \right)
\right\vert d\Gamma ^{\prime } \\
& \leq \left\Vert \frac{dx}{d\Gamma }\right\Vert _{C^{0}}\left\Vert \frac{dy%
}{d\Gamma }\right\Vert _{C^{0}}\left\vert \Gamma -\bar{\Gamma}\right\vert
\int_{\left( \Gamma _{0},\Gamma _{1}\right) \cap \left\{ \frac{\left\vert
\Gamma -\Gamma ^{\prime }\right\vert }{\alpha }\geq 2r\right\} }\left( \frac{%
1}{\alpha ^{2}}\frac{M}{\left\vert \Gamma -\Gamma ^{\prime }\right\vert }+%
\frac{C}{\alpha ^{3}}\right) d\Gamma ^{\prime } \\
& \leq C\left( \frac{1}{\alpha },M,\Gamma _{1},\Gamma _{0},\left\Vert \frac{%
dx}{d\Gamma }\right\Vert _{C^{0}}\right) \left\Vert y\right\Vert _{1,\beta
}\left\vert \Gamma -\bar{\Gamma}\right\vert \left[ 1+\left\vert \log
\left\vert \Gamma -\bar{\Gamma}\right\vert \right\vert \right] ,
\end{align*}%
which implies H\"{o}lder continuity for $0<\beta <1$.

For $I_{222}:$
\begin{align*}
\left\vert I_{222}\right\vert & \leq \int_{\left( \Gamma _{0},\Gamma
_{1}\right) \cap \left\{ \frac{\left\vert \Gamma -\Gamma ^{\prime
}\right\vert }{\alpha }\geq 2r\right\} }\left\vert D^{2}K^{\alpha }\left(
x\left( \Gamma \right) -x\left( \Gamma ^{\prime }\right) \right) \right\vert
\left\vert \frac{dx}{d\Gamma }\left( \Gamma \right) -\frac{dx}{d\Gamma }%
\left( \bar{\Gamma}\right) \right\vert \left\vert y\left( \bar{\Gamma}%
\right) -y\left( \Gamma ^{\prime }\right) \right\vert d\Gamma ^{\prime } \\
& \leq C\left\vert \frac{dx}{d\Gamma }\right\vert _{\beta }\left\vert \Gamma
-\bar{\Gamma}\right\vert ^{\beta }\left\Vert \frac{dy}{d\Gamma }\right\Vert
_{C^{0}}\int_{\left( \Gamma _{0},\Gamma _{1}\right) \cap \left\{ \frac{%
\left\vert \Gamma -\Gamma ^{\prime }\right\vert }{\alpha }\geq 2r\right\}
}\left( \frac{1}{\alpha ^{2}}\frac{M}{\left\vert \Gamma -\Gamma ^{\prime
}\right\vert }+\frac{C}{\alpha ^{3}}\right) \left\vert \bar{\Gamma}-\Gamma
^{\prime }\right\vert d\Gamma ^{\prime } \\
& \leq C\left( \frac{1}{\alpha },M,\left\vert \frac{dx}{d\Gamma }\right\vert
_{\beta },\Gamma _{1},\Gamma _{0}\right) \left\vert \Gamma -\bar{\Gamma}%
\right\vert ^{\beta }\left\Vert y\right\Vert _{1,\beta },
\end{align*}%
here we also used that%
\begin{equation*}
\left\vert \bar{\Gamma}-\Gamma ^{\prime }\right\vert \leq \left\vert \bar{%
\Gamma}-\Gamma \right\vert +\left\vert \Gamma -\Gamma ^{\prime }\right\vert .
\end{equation*}%
For $I_{223}:$%
\begin{equation*}
\left\vert I_{223}\right\vert \leq \left\Vert \frac{dx}{d\Gamma }\right\Vert
_{C^{0}}\left\Vert \frac{dy}{d\Gamma }\right\Vert _{C^{0}}\int_{\left(
\Gamma _{0},\Gamma _{1}\right) \cap \left\{ \frac{\left\vert \Gamma -\Gamma
^{\prime }\right\vert }{\alpha }\geq 2r\right\} }\left\vert D^{2}K^{\alpha
}\left( x\left( \Gamma \right) -x\left( \Gamma ^{\prime }\right) \right)
-D^{2}K^{\alpha }\left( x\left( \bar{\Gamma}\right) -x\left( \Gamma ^{\prime
}\right) \right) \right\vert \left\vert \bar{\Gamma}-\Gamma ^{\prime
}\right\vert d\Gamma ^{\prime }
\end{equation*}%
Since $\frac{\left\vert \Gamma -\Gamma ^{\prime }\right\vert }{\alpha }\geq
2r$ and $\frac{\left\vert \bar{\Gamma}-\Gamma ^{\prime }\right\vert }{\alpha
}\geq r$, i.e., $D^{2}K^{\alpha }\left( x\left( \Gamma \right) -x\left(
\Gamma ^{\prime }\right) \right) -D^{2}K^{\alpha }\left( x\left( \bar{\Gamma}%
\right) -x\left( \Gamma ^{\prime }\right) \right) $ is differentiable in $%
\left[ \Gamma ,\bar{\Gamma}\right] $, we can apply the MVT to obtain that
for $\Gamma ^{\prime \prime }\in \left[ \Gamma ,\bar{\Gamma}\right] $
\begin{equation*}
\left\vert D^{2}K^{\alpha }\left( x\left( \Gamma \right) -x\left( \Gamma
^{\prime }\right) \right) -D^{2}K^{\alpha }\left( x\left( \bar{\Gamma}%
\right) -x\left( \Gamma ^{\prime }\right) \right) \right\vert =r\alpha \frac{%
C\left( M\right) }{\alpha ^{3}}\left( \frac{\alpha }{\left\vert \Gamma
^{\prime \prime }-\Gamma ^{\prime }\right\vert ^{2}}+1\right) .
\end{equation*}%
We also have that $\frac{\left\vert \Gamma ^{\prime \prime }-\Gamma ^{\prime
}\right\vert }{\alpha }\geq r$. Hence%
\begin{align*}
\left\vert I_{223}\right\vert & \leq C\left( \frac{1}{\alpha },M\right)
\left\Vert \frac{dx}{d\Gamma }\right\Vert _{C^{0}}\left\Vert \frac{dy}{%
d\Gamma }\right\Vert _{C^{0}}r\alpha \int_{\left( \Gamma _{0},\Gamma
_{1}\right) \cap \frac{\left\vert \Gamma ^{\prime \prime }-\Gamma ^{\prime
}\right\vert }{\alpha }\geq r}\left( \frac{\alpha }{\left\vert \Gamma
^{\prime \prime }-\Gamma ^{\prime }\right\vert ^{2}}+1\right) \left\vert
\bar{\Gamma}-\Gamma ^{\prime }\right\vert d\Gamma ^{\prime } \\
& \leq C\left( \frac{1}{\alpha },M,\left\Vert \frac{dx}{d\Gamma }\right\Vert
_{C^{0}},\Gamma _{1},\Gamma _{0}\right) \left\Vert y\right\Vert _{1,\beta
}\left\vert \bar{\Gamma}-\Gamma \right\vert \left( 1+\left\vert \log
\left\vert \bar{\Gamma}-\Gamma \right\vert \right\vert \right) ,
\end{align*}%
where we have also used that
\begin{equation*}
\left\vert \bar{\Gamma}-\Gamma ^{\prime }\right\vert \leq \left\vert \bar{%
\Gamma}-\Gamma ^{\prime \prime }\right\vert +\left\vert \Gamma ^{\prime
\prime }-\Gamma ^{\prime }\right\vert \leq \left\vert \bar{\Gamma}-\Gamma
\right\vert +\left\vert \Gamma ^{\prime \prime }-\Gamma ^{\prime
}\right\vert .
\end{equation*}%
This implies H\"{o}lder continuity for $0<\beta <1$.

Summing up we have%
\begin{equation*}
\left\vert \frac{d}{d\Gamma }D_{x}u\left( x\right) y\right\vert _{\beta
}\leq C\left( \frac{1}{\alpha },M,\left\Vert x\right\Vert _{1,\beta },\Gamma
_{1},\Gamma _{0},\beta \right) \left\Vert y\right\Vert _{1,\beta }.
\end{equation*}

The above proof deals with $C^{1,\beta }$ case, for $0<\beta <1$. To show
the local existence and uniqueness of solutions in $C^{1}$ case, one simply
applies directly the above proof (estimates \eqref{eq:Ux_C0_bnd},%
\eqref{eq:Ux_C1_bnd},\eqref{eq:DxUy_C0_bnd},\eqref{eq:DxUy_C1_bnd}) without
resorting to the H\"{o}lder estimates.

The proof of the Lipschitz case is similar to the proof of the $C^{1,\beta }$
case, for example, to show Lipschitz continuity of $u\left( x\left( \Gamma
\right) \right) $ for $x\in \mathrm{Lip}\left( \left( \Gamma _{0},\Gamma
_{1}\right) \right) $, denote $r=\frac{\left\vert \Gamma -\bar{\Gamma}%
\right\vert }{\alpha }$ and write%
\begin{align*}
\left\vert u\left( x\left( \Gamma \right) \right) -u\left( x\left( \bar{%
\Gamma}\right) \right) \right\vert & \leq \int_{\Gamma _{0}}^{\Gamma
_{1}}\left\vert K^{\alpha }\left( x\left( \Gamma \right) -x\left( \Gamma
^{\prime }\right) \right) -K^{\alpha }\left( x\left( \bar{\Gamma}\right)
-x\left( \Gamma ^{\prime }\right) \right) \right\vert d\Gamma ^{\prime } \\
& =\int_{\left( \Gamma _{0},\Gamma _{1}\right) \cap E_{r}}+\int_{\left(
\Gamma _{0},\Gamma _{1}\right) \backslash E_{r}}=I_{1}+I_{2},
\end{align*}%
where
\begin{equation*}
E_{r}=\left\{ \Gamma ^{\prime }\in \left( \Gamma _{0},\Gamma _{1}\right) :%
\frac{\left\vert x\left( \Gamma \right) -x\left( \Gamma ^{\prime }\right)
\right\vert }{\alpha }<2rM\right\} .
\end{equation*}%
For $I_{1},$ due to $\frac{1}{M}\frac{\left\vert \Gamma -\Gamma ^{\prime
}\right\vert }{\alpha }<\left\vert x\right\vert _{\ast }\frac{\left\vert
\Gamma -\Gamma ^{\prime }\right\vert }{\alpha }\leq \frac{\left\vert x\left(
\Gamma \right) -x\left( \Gamma ^{\prime }\right) \right\vert }{\alpha },$ we
have that $\frac{\left\vert \Gamma -\Gamma ^{\prime }\right\vert }{\alpha }%
<2rM^{2}$ and hence $\frac{\left\vert \bar{\Gamma}-\Gamma ^{\prime
}\right\vert }{\alpha }<r\left( 1+2M^{2}\right) $. Thus by \eqref{eq:K_alpha}
and \eqref{eq:Dpsi_bounded} we obtain
\begin{align*}
I_{1}& \leq \int_{\left( \Gamma _{0},\Gamma _{1}\right) \cap \left\{ \frac{%
\left\vert \Gamma -\Gamma ^{\prime }\right\vert }{\alpha }<2rM^{2}\right\}
}\left\vert K^{\alpha }\left( x\left( \Gamma \right) -x\left( \Gamma
^{\prime }\right) \right) \right\vert +\left\vert K^{\alpha }\left( x\left(
\bar{\Gamma}\right) -x\left( \Gamma ^{\prime }\right) \right) \right\vert
d\Gamma ^{\prime } \\
& \leq \frac{C}{\alpha }\left( \int_{\frac{\left\vert \Gamma -\Gamma
^{\prime }\right\vert }{\alpha }<2rM^{2}}d\Gamma ^{\prime }+\int_{\frac{%
\left\vert \bar{\Gamma}-\Gamma ^{\prime }\right\vert }{\alpha }<r\left(
1+2M^{2}\right) }d\Gamma ^{\prime }\right) \leq C\left( M\right) r.
\end{align*}%
For $I_{2}$ due to $M\frac{\left\vert \Gamma -\Gamma ^{\prime }\right\vert }{%
\alpha }\geq \frac{\left\vert x\left( \Gamma \right) -x\left( \Gamma
^{\prime }\right) \right\vert }{\alpha }\geq 2rM$, we have that $\frac{%
\left\vert \Gamma -\Gamma ^{\prime }\right\vert }{\alpha }\geq 2r$, and
hence $\frac{\left\vert \bar{\Gamma}-\Gamma ^{\prime }\right\vert }{\alpha }%
\geq r$, which in turn implies that $\frac{\left\vert x\left( \bar{\Gamma}%
\right) -x\left( \Gamma ^{\prime }\right) \right\vert }{\alpha }>\frac{1}{M}%
\frac{\left\vert \bar{\Gamma}-\Gamma ^{\prime }\right\vert }{\alpha }\geq
\frac{r}{M}$. Also, due to $\left\vert x\left( \Gamma \right) -x\left( \bar{%
\Gamma}\right) \right\vert \leq M\left\vert \Gamma -\bar{\Gamma}\right\vert
=Mr\alpha $ we have for every $x\left( \Gamma ^{\prime \prime }\right) \in
B\left( x\left( \Gamma \right) ,\left\vert x\left( \Gamma \right) -x\left(
\bar{\Gamma}\right) \right\vert \right) $ that $\frac{\left\vert x\left(
\Gamma ^{\prime \prime }\right) -x\left( \Gamma ^{\prime }\right)
\right\vert }{\alpha }\geq Mr$. Hence by the mean value theorem and %
\eqref{eq:D2psi_bound}, we have that for $x\left( \Gamma ^{\prime \prime
}\right) \in B\left( x\left( \Gamma \right) ,\left\vert x\left( \Gamma
\right) -x\left( \bar{\Gamma}\right) \right\vert \right) $
\begin{align*}
\left\vert K^{\alpha }\left( x\left( \Gamma \right) -x\left( \Gamma ^{\prime
}\right) \right) -K^{\alpha }\left( x\left( \bar{\Gamma}\right) -x\left(
\Gamma ^{\prime }\right) \right) \right\vert & \leq \left\vert \nabla
K^{\alpha }\left( x\left( \Gamma ^{\prime \prime }\right) -x\left( \Gamma
^{\prime }\right) \right) \right\vert \left\vert x\left( \Gamma \right)
-x\left( \bar{\Gamma}\right) \right\vert \\
& \leq \left( \frac{1}{4\pi }\frac{1}{\alpha ^{2}}\left\vert \log \frac{%
\left\vert x\left( \Gamma ^{\prime \prime }\right) -x\left( \Gamma ^{\prime
}\right) \right\vert }{\alpha }\right\vert +\frac{C}{\alpha ^{2}}\right)
\left\vert x\left( \Gamma \right) -x\left( \bar{\Gamma}\right) \right\vert \\
& \leq C\left( M,\frac{1}{\alpha }\right) \left\vert \Gamma -\bar{\Gamma}%
\right\vert \left( \log \left( \frac{\left\vert \Gamma ^{\prime \prime
}-\Gamma ^{\prime }\right\vert }{\alpha }\right) +1\right)
\end{align*}%
Hence
\begin{align*}
\left\vert I_{2}\right\vert & \leq rC\left( M,\frac{1}{\alpha }\right)
\int_{\left( \Gamma _{0},\Gamma _{1}\right) \cap \left\{ \frac{\left\vert
\Gamma ^{\prime \prime }-\Gamma ^{\prime }\right\vert }{\alpha }\geq
r\right\} }\left( \log \left( \frac{\left\vert \Gamma ^{\prime \prime
}-\Gamma ^{\prime }\right\vert }{\alpha }\right) +1\right) d\Gamma ^{\prime }
\\
& \leq C\left( M,\frac{1}{\alpha },\Gamma _{1},\Gamma _{0}\right) r\left(
1+r\left\vert \log r\right\vert \right) .
\end{align*}%
Hence $u\left( x\left( \Gamma \right) \right) $ is Lipschitz continuous. We
remark, that in the proof of the $C^{1,\beta }$ part we used partitions
using the fact that $x\left( \Gamma \right) $ is a differentiable, however,
given the fact that differentiable functions are Lipschitz, one could have
used the partitioning introduced in the proof of Lipschitz case on subsets
of $x\left( \Gamma \right) $ also for $C^{1,\beta }$ results.
\end{proof}

Proposition \ref{prop:br_alpha_Lip_map} implies the local existence and
uniqueness of solutions:

\begin{proposition}
Let $-\infty <\Gamma _{0}<\Gamma _{1}<\infty $, let $V$ be either the space $%
C^{1,\beta }\left( \left( \Gamma _{0},\Gamma _{1}\right) \right) $, $0\leq
\beta <1$ or the space $\mathrm{Lip}\left( \left( \Gamma _{0},\Gamma
_{1}\right) \right) $, let $K^{M}=\left\{ x\in V:\left\vert x\right\vert
_{1}<M,\left\vert x\right\vert _{\ast }>\frac{1}{M}\right\} $and let $%
x_{0}\in V\cap \left\{ \left\vert x\right\vert _{\ast }>0\right\} ,$then for
any $M$, $1<M<\infty $, such that $x_{0}\in K^{M}$, there exists a time $%
T(M) $, such that the system \eqref{grp:BR_alpha_onBanach} has a unique
local solution $x\in C^{1}((-T(M),T(M));K^{M})$.
\end{proposition}

\subsection{Step 2. Global existence.}

To show the global existence, we assume by contradiction, that $T_{\max
}<\infty $, where $\left[ 0,T_{\max }\right) $ is the maximal interval of
existence, and hence the solution leaves in a finite time the open set $%
K^{M} $, for all $M>1$, that is, $\limsup_{t\rightarrow T_{\max
}^{-}}\left\Vert x\right\Vert _{V}=\infty $ or $\limsup_{t\rightarrow
T_{\max }^{-}}\frac{1}{\left\vert x\left( \cdot ,t\right) \right\vert _{\ast
}}=\infty $. Therefore, if we show global bounds on $\frac{1}{\left\vert
x\left( \cdot ,t\right) \right\vert _{\ast }}$ and $\left\Vert x\left( \cdot
,t\right) \right\Vert _{V}$ in $\left[ 0,T_{\max }\right) $, we obtain a
contradiction to the blow-up and thus the obtained local solutions can be
continued for all time. The result extends to negative times as well.

To control the quantities $\frac{1}{\left\vert x\left( \cdot ,t\right)
\right\vert _{\ast }}$ and $\left\Vert x\left( \cdot ,t\right) \right\Vert
_{V}$ we need to bound $\int_{0}^{T_{\max }}\left\Vert \nabla _{x}u\left(
x(\cdot ,t),t\right) \right\Vert _{L^{\infty }\left( \left( \Gamma
_{0},\Gamma _{1}\right) \right) }dt$. The next proposition provides the
bound on the gradient of the velocity .

\begin{proposition}
Let $x_{0}\in \mathrm{Lip}\left( \left( \Gamma _{0},\Gamma _{1}\right)
\right) $ and $\left\vert x_{0}\right\vert _{\ast }>0$. Suppose the solution
exists on $\left[ 0,T_{\max }\right) $, then for $t\in \left[ 0,T_{\max
}\right) $ we have
\begin{equation}
\left\vert \nabla _{x}u\left( x\left( \Gamma ,t\right) ,t\right) \right\vert
\leq \frac{1}{\alpha }C\left( \left\vert x_{0}\right\vert _{\ast
},C_{1}\right) \left( e^{tC_{1}}+1\right) ,  \label{eq:Du_bound}
\end{equation}%
where $C_{1}=C\frac{1}{\alpha ^{2}}\left( \Gamma _{1}-\Gamma _{0}\right) $.
\end{proposition}

\begin{proof}
We write $\nabla _{x}u\left( x(\Gamma ,t),t\right) $ as
\begin{align*}
\nabla _{x}u\left( x(\Gamma ,t),t\right) & =\int_{\Gamma _{0}}^{\Gamma
_{1}}\nabla _{x}K^{\alpha }\left( x\left( \Gamma ,t\right) -x\left( \Gamma
^{\prime },t\right) \right) d\Gamma ^{\prime } \\
& =\int_{\left( \Gamma _{0},\Gamma _{1}\right) \cap E_{\varepsilon
}}+\int_{\left( \Gamma _{0},\Gamma _{1}\right) \backslash E_{\varepsilon
}}=I_{1}+I_{2},
\end{align*}%
where
\begin{equation*}
E_{\varepsilon }=\left\{ \Gamma ^{\prime }\in \left( \Gamma _{0},\Gamma
_{1}\right) :\frac{\left\vert x\left( \Gamma ,t\right) -x\left( \Gamma
^{\prime },t\right) \right\vert }{\alpha }<\varepsilon \right\} ,
\end{equation*}%
for a fixed small $0<\varepsilon <1$, to be further refined later.

Let the vorticity $q(x,t)$ be supported on the curve $\left\{ x\left( \Gamma
,t\right) :\Gamma _{0}\leq \Gamma \leq \Gamma _{1}\right\} $, with a density
$\gamma \left( \Gamma ,t\right) =1/|x_{\Gamma }\left( \Gamma ,t\right) |$
(due to the Lipschitz continuity of $x\left( \Gamma ,t\right) $ its
derivative exists almost everywhere and is essentially bounded, and also due
to $\left\{ \left\vert x\right\vert _{\ast }>0\right\} $, the vorticity
density $\gamma \left( \Gamma ,t\right) \in $ $L^{\infty }\left( \left(
\Gamma _{0},\Gamma _{1}\right) \right) $), that is for every $\varphi \in
C_{c}^{\infty }\left( {\mathbb{R}}^{2}\right) $%
\begin{equation*}
\int_{{\mathbb{R}}^{2}}\varphi (x)dq(x,t)=\int_{\Gamma _{0}}^{\Gamma
_{1}}\varphi \left( x(\Gamma ,t)\right) d\Gamma .
\end{equation*}%
Observe that the vorticity $q(x,t)$ is a finite Radon measure which is the
unique weak solution of the Euler equations given by Theorem \ref{thm:OS01}
of Oliver and Shkoller \cite{a_OS01}. Also, $\left\Vert q\right\Vert _{%
\mathcal{M}}=\Gamma _{1}-\Gamma _{0}$.

Let $\eta $ denote the unique Lagrangian flow map $\partial _{t}\eta
(y,t)=\int_{\mathbb{R}^{2}}K^{\alpha }\left( y,z\right) dq\left( z,t\right) $%
, \thinspace $\eta \left( y,0\right) =y$, $q=q^{in}\circ \eta ^{-1}$, $y\in
\mathbb{R}^{2}$ given by Theorem \ref{thm:OS01}. We remark that in the
formulation of BR-$\alpha $ model, we assumed the positivity of the
vorticity $q$, see Proposition \ref{prop:BR_alpha:equivalence}. Denote the
distance between two points $\eta (y,t)$ and $\eta (y^{\prime },t)$ by $%
r\left( t\right) =\left\vert \eta \left( y,t\right) -\eta \left( y^{\prime
},t\right) \right\vert $, where $r\left( 0\right) =\left\vert y-y^{\prime
}\right\vert $.

Then, using the estimate (2.14) of \cite{a_OS01}, we have
\begin{align*}
\left\vert \frac{d}{dt}r\left( t\right) \right\vert & \leq \int_{\mathbb{R}%
^{2}}\left\vert K^{\alpha }\left( y,z\right) -K^{\alpha }\left( y^{\prime
},z\right) \right\vert dq\left( z,t\right) \\
& \leq C\frac{1}{\alpha }\varphi \left( \frac{r\left( t\right) }{\alpha }%
\right) \left\Vert q\right\Vert _{\mathcal{M}} \\
& =C\frac{1}{\alpha }\varphi \left( \frac{r\left( t\right) }{\alpha }\right)
\left\Vert q^{in}\right\Vert _{\mathcal{M}},
\end{align*}%
where
\begin{equation*}
\varphi \left( r\right) =\left\{
\begin{array}{ll}
0, & r=0, \\
r\left( 1-\log r\right) , & 0<r<1, \\
1, & r\geq 1.%
\end{array}%
\right.
\end{equation*}

By comparison with the solution of the differential equation\footnote{%
\begin{equation*}
\frac{\xi \left( t\right) }{\alpha }=\left\{
\begin{array}{lll}
\left( \frac{\xi \left( 0\right) }{\alpha }\right)
^{e^{tC_{1}}}e^{1-e^{tC_{1}}}, & \frac{\xi \left( 0\right) }{\alpha }<1, &
\\
\frac{\xi \left( 0\right) }{\alpha }-C_{1}t, & \frac{\xi \left( 0\right) }{%
\alpha }\geq 1, & t<t^{\ast \ast }, \\
e^{1-e^{tC_{1}-\frac{\xi \left( 0\right) }{\alpha }+1}}, & \frac{\xi \left(
0\right) }{\alpha }\geq 1, & t\geq t^{\ast \ast },%
\end{array}%
\right.
\end{equation*}%
where%
\begin{align*}
& C_{1}=C\frac{1}{\alpha ^{2}}\left\Vert q^{in}\right\Vert _{M}, \\
t^{\ast \ast }& =\frac{1}{C_{1}}\left( \frac{\xi \left( 0\right) }{\alpha }%
-1\right) .
\end{align*}%
}
\begin{align*}
& \frac{d}{dt}\xi \left( t\right) =-C\frac{1}{\alpha }\varphi \left( \frac{%
\xi \left( t\right) }{\alpha }\right) \left\Vert q^{in}\right\Vert _{%
\mathcal{M}}, \\
& \xi \left( 0\right) =\left\vert x\left( \Gamma ,0\right) -x\left( \Gamma
^{\prime },0\right) \right\vert ,
\end{align*}%
we can choose $\varepsilon $ small enough, $\varepsilon <e^{1-e^{tC_{1}}}$,
where $C_{1}=C\frac{1}{\alpha ^{2}}\left\Vert q^{in}\right\Vert _{M}=\Gamma
_{1}-\Gamma _{0}$, e.g., $\varepsilon =e^{-e^{tC_{1}}}$, such that, for $%
\frac{\left\vert x\left( \Gamma ,t\right) -x\left( \Gamma ^{\prime
},t\right) \right\vert }{\alpha }<\varepsilon $, we have that $\frac{%
\left\vert x\left( \Gamma ,0\right) -x\left( \Gamma ^{\prime },0\right)
\right\vert }{\alpha }=\frac{r\left( 0\right) }{\alpha }<1$.\footnote{%
Otherwise, $\frac{\left\vert x\left( \Gamma ,0\right) -x\left( \Gamma
^{\prime },0\right) \right\vert }{\alpha }=\frac{r\left( 0\right) }{\alpha }%
\geq 1$, hence
\begin{equation*}
\varepsilon >\frac{\left\vert x\left( \Gamma ,t\right) -x\left( \Gamma
^{\prime },t\right) \right\vert }{\alpha }\geq \frac{r\left( t\right) }{%
\alpha }\geq \left\{
\begin{array}{ll}
\frac{r\left( 0\right) }{\alpha }-C_{1}t, & t<t^{\ast \ast }, \\
e^{1-e^{tC_{1}-\frac{r\left( 0\right) }{\alpha }+1}}, & t\geq t^{\ast \ast }.%
\end{array}%
\right.
\end{equation*}%
Let $t\geq t^{\ast \ast }$, then we have $\varepsilon \geq e^{1-e^{tC_{1}}}$%
, a contradiction. Otherwise $t<t^{\ast \ast }=\frac{1}{C_{1}}\left( \frac{%
r\left( 0\right) }{\alpha }-1\right) $, hence $C_{1}t<\left( \frac{r\left(
0\right) }{\alpha }-1\right) $%
\begin{equation*}
\varepsilon >\frac{r\left( t\right) }{\alpha }\geq \frac{r\left( 0\right) }{%
\alpha }-C_{1}t>\frac{r\left( 0\right) }{\alpha }-\frac{r\left( 0\right) }{%
\alpha }+1=1,
\end{equation*}%
a contradiction.} Hence
\begin{align}
\frac{\left\vert x\left( \Gamma ,t\right) -x\left( \Gamma ^{\prime
},t\right) \right\vert }{\alpha }& \geq \frac{r\left( t\right) }{\alpha }%
=\left( \frac{r\left( 0\right) }{\alpha }\right)
^{e^{tC_{1}}}e^{1-e^{tC_{1}}}  \label{eq:trajBelowBound} \\
& =\left( \frac{\left\vert x\left( \Gamma ,0\right) -x\left( \Gamma ^{\prime
},0\right) \right\vert }{\alpha }\right) ^{e^{tC_{1}}}e^{1-e^{tC_{1}}}.
\notag
\end{align}%
Now, using also that $|x_{0}|_{\ast }$ is bounded away from zero, we can
bound $\frac{\left\vert x\left( \Gamma ,t\right) -x\left( \Gamma ^{\prime
},t\right) \right\vert }{\alpha }$ from below, using %
\eqref{eq:trajBelowBound},
\begin{equation*}
1>\varepsilon >\frac{\left\vert x\left( \Gamma ,t\right) -x\left( \Gamma
^{\prime },t\right) \right\vert }{\alpha }\geq \left( \frac{\left\vert
x\left( \Gamma ,0\right) -x\left( \Gamma ^{\prime },0\right) \right\vert }{%
\alpha }\right) ^{e^{tC_{1}}}e^{1-e^{tC_{1}}}\geq \left\vert
x_{0}\right\vert _{\ast }^{e^{tC_{1}}}\left( \frac{\left\vert \Gamma -\Gamma
^{\prime }\right\vert }{\alpha }\right) ^{e^{tC_{1}}}e^{1-e^{tC_{1}}},
\end{equation*}%
which in turn implies the bound (using also \eqref{eq:K_alpha} and %
\eqref{eq:D2psi_bound})
\begin{align*}
I_{1}& \leq \int_{\left( \Gamma _{0},\Gamma _{1}\right) \cap \frac{%
\left\vert x\left( \Gamma ,t\right) -x\left( \Gamma ^{\prime },t\right)
\right\vert }{\alpha }<\varepsilon }\left( \frac{1}{2\pi }\frac{1}{\alpha
^{2}}\left\vert \log \frac{\left\vert x\left( \Gamma ,t\right) -x\left(
\Gamma ^{\prime },t\right) \right\vert }{\alpha }\right\vert +\frac{C}{%
\alpha ^{2}}\right) d\Gamma ^{\prime } \\
& \leq \frac{1}{\alpha }C\left( \left\vert x_{0}\right\vert _{\ast }\right)
\left( e^{tC_{1}}+1\right) .
\end{align*}%
While to bound $I_{2}$, we use the boundness of $\left\vert \nabla
_{x}K^{\alpha }\left( x\left( \Gamma ,t\right) -x\left( \Gamma ^{\prime
},t\right) \right) \right\vert $ in $\{\Gamma ^{\prime }\in \left( \Gamma
_{0},\Gamma _{1}\right) :\frac{\left\vert x\left( \Gamma ,t\right) -x\left(
\Gamma ^{\prime },t\right) \right\vert }{\alpha }\geq \varepsilon \}$.
\begin{align*}
I_{2}& \leq \sup_{\frac{\left\vert x\left( \Gamma ,t\right) -x\left( \Gamma
^{\prime },t\right) \right\vert }{\alpha }\geq \varepsilon }\left\vert
\nabla _{x}K^{a}\left( x\left( \Gamma ,t\right) -x\left( \Gamma ^{\prime
},t\right) \right) \right\vert \int_{\Gamma _{0}}^{\Gamma _{1}}d\Gamma
^{\prime } \\
& \leq C\frac{1}{\alpha ^{2}}\left( \left\vert \log \varepsilon \right\vert
+1\right) \left( \Gamma _{1}-\Gamma _{0}\right) \\
& =C_{1}\left( e^{tC_{1}}+1\right) .
\end{align*}
\end{proof}

Now, the bound on $\left\Vert x\left( \cdot ,t\right) \right\Vert _{C^{0}}$
on $\left[ 0,T_{\max }\right) $ follows from $\frac{dx}{dt}\left( \Gamma
,t\right) =u\left( x\left( \Gamma ,t\right) ,t\right) $ and the fact that
\begin{equation*}
\left\vert u\left( x\left( \Gamma ,t\right) ,t\right) \right\vert \leq
\int_{\Gamma _{0}}^{\Gamma _{1}}\left\vert K^{a}\left( x\left( \Gamma
,t\right) -x\left( \Gamma ^{\prime },t\right) \right) \right\vert d\Gamma
^{\prime }\leq \frac{C}{\alpha }\left( \Gamma _{1}-\Gamma _{0}\right) ,
\end{equation*}%
due to the boundness of $K^{a}$ (see \eqref{eq:K_alpha},%
\eqref{eq:Dpsi_bounded}). Also, by Gr\"{o}nwall inequality the bound %
\eqref{eq:Du_bound} provides bounds on $\frac{1}{\left\vert x\left( \cdot
,t\right) \right\vert _{\ast }}$ and $\left\vert x\left( \cdot ,t\right)
\right\vert _{1}$ on $\left[ 0,T_{\max }\right) $.

Finally, for the initial data in $C^{1,\beta }\left( \left( \Gamma
_{0},\Gamma _{1}\right) \right) $, the bound \eqref{eq:Du_bound} provides
bound on $\left\Vert \frac{dx}{d\Gamma }\left( \cdot ,t\right) \right\Vert
_{C^{0}}$ on $\left[ 0,T_{\max }\right) $ by Gr\"{o}nwall inequality. While
the bound on $\left\vert \frac{dx}{d\Gamma }\left( \cdot ,t\right)
\right\vert _{\beta }$ on $\left[ 0,T_{\max }\right) $ is a consequence of
\begin{equation*}
\frac{d}{dt}x_{\Gamma }\left( \Gamma ,t\right) =\nabla _{x}u\left( x\left(
\Gamma ,t\right) ,t\right) \cdot x_{\Gamma }\left( \Gamma ,t\right) ,
\end{equation*}%
the bound (which is shown in a local existence part, see %
\eqref{eq:grad_K_alpha_Holder})%
\begin{equation*}
\left\vert \nabla _{x}u\left( x\left( \cdot ,t\right) ,t\right) \right\vert
_{\beta }\leq C\left( \frac{1}{\alpha },\left\Vert x_{\Gamma }\left( \cdot
,t\right) \right\Vert _{L^{\infty }},\left\vert x\left( \cdot ,t\right)
\right\vert _{\ast },\Gamma _{1},\Gamma _{0}\right) ,
\end{equation*}%
\eqref{eq:Du_bound} and the Gr\"{o}nwall inequality.

This yields global in time existence and uniqueness of $\mathrm{Lip}$ and $%
C^{1,\beta },0\leq \beta <1$, solutions of \eqref{grp:BR_alpha_onBanach}.

\subsection{Step 3. Higher regularity for closed curves.}

Now we show the higher regularity for an initially closed curve $x_{0}\left(
\Gamma \right) \in C^{n,\beta }\left( S^{1}\right) \cap \left\{ \left\vert
x\right\vert _{\ast }>0\right\} $, $n\geq 1$, $0<\beta <1$. We remark that
the high derivatives of the kernel $K^{\alpha }\left( x\right) $ are
singular at the origin, thus the condition on closedness of the curve.

To provide an \textit{a priori} bound for higher derivatives in
terms of lower ones, we show that for $x\in  C^{n,\beta }\left(
S^{1}\right) \cap \left\{\left\vert x\right\vert _{1}<M,\left\vert
x\right\vert _{\ast }>\frac{1}{M}\right\} $, the map $u$ defined by
\eqref{eq:vel_map} satisfies
\begin{equation*}
\left\Vert u\left( x\right) \right\Vert _{n,\beta }\leq C\left( \frac{1}{%
\alpha },M,\left\Vert x\right\Vert _{n-1,\beta },\frac{1}{^{\beta }}\right)
\left\Vert x\right\Vert _{n,\beta },
\end{equation*}%
hence by Gr\"{o}nwall inequality and the induction argument, it is enough to
control $\left\vert x\right\vert _{\ast }$ and $\left\Vert x\right\Vert
_{1,\beta }$, to guarantee that $x\left( \Gamma \right) \in C^{n,\beta
}\left( S^{1}\right) $, for all $n\geq 1$, (and consequently in $C^{\infty
}\left( S^{1}\right) $, whenever $x_{0}\in C^{\infty }\left( S^{1}\right)
\cap \left\{ \left\vert x\right\vert _{\ast }>0\right\} $).

\begin{lemma}
Let $V$ be the space $ C^{n,\beta }\left( S^{1}\right) $, $n\geq 1$,
$0<\beta <1 $, and let $u$ and $K^{M}$ be as defined in Proposition
\ref{prop:br_alpha_Lip_map}. Then for $x\in K^{M}$
\begin{equation*}
\left\Vert u\left( x\right) \right\Vert _{n,\beta }\leq C\left( \frac{1}{%
\alpha },M,\left\Vert x\right\Vert _{n-1,\beta },\frac{1}{^{\beta
}}\right) \left\Vert x\right\Vert _{n,\beta }.
\end{equation*}
\end{lemma}

\begin{proof}
We show the proof for $n=2$, the proof for general $n$ is similar. The
derivative of $u$ with respect to $\Gamma $ (in the sense of distributions)
satisfies (see Appendix, Lemma \ref{lemma:d2u_dGamma2})
\begin{align*}
\frac{d^{2}}{d\Gamma ^{2}}u\left( x\left( \Gamma \right) \right) &
=\int_{0}^{2\pi }\nabla K^{\alpha }\left( x\left( \Gamma \right) -x\left(
\Gamma ^{\prime }\right) \right) \frac{d^{2}x}{d\Gamma ^{2}}\left( \Gamma
\right) d\Gamma ^{\prime } \\
& +\mathrm{p.v.}\int_{0}^{2\pi }\sum_{i,j=1}^{2}\partial _{x_{i}}\partial
_{x_{j}}K^{\alpha }\left( x\left( \Gamma \right) -x\left( \Gamma ^{\prime
}\right) \right) \frac{dx_{i}}{d\Gamma }\left( \Gamma \right) \frac{dx_{j}}{%
d\Gamma }\left( \Gamma \right) d\Gamma ^{\prime } \\
& =I_{1}+I_{2}.
\end{align*}%
$I_{1}$ can be bounded using similar arguments as for %
\eqref{eq:grad_K_alpha_bnd},
\begin{equation*}
\left\vert I_{1}\right\vert \leq \frac{1}{\alpha ^{2}}C\left( \frac{1}{%
\alpha },M\right) \left\Vert x\right\Vert _{2,\beta }.
\end{equation*}%
We write $I_{2}$ as
\begin{align*}
I_{2}& =\mathrm{p.v.}\left( \int_{\frac{\left\vert \Gamma -\Gamma ^{\prime
}\right\vert }{\alpha }<\varepsilon }+\int_{\left( 0,2\pi \right) \backslash
\left\{ \frac{\left\vert \Gamma -\Gamma ^{\prime }\right\vert }{\alpha }%
<\varepsilon \right\} }\right) \sum_{i,j=1}^{2}\partial _{x_{i}}\partial
_{x_{j}}K^{\alpha }\left( x\left( \Gamma \right) -x\left( \Gamma ^{\prime
}\right) \right) \frac{dx_{i}}{d\Gamma }\left( \Gamma \right) \frac{dx_{j}}{%
d\Gamma }\left( \Gamma \right) d\Gamma ^{\prime } \\
& =I_{21}+I_{22},
\end{align*}%
where, because the curve is closed, we can fix a small $\varepsilon
<\pi /2$ independent of $\Gamma $, by taking $I_{22}$
$=\int_{D\backslash
\left\{ \frac{\left\vert \Gamma -\Gamma ^{\prime }\right\vert }{\alpha }%
<\varepsilon \right\} }$, where $D=\left( 0,2\pi \right) $ if $\varepsilon
\alpha <\Gamma $ $<2\pi -\varepsilon \alpha $, $D=\left( -\pi ,\pi \right) $
if $0\leq \Gamma \leq \varepsilon \alpha $, or $D=\left( \pi ,3\pi \right) $
if $2\pi -\varepsilon \alpha \leq \Gamma \leq 2\pi $. Treating $I_{22}$ as
in the local existence proof, we have%
\begin{align*}
\left\vert I_{22}\right\vert & \leq C\left( M,\frac{1}{\alpha }\right)
\left\Vert x\right\Vert _{1,\beta }^{3}\int_{D\cap \frac{\left\vert \Gamma
-\Gamma ^{\prime }\right\vert }{\alpha }\geq \varepsilon }\left( \frac{%
\alpha }{\left\vert \Gamma -\Gamma ^{\prime }\right\vert }+1\right) d\Gamma
^{\prime } \\
& \leq C\left( \frac{1}{\alpha },M,\left\Vert x\right\Vert _{1,\beta
}\right) \left\vert \log \varepsilon \right\vert .
\end{align*}%
For $I_{21}$ we have that 
\begin{align*}
I_{21}& =\frac{1}{4\pi \alpha ^{2}}\mathrm{p.v.}\int_{\frac{\left\vert
\Gamma -\Gamma ^{\prime }\right\vert }{\alpha }<\varepsilon }\sum_{i,j}\frac{%
\sigma _{ij}\left( x\left( \Gamma \right) -x\left( \Gamma ^{\prime }\right)
\right) \frac{dx_{i}}{d\Gamma }\left( \Gamma \right) }{\left\vert x\left(
\Gamma \right) -x\left( \Gamma ^{\prime }\right) \right\vert }\frac{dx_{j}}{%
d\Gamma }\left( \Gamma \right) d\Gamma ^{\prime } \\
& +\frac{1}{4\pi \alpha ^{2}}\int_{\frac{\left\vert \Gamma -\Gamma ^{\prime
}\right\vert }{\alpha }<\varepsilon }\sum_{i,j}O\left( \left\vert \frac{%
\left\vert x\left( \Gamma \right) -x\left( \Gamma ^{\prime }\right)
\right\vert }{\alpha ^{2}}\log \frac{\left\vert x\left( \Gamma \right)
-x\left( \Gamma ^{\prime }\right) \right\vert }{\alpha }\right\vert \right)
\frac{dx_{i}}{d\Gamma }\left( \Gamma \right) \frac{dx_{j}}{d\Gamma }\left(
\Gamma \right) d\Gamma ^{\prime } \\
& =I_{211}+I_{212},
\end{align*}%
where
\begin{equation}
\sigma _{11}\left( x\right) =\frac{1}{\left\vert x\right\vert ^{3}}%
\begin{pmatrix}
-x_{2}\left( x_{1}^{2}-x_{2}^{2}\right) \\
-x_{1}(x_{1}^{2}+3x_{2}^{2})%
\end{pmatrix}%
,\text{ }\sigma _{12}\left( x\right) =\sigma _{21}\left( x\right) =\frac{1}{%
\left\vert x\right\vert ^{3}}%
\begin{pmatrix}
x_{1}\left( x_{1}^{2}-x_{2}^{2}\right) \\
x_{2}\left( x_{1}^{2}-x_{2}^{2}\right)%
\end{pmatrix}%
,\sigma _{22}\left( x\right) =\frac{1}{\left\vert x\right\vert ^{3}}%
\begin{pmatrix}
x_{2}\left( 3x_{1}^{2}+x_{2}^{2}\right) \\
-x_{1}\left( x_{1}^{2}-x_{2}^{2}\right)%
\end{pmatrix}%
.  \label{eq:sigma}
\end{equation}%
$I_{212}$ is not a singular integral and due to
\begin{equation*}
\left\vert x\left( \Gamma \right) -x\left( \Gamma ^{\prime }\right)
\right\vert \leq \left\Vert x_{\Gamma }\right\Vert _{C^{0}}\left\vert \Gamma
-\Gamma ^{\prime }\right\vert ,
\end{equation*}%
we obtain that%
\begin{equation*}
\left\vert I_{212}\right\vert \leq C\frac{1}{\alpha ^{2}}\left\Vert
x\right\Vert _{1,\beta }^{3}.
\end{equation*}%
We use the observation
\begin{equation}
\left\vert f\left( x\right) -f\left( y\right) -\left( x-y\right) f^{\prime
}\left( x\right) \right\vert \leq \left\vert x-y\right\vert ^{1+\beta
}\left\vert f^{\prime }\right\vert _{\beta }  \label{eq:HolderEstimate}
\end{equation}%
to desingularize the $I_{211}$. We rewrite%
\begin{align*}
\sum_{i,j}\frac{\sigma _{ij}\left( x\left( \Gamma \right) -x\left( \Gamma
^{\prime }\right) \right) \frac{dx_{i}}{d\Gamma }\left( \Gamma \right) }{%
\left\vert x\left( \Gamma \right) -x\left( \Gamma ^{\prime }\right)
\right\vert }\frac{dx_{j}}{d\Gamma }\left( \Gamma \right) &=\sum_{i,j}\frac{%
\sigma _{ij}\left( x\left( \Gamma \right) -x\left( \Gamma ^{\prime }\right)
\right) \frac{dx_{i}}{d\Gamma }\left( \Gamma \right) \left( \Gamma -\Gamma
^{\prime }\right) }{\left\vert x\left( \Gamma \right) -x\left( \Gamma
^{\prime }\right) \right\vert \left( \Gamma -\Gamma ^{\prime }\right) }\frac{%
dx_{j}}{d\Gamma }\left( \Gamma \right) \\
&=\sum_{i,j}\frac{\sigma _{ij}\left( x\left( \Gamma \right) -x\left( \Gamma
^{\prime }\right) \right) \left( \frac{dx_{i}}{d\Gamma }\left( \Gamma
\right) \left( \Gamma -\Gamma ^{\prime }\right) -x_{i}\left( \Gamma \right)
+x_{i}\left( \Gamma ^{\prime }\right) \right) }{\left\vert x\left( \Gamma
\right) -x\left( \Gamma ^{\prime }\right) \right\vert \left( \Gamma -\Gamma
^{\prime }\right) }\frac{dx_{j}}{d\Gamma }\left( \Gamma \right) \\
&+\sum_{i,j}\frac{\sigma _{ij}\left( x\left( \Gamma \right) -x\left( \Gamma
^{\prime }\right) \right) \left( x_{i}\left( \Gamma \right) -x_{i}\left(
\Gamma ^{\prime }\right) \right) }{\left\vert x\left( \Gamma \right)
-x\left( \Gamma ^{\prime }\right) \right\vert \left( \Gamma -\Gamma ^{\prime
}\right) }\frac{dx_{j}}{d\Gamma }\left( \Gamma \right) \\
&=J_{1}+J_{2}.
\end{align*}%
Observe that $J_{2}=\frac{1}{\left( \Gamma ^{\prime } -\Gamma\right) }\left(%
\frac{dx}{d\Gamma }\left( \Gamma \right)\right)^{\perp}$ and $J_{1}\leq
C\left( M\right) \left\Vert x\right\Vert _{1,\beta }^{2}\left\vert \Gamma
-\Gamma ^{\prime }\right\vert ^{-1+\beta }$ due to $\left\vert \sigma
_{ij}\right\vert \leq 1$ (see \eqref{eq:sigma}) and \eqref{eq:HolderEstimate}%
. Hence
\begin{align*}
\left\vert I_{211}\right\vert &=\left\vert \frac{1}{4\pi \alpha ^{2}}\int_{%
\frac{\left\vert \Gamma -\Gamma ^{\prime }\right\vert }{\alpha }<\varepsilon
} J_{1}d\Gamma ^{\prime }+\frac{1}{4\pi \alpha ^{2}} \left(\frac{dx\left(
\Gamma \right)}{d\Gamma }\right)^{\perp} \mathrm{p.v.}\int_{\frac{
\left\vert \Gamma -\Gamma ^{\prime }\right\vert }{\alpha }<\varepsilon }%
\frac{1}{\left( \Gamma^{\prime } -\Gamma \right) }d\Gamma ^{\prime }
\right\vert \\
&\leq C\left( M\right) \frac{1}{\alpha ^{2-\beta }}\left\Vert x\right\Vert
_{1,\beta }^{2}\frac{1}{^{\beta }}\varepsilon ^{\beta }.
\end{align*}%
Summing up, we have that%
\begin{equation*}
\left\vert \frac{d^{2}}{d\Gamma ^{2}}u\left( x\left( \Gamma \right) \right)
\right\vert \leq C\left( \frac{1}{\alpha },M,\left\Vert x\right\Vert
_{1,\beta },\left\vert x\right\vert _{\ast }^{-1},\frac{1}{^{\beta }}\right)
\left\Vert x\right\Vert _{2,\beta }.
\end{equation*}%
Using the same ideas we also bound $\left\vert \frac{d^{2}}{d\Gamma ^{2}}%
u\left( x\left( \Gamma \right) \right) \right\vert _{\beta }$.
\end{proof}

\section*{Appendix}

\setcounter{section}{1} \setcounter{theorem}{0} \renewcommand{\thesection}{%
\Alph{section}}

\begin{lemma}
\label{lemma:d2u_dGamma2}Let $x\in C^{2.\beta }\left( \left( \Gamma
_{0},\Gamma _{1}\right) \right) \cap \left\{ \left\vert x\right\vert _{\ast
}>0\right\} $ then%
\begin{align*}
\frac{d^{2}u}{d\Gamma ^{2}}\left( x\left( \Gamma \right) \right) &
=\int_{\Gamma _{0}}^{\Gamma _{1}}\nabla K^{\alpha }\left( x\left( \Gamma
\right) -x\left( \Gamma ^{\prime }\right) \right) \frac{d^{2}x}{d\Gamma ^{2}}%
\left( \Gamma \right) d\Gamma ^{\prime } \\
& +\mathrm{p.v.}\int_{\Gamma _{0}}^{\Gamma _{1}}\sum_{i,j=1}^{2}\partial
_{x_{i}}\partial _{x_{j}}K^{\alpha }\left( x\left( \Gamma \right) -x\left(
\Gamma ^{\prime }\right) \right) \frac{dx_{i}}{d\Gamma }\left( \Gamma
\right) \frac{dx_{j}}{d\Gamma }\left( \Gamma \right) d\Gamma ^{\prime }
\end{align*}%
(in the sense of distributions).
\end{lemma}

\begin{proof}
By the definition of the distribution derivative, for all $\varphi \in
C_{c}^{\infty }\left( \left( \Gamma _{0},\Gamma _{1}\right) ;\mathbb{R}%
^{2}\right) $
\begin{align*}
\left\langle \frac{d^{2}u}{d\Gamma ^{2}}\left( x\left( \Gamma \right)
\right) ,\varphi \left( \Gamma \right) \right\rangle & =-\left\langle \frac{%
du}{d\Gamma }\left( x\left( \Gamma \right) \right) ,\frac{d\varphi }{d\Gamma
}\left( \Gamma \right) \right\rangle \\
& =-\left\langle \int_{\Gamma _{0}}^{\Gamma _{1}}\nabla K^{\alpha }\left(
x\left( \Gamma \right) -x\left( \Gamma ^{\prime }\right) \right) \frac{dx}{%
d\Gamma }\left( \Gamma \right) d\Gamma ^{\prime },\frac{d\varphi }{d\Gamma }%
\left( \Gamma \right) \right\rangle \\
& =-\lim_{\varepsilon \rightarrow 0}\int_{\Gamma _{0}}^{\Gamma
_{1}}\int_{\left( \Gamma _{0},\Gamma _{1}\right) \cap \frac{\left\vert
\Gamma -\Gamma ^{\prime }\right\vert }{\alpha }>\varepsilon }\frac{d\varphi
}{d\Gamma }\left( \Gamma \right) \nabla K^{\alpha }\left( x\left( \Gamma
\right) -x\left( \Gamma ^{\prime }\right) \right) \frac{dx}{d\Gamma }\left(
\Gamma \right) d\Gamma d\Gamma ^{\prime }
\end{align*}%
where for a fixed $\Gamma $ we take $\varepsilon <\min \left\{ \frac{\Gamma
-\Gamma _{0}}{\alpha },\frac{\Gamma _{1}-\Gamma }{\alpha }\right\} $. Denote
$D=\left( \Gamma _{0},\Gamma _{1}\right) \cap \frac{\left\vert \Gamma
-\Gamma ^{\prime }\right\vert }{\alpha }>\varepsilon $, by integration by
parts we get%
\begin{align*}
& =\lim_{\varepsilon \rightarrow 0}\int_{\Gamma _{0}}^{\Gamma _{1}}\left( -
\left[ \varphi \left( \Gamma \right) \nabla K^{\alpha }\left( x\left( \Gamma
\right) -x\left( \Gamma ^{\prime }\right) \right) \frac{dx}{d\Gamma }\left(
\Gamma \right) \right] _{\partial D}+\int_{D}\varphi \left( \Gamma \right)
\frac{d}{d\Gamma }\left[ \nabla K^{\alpha }\left( x\left( \Gamma \right)
-x\left( \Gamma ^{\prime }\right) \right) \frac{dx}{d\Gamma }\left( \Gamma
\right) \right] d\Gamma \right) d\Gamma ^{\prime } \\
& =\lim_{\varepsilon \rightarrow 0}\int_{\Gamma _{0}}^{\Gamma _{1}}\left(
A+B\right) d\Gamma ^{\prime }
\end{align*}%
For $A$ we have%
\begin{align*}
A& =-\varphi \left( \Gamma ^{\prime }-\varepsilon \alpha \right) \nabla
K^{\alpha }\left( x\left( \Gamma ^{\prime }-\varepsilon \alpha \right)
-x\left( \Gamma ^{\prime }\right) \right) \frac{dx}{d\Gamma }\left( \Gamma
^{\prime }-\varepsilon \alpha \right) \\
& +\varphi \left( \Gamma ^{\prime }+\varepsilon \alpha \right) \nabla
K^{\alpha }\left( x\left( \Gamma ^{\prime }+\varepsilon \alpha \right)
-x\left( \Gamma ^{\prime }\right) \right) \frac{dx}{d\Gamma }\left( \Gamma
^{\prime }+\varepsilon \alpha \right) \\
& =I_{1}+I_{2}+I_{3},
\end{align*}%
where%
\begin{align*}
I_{1}& =\left[ \varphi \left( \Gamma ^{\prime }+\varepsilon \alpha \right)
-\varphi \left( \Gamma ^{\prime }-\varepsilon \alpha \right) \right] \nabla
K^{\alpha }\left( x\left( \Gamma ^{\prime }+\varepsilon \alpha \right)
-x\left( \Gamma ^{\prime }\right) \right) \frac{dx}{d\Gamma }\left( \Gamma
^{\prime }+\varepsilon \alpha \right) , \\
I_{2}& =\varphi \left( \Gamma ^{\prime }-\varepsilon \alpha \right) \left[
\nabla K^{\alpha }\left( x\left( \Gamma ^{\prime }+\varepsilon \alpha
\right) -x\left( \Gamma ^{\prime }\right) \right) -\nabla K^{\alpha }\left(
x\left( \Gamma ^{\prime }-\varepsilon \alpha \right) -x\left( \Gamma
^{\prime }\right) \right) \right] \frac{dx}{d\Gamma }\left( \Gamma ^{\prime
}+\varepsilon \alpha \right) , \\
I_{3}& =\varphi \left( \Gamma ^{\prime }-\varepsilon \alpha \right) \nabla
K^{\alpha }\left( x\left( \Gamma ^{\prime }-\varepsilon \alpha \right)
-x\left( \Gamma ^{\prime }\right) \right) \left[ \frac{dx}{d\Gamma }\left(
\Gamma ^{\prime }+\varepsilon \alpha \right) -\frac{dx}{d\Gamma }\left(
\Gamma ^{\prime }-\varepsilon \alpha \right) \right] .
\end{align*}%
Now, since for $y\in \mathbb{R}^{2},\frac{\left\vert y\right\vert }{\alpha }%
\rightarrow 0:\left\vert \nabla K^{\alpha }\left( y\right) \right\vert \leq -%
\frac{1}{2\pi }\frac{1}{\alpha ^{2}}\log \frac{\left\vert y\right\vert }{%
\alpha }+O\left( \frac{1}{\alpha ^{2}}\right) $ and $\left\vert \frac{dx}{%
d\Gamma }\right\vert \varepsilon \geq \frac{\left\vert x\left( \Gamma
^{\prime }+\varepsilon \alpha \right) -x\left( \Gamma ^{\prime }\right)
\right\vert }{\alpha }\geq \left\vert x\right\vert _{\ast }\varepsilon $ we
have
\begin{align*}
\left\vert \nabla K^{\alpha }\left( x\left( \Gamma ^{\prime }+\varepsilon
\alpha \right) -x\left( \Gamma ^{\prime }\right) \right) \right\vert & \leq -%
\frac{1}{2\pi }\frac{1}{\alpha ^{2}}\log \frac{\left\vert x\left( \Gamma
^{\prime }+\varepsilon \alpha \right) -x\left( \Gamma ^{\prime }\right)
\right\vert }{\alpha }+\frac{C}{\alpha ^{2}} \\
& \leq C\left( \left\vert x\right\vert _{\ast },\frac{1}{\alpha }\right)
\left( \log \varepsilon +1\right) ,
\end{align*}%
Hence%
\begin{equation*}
\left\vert I_{1}\right\vert \leq C\left( \left\vert x\right\vert _{\ast },%
\frac{1}{\alpha }\right) \left\Vert \frac{d\varphi }{d\Gamma }\right\Vert
_{C^{0}}\left\Vert \frac{dx}{d\Gamma }\right\Vert _{C^{0}}\varepsilon \left(
\log \varepsilon +1\right) \rightarrow 0\text{, as }\varepsilon \rightarrow 0%
\text{.}
\end{equation*}%
Similarly,%
\begin{equation*}
\left\vert I_{3}\right\vert \leq C\left( \left\vert x\right\vert _{\ast },%
\frac{1}{\alpha }\right) \left\Vert \varphi \right\Vert _{C^{0}}\left\Vert
\frac{d^{2}x}{d\Gamma ^{2}}\right\Vert _{C^{0}}\varepsilon \left( \log
\varepsilon +1\right) \rightarrow 0\text{, as }\varepsilon \rightarrow 0%
\text{.}
\end{equation*}%
For $I_{2}$ we have%
\begin{equation*}
\left\vert I_{2}\right\vert \leq \left\Vert \varphi \right\Vert
_{C^{0}}\left\Vert \frac{dx}{d\Gamma }\right\Vert _{C^{0}}\left\vert \nabla
K^{\alpha }\left( x\left( \Gamma ^{\prime }+\varepsilon \alpha \right)
-x\left( \Gamma ^{\prime }\right) \right) -\nabla K^{\alpha }\left( x\left(
\Gamma ^{\prime }-\varepsilon \alpha \right) -x\left( \Gamma ^{\prime
}\right) \right) \right\vert ,
\end{equation*}%
using that for $y\in \mathbb{R}^{2}$
\begin{equation*}
\nabla K^{\alpha }\left( y\right) =\frac{1}{\left\vert y\right\vert }D\Psi
^{\alpha }\left( \left\vert y\right\vert \right) \left( \sigma \left(
y\right) +J\right) -\sigma \left( y\right) D^{2}\Psi ^{\alpha }\left(
\left\vert y\right\vert \right) ,\quad
\end{equation*}%
where
\begin{equation*}
\sigma \left( y\right) =\frac{1}{\left\vert y\right\vert ^{2}}%
\begin{pmatrix}
y_{1}y_{2} & y_{2}^{2} \\
-y_{1}^{2} & -y_{1}y_{2}%
\end{pmatrix}%
,J=%
\begin{pmatrix}
0 & -1 \\
1 & 0%
\end{pmatrix}%
,
\end{equation*}%
for $\frac{\left\vert y\right\vert }{\alpha },\frac{\left\vert y^{\prime
}\right\vert }{\alpha }\rightarrow 0$, we obtain
\begin{align*}
\left\vert \nabla K^{\alpha }\left( y\right) -\nabla K^{\alpha }\left(
y^{\prime }\right) \right\vert & \leq \frac{1}{4\pi }\frac{1}{\alpha ^{2}}%
\left\vert \sigma \left( y\right) -\sigma \left( y^{\prime }\right)
\right\vert +C\frac{\left\vert y\right\vert ^{2}}{\alpha ^{4}}\left\vert
\log \frac{\left\vert y\right\vert }{\alpha }\right\vert +C\frac{\left\vert
y^{\prime }\right\vert ^{2}}{\alpha ^{4}}\left\vert \log \frac{\left\vert
y^{\prime }\right\vert }{\alpha }\right\vert \\
& +\frac{1}{4\pi }\frac{1}{\alpha ^{2}}\left\vert \log \left\vert
y\right\vert -\log \left\vert y^{\prime }\right\vert \right\vert .
\end{align*}%
Now, due to
\begin{align*}
& \left\vert \log \left\vert x\left( \Gamma ^{\prime }+\varepsilon \alpha
\right) -x\left( \Gamma ^{\prime }\right) \right\vert -\log \left\vert
x\left( \Gamma ^{\prime }-\varepsilon \alpha \right) -x\left( \Gamma
^{\prime }\right) \right\vert \right\vert \leq C\left( \left\Vert
x\right\Vert _{1,\beta },\frac{1}{\left\vert x\right\vert _{\ast }}\right)
\alpha ^{\beta }\varepsilon ^{\beta }, \\
& \left\vert \sigma \left( x\left( \Gamma ^{\prime }+\varepsilon \alpha
\right) -x\left( \Gamma ^{\prime }\right) \right) -\sigma \left( x\left(
\Gamma ^{\prime }-\varepsilon \alpha \right) -x\left( \Gamma ^{\prime
}\right) \right) \right\vert \leq C\left( \left\Vert x\right\Vert _{1,\beta
},\frac{1}{\left\vert x\right\vert _{\ast }}\right) \alpha ^{\beta
}\varepsilon ^{\beta },
\end{align*}%
we obtain%
\begin{equation*}
\left\vert \nabla _{x}K^{\alpha }\left( x\left( \Gamma ^{\prime
}+\varepsilon \alpha \right) -x\left( \Gamma ^{\prime }\right) \right)
-\nabla _{x}K^{\alpha }\left( x\left( \Gamma ^{\prime }-\varepsilon \alpha
\right) -x\left( \Gamma ^{\prime }\right) \right) \right\vert \rightarrow 0%
\text{, as }\varepsilon \rightarrow 0\text{.}
\end{equation*}%
We also have that
\begin{align*}
\vert \nabla _{x}K^{\alpha }\left( x\left( \Gamma ^{\prime
}+\varepsilon \alpha \right) -x\left( \Gamma ^{\prime }\right)
\right) &-\nabla _{x}K^{\alpha }\left( x\left( \Gamma ^{\prime
}-\varepsilon \alpha
\right) -x\left( \Gamma ^{\prime }\right) \right) \vert   \\
& \leq \frac{1%
}{2\pi }\frac{1}{\alpha ^{2}}+\frac{C}{\alpha ^{2}}\left\Vert \frac{dx}{%
d\Gamma }\right\Vert _{C^{0}}^{2}\varepsilon ^{2}\left\vert \log \left\Vert
\frac{dx}{d\Gamma }\right\Vert _{C^{0}}\varepsilon \right\vert +\frac{C}{%
\alpha ^{2}}\frac{\alpha ^{\beta }\varepsilon ^{\beta }\left\vert \frac{dx}{%
d\Gamma }\right\vert _{\beta }}{\left\vert x\right\vert _{\ast }} \\
& \leq C\left( \frac{1}{\alpha },\left\Vert \frac{dx}{d\Gamma }\right\Vert
_{1,\beta },\left\vert x\right\vert _{\ast }\right) ,
\end{align*}%
hence by the Lebesgue's dominated convergence theorem
\begin{equation*}
\lim_{\varepsilon \rightarrow 0}\int_{\Gamma _{0}}^{\Gamma _{1}}Ad\Gamma
^{\prime }=0.
\end{equation*}%
For $B$ we have%
\begin{align*}
\frac{d}{d\Gamma }\left[ \nabla K^{\alpha }\left( x\left( \Gamma \right)
-x\left( \Gamma ^{\prime }\right) \right) \frac{dx}{d\Gamma }\left( \Gamma
\right) \right] & =\nabla K^{\alpha }\left( x\left( \Gamma \right) -x\left(
\Gamma ^{\prime }\right) \right) \frac{d^{2}x}{d\Gamma ^{2}}\left( \Gamma
\right) \\
& +\sum_{i,j}\partial _{x_{i}}\partial _{x_{j}}K^{\alpha }\left( x\left(
\Gamma \right) -x\left( \Gamma ^{\prime }\right) \right) \frac{dx_{i}}{%
d\Gamma }\left( \Gamma \right) \frac{dx_{j}}{d\Gamma }\left( \Gamma \right) .
\end{align*}%
Hence%
\begin{align*}
& \left\langle \frac{d^{2}u}{d\Gamma ^{2}}\left( x\left( \Gamma ,t\right)
,t\right) ,\varphi \left( \Gamma \right) \right\rangle =\int_{\Gamma
_{0}}^{\Gamma _{1}}d\Gamma \varphi \left( \Gamma \right) \\
& \cdot \left( \int_{\Gamma _{0}}^{\Gamma _{1}}\nabla K^{\alpha }\left(
x\left( \Gamma \right) -x\left( \Gamma ^{\prime }\right) \right) \frac{d^{2}x%
}{d\Gamma ^{2}}\left( \Gamma \right) d\Gamma ^{\prime }+\mathrm{p.v.}%
\int_{\Gamma _{0}}^{\Gamma _{1}}\sum_{i,j}\partial _{x_{i}}\partial
_{x_{j}}K^{\alpha }\left( x\left( \Gamma \right) -x\left( \Gamma ^{\prime
}\right) \right) \frac{dx_{i}}{d\Gamma }\left( \Gamma \right) \frac{dx_{j}}{%
d\Gamma }\left( \Gamma \right) d\Gamma ^{\prime }\right) ,
\end{align*}%
which concludes the proof.
\end{proof}

\section*{Acknowledgements}

The authors would like to thank the anonymous referee for a careful reading
of the manuscript and for valuable comments. C.B. would like to thank the
Faculty of Mathematics and Computer Science at the Weizmann Institute of
Science for the kind hospitality where this work was initiated. This work
was supported in part by the BSF grant no.~2004271, the ISF grant
no.~120/06, and the NSF grants no.~DMS-0504619 and no.~DMS-0708832.


\bibliographystyle{siam}
\bibliography{VortexSheetBib}

\begin{thebibliography}{10}

\bibitem{a_BP06}
{\sc G.~R. Baker and L.~D. Pham}, {\em A comparison of blob methods for vortex
  sheet roll-up}, J. Fluid Mech., 547 (2006), pp.~297--316.

\bibitem{a_BFR80}
{\sc J.~Bardina, J.~H. Ferziger, and W.~C. Reynolds}, {\em Improved
  subgrid-scale models for large-eddy simulation}, Am. Inst. Aeronaut.
  Astronaut. Paper,  (1980), pp.~80--1357.

\bibitem{a_BLT08}
{\sc C.~Bardos, J.~S. Linshiz, and E.~S. Titi}, {\em Global regularity for a
  {B}irkhoff-{R}ott-$\alpha$ approximation of the dynamics of vortex sheets of
  the {2D} {E}uler equations}, Physica D: Nonlinear Phenomena, 237 (2008),
  pp.~1905--1911.

\bibitem{a_BT07}
{\sc C.~Bardos and E.~S. Titi}, {\em Euler equations of incompressible ideal
  fluids}, Uspekhi Matematicheskikh Nauk, 62 (2007), pp.~5--46.

\bibitem{a_BKM84}
{\sc J.~T. Beale, T.~Kato, and A.~Majda}, {\em Remarks on the breakdown of
  smooth solutions for the {$3$}-{D} {E}uler equations}, Comm. Math. Phys., 94
  (1984), pp.~61--66.

\bibitem{a_BM85}
{\sc J.~T. Beale and A.~Majda}, {\em High order accurate vortex methods with
  explicit velocity kernels}, J. Comput. Phys., 58 (1985), pp.~188--208.

\bibitem{b_BIL06}
{\sc L.~C. Berselli, T.~Iliescu, and W.~J. Layton}, {\em Mathematics of Large
  Eddy Simulation of Turbulent Flows}, Scientific Computation, Springer-Verlag,
  Berlin, 2006.

\bibitem{a_B62}
{\sc G.~Birkhoff}, {\em Helmholtz and {T}aylor instability}, in Proc. Sympos.
  Appl. Math., Vol. XIII, American Mathematical Society, Providence, R.I.,
  1962, pp.~55--76.

\bibitem{a_CO89}
{\sc R.~E. Caflisch and O.~F. Orellana}, {\em Singular solutions and
  ill-posedness for the evolution of vortex sheets}, SIAM J. Math. Anal., 20
  (1989), pp.~293--307.

\bibitem{a_CHT04}
{\sc C.~Cao, D.~D. Holm, and E.~S. Titi}, {\em On the {C}lark-{$\alpha$} model
  of turbulence: global regularity and long-time dynamics}, J. Turbul., 6
  (2005), pp.~Paper 20, 11 pp. (electronic).

\bibitem{a_CLT06}
{\sc Y.~Cao, E.~Lunasin, and E.~S. Titi}, {\em Global well-posedness of the
  three-dimensional viscous and inviscid simplified {B}ardina turbulence
  models}, Commun. Math. Sci., 4 (2006), pp.~823--848.

\bibitem{b_C98}
{\sc J.-Y. Chemin}, {\em Perfect Incompressible Fluids}, vol.~14 of Oxford
  Lecture Series in Mathematics and its Applications, The Clarendon Press
  Oxford University Press, New York, 1998.
\newblock Translated from the 1995 French original by Isabelle Gallagher and
  Dragos Iftimie.

\bibitem{a_CFHOTW98}
{\sc S.~Chen, C.~Foias, D.~D. Holm, E.~Olson, E.~S. Titi, and S.~Wynne}, {\em
  Camassa-{H}olm equations as a closure model for turbulent channel and pipe
  flow}, Phys. Rev. Lett., 81 (1998), pp.~5338--5341.

\bibitem{a_CFHOTW99}
\leavevmode\vrule height 2pt depth -1.6pt width 23pt, {\em The {C}amassa-{H}olm
  equations and turbulence}, Phys. D, 133 (1999), pp.~49--65.

\bibitem{a_CFHOTW99_ChanPipe}
\leavevmode\vrule height 2pt depth -1.6pt width 23pt, {\em A connection between
  the {C}amassa-{H}olm equations and turbulent flows in channels and pipes},
  Phys. Fluids, 11 (1999), pp.~2343--2353.
\newblock The International Conference on Turbulence (Los Alamos, NM, 1998).

\bibitem{a_CHMZ99}
{\sc S.~Chen, D.~D. Holm, L.~G. Margolin, and R.~Zhang}, {\em Direct numerical
  simulations of the {N}avier-{S}tokes alpha model}, Phys. D, 133 (1999),
  pp.~66--83.

\bibitem{a_CTV05}
{\sc V.~V. Chepyzhov, E.~S. Titi, and M.~I. Vishik}, {\em On the convergence of
  solutions of the {L}eray-{$\alpha$} model to the trajectory attractor of the
  {3D} {N}avier-{S}tokes system}, Discrete Contin. Dyn. Syst., 17 (2007),
  pp.~481--500.

\bibitem{a_CHOT05}
{\sc A.~Cheskidov, D.~D. Holm, E.~Olson, and E.~S. Titi}, {\em On a
  {L}eray-{$\alpha$} model of turbulence}, Proc. R. Soc. Lond. Ser. A Math.
  Phys. Eng. Sci., 461 (2005), pp.~629--649.

\bibitem{a_CB73}
{\sc A.~J. Chorin and P.~J. Bernard}, {\em Discretization of a vortex sheet,
  with an example of roll-up (for elliptically loaded wings)}, J. Comput.
  Phys., 13 (1973), pp.~423--429.

\bibitem{a_CFR79}
{\sc R.~A. Clark, J.~H. Ferziger, and W.~C. Reynolds}, {\em Evaluation of
  subgrid-scale models using an accurately simulated turbulent flow}, J. Fluid
  Mech., 91 (1979), pp.~1--16.

\bibitem{a_C01}
{\sc P.~Constantin}, {\em An {E}ulerian-{L}agrangian approach to the
  {N}avier-{S}tokes equations}, Comm. Math. Phys., 216 (2001), pp.~663--686.

\bibitem{a_CK00}
{\sc G.-H. Cottet and P.~D. Koumoutsakos}, {\em Vortex Methods: Theory and
  Practice}, Cambridge University Press, 2000.

\bibitem{a_CBT00}
{\sc S.~J. Cowley, G.~R. Baker, and S.~Tanveer}, {\em On the formation of
  {M}oore curvature singularities in vortex sheets}, J. Fluid Mech., 378
  (2000), pp.~233--267.

\bibitem{a_LS07}
{\sc C.~De~Lellis and J.~L. Sz\'ekelyhidi}, {\em On admissibility criteria for
  weak solutions of the {E}uler equations}, Preprint,  (2007).

\bibitem{a_D91}
{\sc J.-M. Delort}, {\em Existence de nappes de tourbillon en dimension deux},
  J. Amer. Math. Soc., 4 (1991), pp.~553--586.

\bibitem{a_DM87}
{\sc R.~J. DiPerna and A.~J. Majda}, {\em Concentrations in regularizations for
  {$2$}-{D} incompressible flow}, Comm. Pure Appl. Math., 40 (1987),
  pp.~301--345.

\bibitem{a_DM87b}
\leavevmode\vrule height 2pt depth -1.6pt width 23pt, {\em Oscillations and
  concentrations in weak solutions of the incompressible fluid equations},
  Comm. Math. Phys., 108 (1987), pp.~667--689.

\bibitem{a_DM88}
\leavevmode\vrule height 2pt depth -1.6pt width 23pt, {\em Reduced {H}ausdorff
  dimension and concentration-cancellation for two-dimensional incompressible
  flow}, J. Amer. Math. Soc., 1 (1988), pp.~59--95.

\bibitem{a_DR88}
{\sc J.~Duchon and R.~Robert}, {\em Global vortex sheet solutions of {E}uler
  equations in the plane}, J. Diff. Eq., 73 (1988), pp.~215--224.

\bibitem{a_EM94}
{\sc L.~C. Evans and S.~Muller}, {\em Hardy spaces and the two-dimensional
  {E}uler equations with nonnegative vorticity}, J. Amer. Math. Soc., 7 (1994),
  pp.~199--219.

\bibitem{a_FHT01}
{\sc C.~Foias, D.~D. Holm, and E.~S. Titi}, {\em The {N}avier-{S}tokes-alpha
  model of fluid turbulence}, Phys. D, 152/153 (2001), pp.~505--519.

\bibitem{a_FHT02}
\leavevmode\vrule height 2pt depth -1.6pt width 23pt, {\em The three
  dimensional viscous {C}amassa-{H}olm equations, and their relation to the
  {N}avier-{S}tokes equations and turbulence theory}, J. Dynam. Differential
  Equations, 14 (2002), pp.~1--35.

\bibitem{b_F99}
{\sc G.~B. Folland}, {\em Real {A}nalysis: {M}odern {T}echniques and {T}heir
  {A}pplications}, John Wiley \& Sons Inc., New York, 2nd~ed., 1999.

\bibitem{a_GKT08}
{\sc B.~Geurts, A.~Kuczaj, and E.~S. Titi}, {\em Regularization modeling for
  large-eddy simulation of homogeneous isotropic decaying turbulence}, J. Phys.
  A: Math. Theor., 41 (2008), p.~344008 (29pp).

\bibitem{a_GH03}
{\sc B.~J. Geurts and D.~D. Holm}, {\em Regularization modeling for large-eddy
  simulation}, Phys. Fluids, 15 (2003), pp.~L13--L16.

\bibitem{a_GH06}
\leavevmode\vrule height 2pt depth -1.6pt width 23pt, {\em Leray and
  {LANS}-$\alpha$ modelling of turbulent mixing}, J. Turbul., 7 (2006),
  pp.~1--33.

\bibitem{a_H02_pA}
{\sc D.~D. Holm}, {\em Variational principles for {L}agrangian-averaged fluid
  dynamics}, J. Phys. A, 35 (2002), pp.~679--688.

\bibitem{a_HMR98b}
{\sc D.~D. Holm, J.~E. Marsden, and T.~S. Ratiu}, {\em The {E}uler-{P}oincar\'e
  equations and semidirect products with applications to continuum theories},
  Adv. Math., 137 (1998), pp.~1--81.

\bibitem{a_HMR98a}
\leavevmode\vrule height 2pt depth -1.6pt width 23pt, {\em {E}uler-{P}oincar\'e
  models of ideal fluids with nonlinear dispersion}, Phys. Rev. Lett., 80
  (1998), pp.~4173--4176.

\bibitem{a_HN03}
{\sc D.~D. Holm and B.~T. Nadiga}, {\em Modeling mesoscale turbulence in the
  barotropic double-gyre circulation}, J. Phys. Oceanogr., 33 (2003),
  pp.~2355--2365.

\bibitem{a_HNP06}
{\sc D.~D. Holm, M.~Nitsche, and V.~Putkaradze}, {\em Euler-alpha and vortex
  blob regularization of vortex filament and vortex sheet motion}, J. Fluid
  Mech., 555 (2006), pp.~149--176.

\bibitem{a_HT05}
{\sc D.~D. Holm and E.~S. Titi}, {\em Computational models of turbulence: the
  {LANS}-alpha model and the role of global analysis}, SIAM News, 38 (2005).

\bibitem{a_ILT05}
{\sc A.~A. Ilyin, E.~Lunasin, and E.~S. Titi}, {\em A
  {M}odified-{L}eray-{$\alpha$} subgrid scale model of turbulence},
  Nonlinearity, 19 (2006), pp.~879--897.

\bibitem{a_KT07}
{\sc B.~Khouider and E.~S. Titi}, {\em An inviscid regularization for the
  surface quasi-geostrophic equation}, Comm. Pure Appl. Math.,  (2007 (to
  appear)).

\bibitem{a_K86b}
{\sc R.~Krasny}, {\em Desingularization of periodic vortex sheet roll-up}, J.
  Comput. Phys., 65 (1986), pp.~292--313.

\bibitem{a_K86a}
\leavevmode\vrule height 2pt depth -1.6pt width 23pt, {\em A study of
  singularity formation in a vortex sheet by the point-vortex approximation},
  J. Fluid Mech., 167 (1986), pp.~65--93.

\bibitem{a_K87}
\leavevmode\vrule height 2pt depth -1.6pt width 23pt, {\em Computation of
  vortex sheet roll-up in the {T}refftz plane}, J. Fluid Mech., 184 (1987),
  pp.~123--155.

\bibitem{a_LL03}
{\sc W.~Layton and R.~Lewandowski}, {\em A simple and stable scale-similarity
  model for large eddy simulation: energy balance and existence of weak
  solutions}, Appl. Math. Lett., 16 (2003), pp.~1205--1209.

\bibitem{a_LL06}
\leavevmode\vrule height 2pt depth -1.6pt width 23pt, {\em On a well-posed
  turbulence model}, Discrete Contin. Dyn. Syst. Ser. B, 6 (2006),
  pp.~111--128.

\bibitem{a_L02}
{\sc G.~Lebeau}, {\em R\'egularit\'e du probl\`eme de {K}elvin-{H}elmholtz pour
  l'\'equation d'{E}uler 2d}, ESAIM Control Optim. Calc. Var., 8 (2002),
  pp.~801--825 (electronic).

\bibitem{a_L06}
{\sc R.~Lewandowski}, {\em Vorticities in a {LES} model for {3D} periodic
  turbulent flows}, Journ. Math. Fluid. Mech., 8 (2006), pp.~398--422.

\bibitem{a_LT07}
{\sc J.~S. Linshiz and E.~S. Titi}, {\em Analytical study of certain
  magnetohydrodynamic-alpha models}, J. Math. Phys., 48 (2007), pp.~065504, 28.

\bibitem{a_LX95}
{\sc J.~G. Liu and Z.~Xin}, {\em Convergence of vortex methods for weak
  solutions to the {2D} {E}uler equations with vortex sheet data}, Comm. Pure
  Appl. Math., 48 (1995), pp.~611--628.

\bibitem{a_LX00}
{\sc J.~G. Liu and Z.~P. Xin}, {\em Convergence of the point vortex method for
  2-{D} vortex sheet}, Math. Comp., 70 (2000), pp.~595--606.

\bibitem{a_FLLZ06}
{\sc M.~C. Lopes~Filho, J.~Lowengrub, H.~J. Nussenzveig~Lopes, and Y.~Zheng},
  {\em Numerical evidence of nonuniqueness in the evolution of vortex sheets},
  M2AN Math. Model. Numer. Anal., 40 (2006), pp.~225--237.

\bibitem{a_FLS06}
{\sc M.~C. Lopes~Filho, H.~J. Nussenzveig~Lopes, and S.~Schochet}, {\em A
  criterion for the equivalence of the {B}irkhoff-{R}ott and {E}uler
  descriptions of vortex sheet evolution}, Trans. Amer. Math. Soc., 359 (2007),
  pp.~4125--4142 (electronic).

\bibitem{a_LfNlX01}
{\sc M.~C. Lopes~Filho, H.~J. Nussenzveig~Lopes, and Z.~Xin}, {\em Existence of
  vortex sheets with reflection symmetry in two space dimensions}, Arch.
  Ration. Mech. Anal., 158 (2001), pp.~235--257.

\bibitem{a_M93}
{\sc A.~Majda}, {\em Remarks on weak solutions for vortex sheets with a
  distinguished sign}, Indiana Univ. Math. J, 42 (1993), pp.~921--939.

\bibitem{b_MB02}
{\sc A.~J. Majda and A.~L. Bertozzi}, {\em Vorticity and Incompressible Flow},
  vol.~27 of Cambridge Texts in Applied Mathematics, Cambridge University
  Press, Cambridge, 2002.

\bibitem{b_MP94}
{\sc C.~Marchioro and M.~Pulvirenti}, {\em Mathematical Theory of
  Incompressible Nonviscous Fluids}, vol.~96 of Applied Mathematical Sciences,
  Springer-Verlag, New York, 1994.

\bibitem{a_MS03}
{\sc J.~E. Marsden and S.~Shkoller}, {\em The anisotropic {L}agrangian averaged
  {E}uler and {N}avier-{S}tokes equations}, Arch. Ration. Mech. Anal., 166
  (2003), pp.~27--46.

\bibitem{a_MBO82}
{\sc D.~I. Meiron, G.~R. Baker, and S.~A. Orszag}, {\em Analytic structure of
  vortex sheet dynamics. {I}. {K}elvin-{H}elmholtz instability}, J. Fluid
  Mech., 114 (1982), pp.~283--298.

\bibitem{a_MKSM03}
{\sc K.~Mohseni, B.~Kosovi{\'c}, S.~Shkoller, and J.~E. Marsden}, {\em
  Numerical simulations of the {L}agrangian averaged {N}avier-{S}tokes
  equations for homogeneous isotropic turbulence}, Phys. Fluids, 15 (2003),
  pp.~524--544.

\bibitem{a_M79}
{\sc D.~W. Moore}, {\em The spontaneous appearance of a singularity in the
  shape of an evolving vortex sheet}, Proc. R. Soc. Lond. A, 365 (1979),
  pp.~105--119.

\bibitem{a_OS01}
{\sc M.~Oliver and S.~Shkoller}, {\em The vortex blob method as a second-grade
  non-{N}ewtonian fluid}, Comm. Partial Differential Equations, 26 (2001),
  pp.~295--314.

\bibitem{a_P89}
{\sc D.~I. Pullin}, {\em On similarity flows containing two-branched vortex
  sheets}, in Mathematical aspects of vortex dynamics (Leesburg, VA, 1988),
  SIAM, Philadelphia, PA, 1989, pp.~97--106.

\bibitem{a_R56}
{\sc N.~Rott}, {\em Diffraction of a weak shock with vortex generation}, J.
  Fluid Mech., 1 (1956), pp.~111--128.

\bibitem{b_S92}
{\sc P.~G. Saffman}, {\em Vortex Dynamics}, Cambridge Monographs on Mechanics
  and Applied Mathematics, Cambridge University Press, New York, 1992.

\bibitem{a_SB79}
{\sc P.~G. Saffman and G.~R. Baker}, {\em Vortex interactions}, Ann. Rev. Fluid
  Mech., 11 (1979), pp.~95--121.

\bibitem{a_S93}
{\sc V.~Scheffer}, {\em An inviscid flow with compact support in space-time},
  J. Geom. Anal., 3 (1993), pp.~343--401.

\bibitem{a_S95}
{\sc S.~Schochet}, {\em The weak vorticity formulation of the {2D} {E}uler
  equations and concentration-cancellation}, Comm. P.D.E., 20 (1995),
  pp.~1077--1104.

\bibitem{a_S96}
\leavevmode\vrule height 2pt depth -1.6pt width 23pt, {\em Point-vortex method
  for periodic weak solutions of the {2-D} {E}uler equations}, Comm. Pure Appl.
  Math., 49 (1996), pp.~911--965.

\bibitem{a_S97}
{\sc A.~Shnirelman}, {\em On the nonuniqueness of weak solution of the {E}uler
  equation}, Comm. Pure Appl. Math., 50 (1997), pp.~1261--1286.

\bibitem{a_SSBF81}
{\sc C.~Sulem, P.-L. Sulem, C.~Bardos, and U.~Frisch}, {\em Finite time
  analyticity for the two- and three-dimensional {K}elvin-{H}elmholtz
  instability}, Comm. Math. Phys., 80 (1981), pp.~485--516.

\bibitem{a_T04}
{\sc Y.~Taniuchi}, {\em Uniformly local {$L\sp p$} estimate for 2-{D} vorticity
  equation and its application to {E}uler equations with initial vorticity in
  {${\bf bmo}$}}, Comm. Math. Phys., 248 (2004), pp.~169--186.

\bibitem{a_V99}
{\sc M.~Vishik}, {\em Incompressible flows of an ideal fluid with vorticity in
  borderline spaces of {B}esov type}, Ann. Sci. \'Ecole Norm. Sup. (4), 32
  (1999), pp.~769--812.

\bibitem{a_VTC05}
{\sc M.~I. Vishik, E.~S. Titi, and V.~V. Chepyzhov}, {\em Trajectory attractor
  approximations of the {3D} {N}avier-{S}tokes system by a {L}eray-{$\alpha$}
  model}, Russian Mathematical Dokladi (Translated from Russian), 71 (2005),
  pp.~92--95.

\bibitem{b_W44}
{\sc G.~N. Watson}, {\em A Treatise on the Theory of {B}essel Functions},
  Cambridge Mathematical Library, Cambridge University Press, Cambridge, 1995.
\newblock Reprint of the second (1944) edition.

\bibitem{a_W02}
{\sc S.~Wu}, {\em Recent progress in mathematical analysis of vortex sheets},
  in Proceedings of the International Congress of Mathematicians, Vol. III
  (Beijing, 2002), Beijing, 2002, Higher Ed. Press, pp.~233--242.

\bibitem{a_W06}
\leavevmode\vrule height 2pt depth -1.6pt width 23pt, {\em Mathematical
  analysis of vortex sheets}, Comm. Pure Appl. Math., 59 (2006),
  pp.~1065--1206.

\bibitem{a_Y63}
{\sc V.~I. Yudovich}, {\em Non-stationary flow of an ideal incompressible
  liquid}, Zh. Vychisl. Mat. i Mat. Fiz., 3 (1963), pp.~1032--1066.

\end{thebibliography}

\end{document}